\documentclass[a4paper,12pt]{article}

\usepackage{epstopdf}
\usepackage {svg}
\usepackage{here}

\usepackage{amsmath}
\usepackage{amssymb}
\usepackage{amsthm}
\usepackage{mathrsfs}
\usepackage[all]{xy}
\usepackage{float}
\usepackage[all]{xy}

\usepackage{tikz-cd}
\usepackage{hyperref}

\newtheorem{thm}{Theorem}[section]
\newtheorem{prop}[thm]{Proposition}
\newtheorem{lem}[thm]{Lemma}
\newtheorem{dfn}[thm]{Definition}

\newtheorem{corollary}[thm]{Corollary}
\newtheorem{remark}[thm]{Remark}

\newcommand{\C}{{\mathbb C}}

\newcommand{\PP}{{\mathcal P}}
\newcommand{\K}{{\mathcal K}}

\newcommand{\V}{{\mathcal V}}

\newcommand{\X}{{\mathcal X}}
\newcommand{\W}{{\mathcal W}}

\newcommand{\Z}{{\mathbb Z}}

\numberwithin{equation}{section}

\usepackage{braket}

\newcommand\restr[2]{{
  \left.\kern-\nulldelimiterspace 
  #1 
  \vphantom{\big|} 
  \right|_{#2} 
  }}

\makeatletter 

\@addtoreset{figure}{section}
\@addtoreset{table}{section}

\makeatother 

\usepackage{comment}
\title{Fusion rules for the triplet $W$-algebra $\mathcal{W}_{p_+,p_-}$}
\author{Hiromu Nakano}
\date{}

\begin{document}
\maketitle
\begin{abstract}
We study the structure of fusion rules for the triplet $W$-algebra $\mathcal{W}_{p_+,p_-}$. 
By using the vertex tensor category theory developed by Huang, Lepowsky and Zhang, we rederive certain non-semisimple fusion rules given by Gaberdiel-Runkel-Wood and Rasmussen.
We further show that certain rank two and three indecomposable modules are self-dual. 
\end{abstract}

\section{Introduction}

Let $p_+$ and $p_-$ be coprime integers satisfying $p_->p_+\geq 2$, and
\begin{align*}
c_{p_+,p_-}=1-6\frac{(p_+-p_-)^2}{p_+p_-}
\end{align*}
the minimal central charge for the Virasoro minimal model $L(c_{p_+,p_-},0)$. 
In the paper \cite{FF2}, Feigin, Gainutdinov, Semikhatov and Tipunin define an important irrational vertex operator algebra $\W_{p_+,p_-}$, which is called the triplet $W$-algebra, as an extension of the Virasoro minimal model $L(c_{p_+,p_-},0)$. 
They study the structure of the simple $\W_{p_+,p_-}$-modules and the Grothendieck fusion ring, and show that the representation theory of $\W_{p_+,p_-}$ is closely related to representations of the quantum group $\mathfrak{g}_{p_+,p_-}$ at roots of unity \cite{Arike,FF2,FF22}. 

Representation theoretical results such as the structure theorem of the simple modules and the classification of the simple modules are given by Adamovi\'{c}-Milas \cite{AMW2p,AMW3p} and Tsuchiya-Wood \cite{TW}.
They also show that the triplet $W$-algebra $\W_{p_+,p_-}$ satisfies the $C_2$-cofinite condition.
An important consequence that follows from the $C_2$-cofinite condition is that the $\W_{p_+,p_-}$-module category $\mathcal{W}_{p_+,p_-}\mathchar`-{\rm Mod}$ has the structure of the vertex tensor category developed by Huang, Lepowsky and Zhang \cite{HLZ1,HLZ2,HLZ3,HLZ4,HLZ5,HLZ6,HLZ7,HLZ8}.

The more detailed tensor structure of $\mathcal{W}_{p_+,p_-}\mathchar`-{\rm Mod}$ is studied from the directions of boundary conformal field theories and character theories.
In \cite{R}, Rasmussen gives the fusion rules between certain rank two and rank three indecomposable modules by using the method of 
certain scaling limits of solvable lattice models. In \cite{GRW0,GRW,W}, Gaberdiel, Runkel and Wood give the detailed tensor structure by using the Nahm-Gaberdiel-Kausch fusion algorithm \cite{Kanade,GK} and show that the tensor category on $\mathcal{W}_{p_+,p_-}\mathchar`-{\rm Mod}$ is not rigid. This non-rigidity makes it much more difficult to determine the tensor structure of $\mathcal{W}_{p_+,p_-}\mathchar`-{\rm Mod}$ compared to the triplet algebras $\W_{p}$ \cite{FF4,McRae,TWFusion}. In \cite{RW}, Ridout and Wood determine the structure of the Grothendieck ring of $\W_{p_+,p_-}$ by using the Verlinde ring of the singlet $W$-algebra consisting of the characters of the singlet $W$-algebra.

The main goal of this paper is to rederive certain fusion rules given in \cite{GRW0,R,W}, by using the theory of Huang-Lepowsky-Zhang tensor category, and to show that certain rank two and rank three indecomposable modules are rigid and self-dual. 
In this paper we do not consider the projective covers of the Virasoro minimal simple modules. The socle structure of these projective modules are given by \cite{GRW,Nakano}. According to \cite{GRW}, these projective modules are known to have no duals, which complicates the tensor structure of $\mathcal{W}_{p_+,p_-}\mathchar`-{\rm Mod}$.  

This paper is organized as follows.

In Section \ref{Basic}, we briefly review free field theory and the structure theorem of Fock modules in accordance with \cite{FF,FF2,IK,TW}.

In Section \ref{secWabel}, we review the structure of the abelian category $\mathcal{W}_{p_+,p_-}\mathchar`-{\rm Mod}$. 
In Subsections \ref{Walgebra} and \ref{IrreducibleWalgebra}, we introduce the triplet $W$-algebra $\W_{p_+,p_-}$, all simple $\W_{p_+,p_-}$-modules in accordance with \cite{AMW2p,AMW3p,FF2,TW}. The simple modules are finite are broadly classified into $2p_+p_-$ $\mathcal{X}^\pm_{r,s}$ and $\frac{(p_+-1)(p_--1)}{2}$ Virasoro minimal simple modules $L(h_{r,s})$. In Subsection \ref{BlockWalgebra}, we introduce the block decomposition of the ablelian category $\mathcal{W}_{p_+,p_-}\mathchar`-{\rm Mod}$.
Each block is assigned to one of three groups: $\frac{(p_+-1)(p_--1)}{2}$ thick blocks $C^{thick}_{r,s}=C^{thick}_{p_+-r,p_--s}$, $p_++p_--2$ thin blocks $C^{thin}_{r,p_-},C^{thin}_{p_+,s}$ and two semisimple blocks $C^\pm_{p_+,p_-}$. The most complex groups are the thick blocks and each thick block $C^{thick}_{r,s}$ contains five simple modules $\X^+_{r,s},\X^+_{p_+-r,p_--s},\X^-_{p_+-r,s},\X^-_{r,p_--s}$ and $L(h_{r,s})$.
In Subsection \ref{LogarithmicWalgebra}, we introduce logarithmic $\W_{p_+,p_-}$-modules $\mathcal{Q}(\X^\pm_{r,s})_{\bullet,\bullet}$ and $\mathcal{P}^\pm_{r,s}$ whose $L_0$-nilpotent rank are two and three, in accordance with \cite{AK,GRW0,GRW,Nakano}.

In Section \ref{FusionRules}, we study the structure of the tensor category on $\mathcal{W}_{p_+,p_-}\mathchar`-{\rm Mod}$. We introduce indecomposable modules $\K_{r,s}(0<r<p_+,0<s<p_-)$ satisfying the exact sequence
\begin{align*}
0\rightarrow \mathcal{X}^+_{r,s}\rightarrow \mathcal{K}_{r,s}\rightarrow L(h_{r,s})\rightarrow 0,
\end{align*}
and following \cite{CMY,MS,McRae,TWFusion}, prove the self-duality of $\K_{1,2}$ and $\K_{2,1}$ in Theorems \ref{rigid12ori} and \ref{rigid21}. By using ${\rm Ext}^1$-groups in each block of $\mathcal{W}_{p_+,p_-}\mathchar`-{\rm Mod}$ and the self-duality for $\K_{1,2}$ and $\K_{2,1}$, we show that the indecomposable modules $\K_{r,s}$, $\mathcal{Q}(\X^\pm_{r,s})_{\bullet,\bullet}$ and $\mathcal{P}^\pm_{r,s}$ appear in the repeated tensor products of $\K_{1,2}$ and $\K_{2,1}$. As a result we see that all indecomposable modules of $\K_{r,s}$, $\mathcal{Q}(\X^\pm_{r,s})_{\bullet,\bullet}$ and $\mathcal{P}^\pm_{r,s}$ are self-dual objects. In Subsection \ref{simple0817}, we also determine the tensor products between all simple modules.

In Section \ref{FusionRing}, following \cite{TWFusion}, we introduce a certain non-semisimple fusion ring $P(\mathbb{I}_{p_+,p_-})$ generated by the indecomposable modules $\mathcal{X}^-_{1,1}$, $\mathcal{K}_{1,2}$ and $\mathcal{K}_{2,1}$. Based on the results in Section \ref{FusionRules}, we determine the ring structure of $P(\mathbb{I}_{p_+,p_-})$. Similar to the case of the triplet algebra $\mathcal{W}_{p}$ \cite{TWFusion}, by using the structure of $P(\mathbb{I}_{p_+,p_-})$, we can compute the tensor products between the indecomposable modules $\mathcal{X}^\pm_{r,s}$, $\K_{r,s}$, $\K^*_{r,s}$, $\mathcal{Q}(\X^\pm_{r,s})_{\bullet,\bullet}$ and $\mathcal{P}^\pm_{r,s}$. We only state the explicit formulas of the tensor products between the simple modules $\mathcal{X}^\pm_{r,s}$.
In Subsection \ref{secK}, we review the structure of the Grothendieck fusion ring determined by \cite{RW}.

\section{Bosonic Fock modules}
\label{Basic}
Let $\mathcal{L}$ be the Virasoro algebra generated by ${L}_n(n\in \mathbb{Z})$ and $C$ (the central charge) with the relation
\begin{align*}
&[{L}_m,{L}_n]=(m-n){L}_{m+n}+\frac{m^3-m}{12}C\delta_{m+n,0},
&[{L}_n,C]=0,
\end{align*}
and let $U(\mathcal{L})$ be the universal enveloping algebra of $\mathcal{L}$.
Fix two coprime integers $p_+,p_-$ such that $p_->p_+\geq 2$, and let
\begin{align*}
c_{p_+,p_-}=1-6\frac{(p_+-p_-)^2}{p_+p_-}
\end{align*}
be the central charge of the Virasoro minimal model $L(c_{p_+,p_-},0)$.
In this paper, we consider only Virasoro modules whose central charge are $C=c_{p_+,p_-}\cdot{\rm id}$. Henceforth we fix $C$ as $c_{p_+,p_-}$.
In this section, we briefly review theory of Fock modules following \cite{TW}.
As for the representation theory of the Virasoro algebra, see \cite{FF,IK}. 
\subsection{Free field theory}
The Heisenberg Lie algebra 
\begin{align*}
\mathcal{H}=\bigoplus_{n\in\mathbb{Z}}\mathbb{C} a_{n}\oplus \mathbb{C} K_{\mathcal{H}}
\end{align*}
is the Lie algebra whose commutator relations are given by
\begin{align*} 
&[a_m,a_n]=m\delta_{m+n,0}K_{\mathcal{H}},
&[K_{\mathcal{H}},\mathcal{H}]=0.
\end{align*}
Let
\begin{align*}
\mathcal{H}^\pm=\bigoplus_{n>0}\mathbb{C} a_{\pm n},\ \ \ \ \ 
\mathcal{H}^0=\mathbb{C} a_0\oplus \mathbb{C} K_{\mathcal{H}},\ \ \ \ \ 
\mathcal{H}^{\geq}=\mathcal{H}^+\oplus \mathcal{H}^0.
\end{align*}
For any $\alpha\in\mathbb{C}$, let $\mathbb{C}{\mid}\alpha\rangle$ be the one dimensional $\mathcal{H}^\geq$-module defined by
\begin{align*}
&a_n{\mid}\alpha\rangle=\delta_{n,0}\alpha{\mid}\alpha\rangle\ (n\geq 0),
&K_{\mathcal{H}}{\mid}\alpha\rangle={\mid}\alpha\rangle. 
\end{align*}
For any $\alpha\in \mathbb{C}$, the bosonic Fock module is defined by
\begin{align*}
F_{\alpha}={\rm Ind}_{\mathcal{H}^\geq}^{\mathcal{H}}\mathbb{C}{\mid}\alpha\rangle.
\end{align*} 
Let 
\begin{equation*}
a(z)=\sum_{n\in\mathbb{Z}}a_nz^{-n-1}
\end{equation*}
be the bosonic current. Then we have the following operator expansion
\begin{equation*}
 a(z)  a(w)=\frac{1}{(z-w)^2}+\cdots,
\end{equation*}
where $\cdots$ denotes the regular part in $z=w$.
We define the energy-momentum tensor
\begin{equation*}
T(z):=\frac{1}{2}:a(z)a(z):+\frac{\alpha_0}{2}\partial a(z),\ \ \ \ \ \alpha_0:=\sqrt{\frac{2p_-}{p_+}}-\sqrt{\frac{2p_+}{p_-}}.
\end{equation*}
where $:\ :$ is the normal ordered product. The energy-momentum tensor satisfies the following operator expansion
\begin{align*}
T(z)T(w)= \frac{c_{p_+,p_-}}{2(z-w)^4}+\frac{2T(w)}{(z-w)^2}+\frac{\partial T(w)}{z-w}+\cdots.
\end{align*}
The Fourier modes of 
$
T(z)=\sum_{n\in\mathbb{Z}}L_{n}z^{-n-2}
$
generate the Virasoro algebra whose central charge is $c_{p_+,p_-}$.
Thus, by the energy-momentum tensor $T(z)$, each Fock module $F_{\alpha}$ has the structure of a Virasoro module whose central charge is $c_{p_+,p_-}$.

For any $\alpha\in \mathbb{C}$, the Fock module $F_{\alpha}$ has the following $L_0$ weight decomposition
\begin{align*}
&F_{\alpha}=\bigoplus_{n\in \mathbb{Z}_{\geq 0}}F_{\alpha}[n],
&F_{\alpha}[n]:=\{v\in F_{\alpha}\setminus\{0\}\mid L_0v=(h_{\alpha}+n)v\},
\end{align*}
where 
\begin{align}
\label{alphu}
h_{\alpha}:=\frac{1}{2}\alpha(\alpha-\alpha_0).
\end{align}
We define the following conformal vector in $F[2]$
\begin{align*}
T=\frac{1}{2}(a^2_{-1}+\alpha_0 a_{-2})\ket{0}.
\end{align*}
\begin{dfn}
The Fock module ${F}_0$ carries the structure of a $\Z_{\geq 0}$-graded vertex operator algebra, with
\begin{align*}
&Y(\ket{0},z)={\rm id},\ \ \ \ \ \ Y(a_{-1}\ket{0},z)=a(z),\ \ \ \ \ \ Y(T,z)=T(z).
\end{align*}
We denote this vertex operator algebra by $\mathcal{F}^{}_{\alpha_0}$.
\end{dfn}

\subsection{The structure of Fock modules}
We set
\begin{align*}
&\alpha_+=\sqrt{\frac{2p_-}{p_+}},
&\alpha_-=-\sqrt{\frac{2p_+}{p_-}}
\end{align*}
For $r,s,n\in\Z$ we introduce the following symbols
\begin{align}
&\alpha_{r,s;n}=\frac{1-r}{2}\alpha_++\frac{1-s}{2}\alpha_-+\frac{\sqrt{2p_+p_-}}{2}n,
&\alpha_{r,s}=\alpha_{r,s;0}.
\label{alpha_{r,s;n}}
\end{align}
and
\begin{align}
&h_{r,s;n}:=h_{\alpha_{r,s;n}},
&h_{r,s}:=h_{\alpha_{r,s}},
\label{h_alphars}
\end{align}
where the symbol $h_{\alpha}$ is defined by (\ref{alphu}).
Note that 
\begin{align*}
&\alpha_{r,s;n}=\alpha_{r-np_+,s}=\alpha_{r,s+np_-},\\
&h_{r,s;n}=h_{r-np_+,s}=h_{r,s+np_-},
&h_{r,s;n}&=h_{-r,-s;-n}.
\end{align*}
For $r,s,n\in\Z$, we set
\begin{align*}
&F_{r,s;n}=F_{\alpha_{r,s;n}},
&F_{r,s}=F_{\alpha_{r,s}},
\end{align*}
and let $L(h_{r,s;n})$ be the simple Virasoro module whose highest weight is $h_{r,s;n}$ and whose central charge $C=c_{p_+,p_-}\cdot{\rm id}$. 

Before describing the structure of Fock modules, let us introduce the notion of socle series. 
\begin{dfn}
Let $V$ be a vertex operator algebra or the Virasoro algebra.
We say that $M\in V\mathchar`-{\rm Mod}$ has a {\rm socle series} when there exists a finite length sequence of submodules of $M$ such that
\begin{align*}
{\rm Soc}_1(M)\subsetneq {\rm Soc}_2(M)\subsetneq \cdots \subsetneq {\rm Soc}_n(M)=M
\end{align*}
with ${\rm Soc}_1(M)={\rm Soc}(M)$ and ${\rm Soc}_{i+1}(M)/{\rm Soc}_{i}(M)={\rm Soc}(M/{\rm Soc}_{i}(M))$, where ${\rm Soc}(M)$ is the socle of $M$, that is ${\rm Soc}(M)$ is the maximal semisimple submodule of $M$. 
\end{dfn}

The following proposition is due to \cite{FF}.

\begin{prop}
\label{FockSocle}
As the $U(\mathcal{L})$-module, every $F_{r,s;n}\in \mathcal{F}_{\alpha_0}\mathchar`-{\rm Mod}$ has a socle series.
There are four cases of socle series for the Fock modules $F_{r,s;n}\in \mathcal{F}_{\alpha_0}\mathchar`-{\rm Mod}$:
\begin{enumerate}
\item For each $1\leq r\leq p_+-1,\ 1\leq s\leq p_--1,\ n\in\Z$, we have
\begin{align*}
{\rm Soc}_1(F_{r,s;n})\subsetneq {\rm Soc}_2(F_{r,s;n})\subsetneq {\rm Soc}_3(F_{r,s;n})=F_{r,s;n}
\end{align*} 
such that
\begin{align*}
&{\rm Soc}_1=\bigoplus_{k\geq 0}L(h_{r,p_--s;|n|+2k+1}),\\
&{\rm Soc}_2/{\rm Soc}_1=\bigoplus_{k\geq a}L(h_{r,s;|n|+2k})\oplus \bigoplus_{k\geq 1-a}L(h_{p_+-r,p_--s;|n|+2k}),\\
&{\rm Soc}_3/{\rm Soc}_2=\bigoplus_{k\geq 0}L(h_{p_+-r,s;|n|+2k+1}),
\end{align*}
where $a=0$ if $n\geq 0$ and $a=1$ if $n<0$.

\item For each $1\leq s\leq p_--1,\ n\in\Z$, we have
\begin{align*}
{\rm Soc}_1(F_{p_+,s;n})\subsetneq {\rm Soc}_2(F_{p_+,s;n})=F_{p_+,s;n}
\end{align*}
such that
\begin{align*}
&{\rm Soc}_1(F_{p_+,s;n})=\bigoplus_{k\geq 0}L(h_{p_+,p_--s;|n|+2k+1}),\\
&{\rm Soc}_2(F_{p_+,s;n})/{\rm Soc}_1(F_{p_+,s;n})=\bigoplus_{k\geq a}L(h_{p_+,s;|n|+2k})
\end{align*}
where $a=0$ if $n\geq 1$ and $a=1$ if $n<1$.

\item For each $1\leq r \leq p_+-1,\ n\in\Z$, we have
\begin{align*}
{\rm Soc}_1(F_{r,p_-;n})\subsetneq {\rm Soc}_2(F_{r,p_-;n})=F_{r,p_-;n}
\end{align*}
such that
\begin{align*}
&{\rm Soc}_1(F_{r,p_-;n})=\bigoplus_{k\geq 0}L(h_{r,p_-;|n|+2k}),\\
&{\rm Soc}_2(F_{r,p_-;n})/{\rm Soc}_1(F_{r,p_-;n})=\bigoplus_{k\geq a}L(h_{p_+-r,p_-;|n|+2k-1})
\end{align*}
where $a=1$ if $n\geq 0$ and $a=0$ if $n<0$.

\item For each $n \in\Z$, the Fock module $F_{p_+,p_-;n}$  is semi-simple as a Virasoro module:
\begin{align*}
F_{p_+,p_-;n}={\rm Soc}(F_{p_+,p_-;n})=\bigoplus_{k\geq 0}L(h_{p_+,p_-;|n|+2k}).
\end{align*}
\end{enumerate}
\end{prop}

\subsection{Screening operators}

We define the following free scalar field $\phi(z)$
\begin{equation*}
\phi(z)=\hat{a}+a_0{\rm log}z-\sum_{n\neq 0}\frac{a_n}{n}z^{-n}
\end{equation*}
where $\hat{a}$ is defined by
\begin{equation*}
[a_m,\hat{a}]=\delta_{m,0}{\rm id}.
\end{equation*}
The scalar field $\phi(z)$ satisfies the operator product expansion
\begin{equation*}
\phi(z)\phi(w)={\rm log}(z-w)+\cdots.
\end{equation*} 
For any $\alpha\in\C$ we introduce the field $V_{\alpha}(z)$
\begin{align*}
&V_{\alpha}(z)=:e^{\alpha\phi(z)}:=e^{\alpha\hat{a}}z^{\alpha a_{0}}\overline{V}_{\alpha}(z),\ z^{\alpha a_{0}}=e^{\alpha a_{0}{\rm log}z}\ ,\\
&\overline{V}_{\alpha}(z)=e^{\alpha\sum_{n\geq 1}\frac{a_{-n}}{n}z^{n}}e^{-\alpha\sum_{n \geq 1}\frac{a_n}{n}z^{-n}}.
\end{align*}
The fields $V_{\alpha}(z)$ satisfy the following operator product expansion
\begin{align*}
V_{\alpha}(z)V_{\beta}(w)= (z-w)^{\alpha\beta}:{V}_{\alpha}(z){V}_{\beta}(w):.
\end{align*}
We introduce the following two screening currents ${Q}_{+}(z),{Q}_-(z)$
\begin{align*}
{Q}_{\pm}(z)=V_{\alpha_\pm}(z)
\end{align*} 
whose conformal weights are $h_{\alpha_{\pm}}=1$ : 
\begin{align*}
T(z){Q}_{\pm}(w)&=\frac{{Q}_{\pm}(w)}{(z-w)^2}+\frac{\partial_{w}{Q}_{\pm}(w)}{z-w}+\cdots\\
&=\partial_{w}\bigl(\frac{{Q}_{\pm}(w)}{z-w}\bigr)+\cdots.
\end{align*}
Therefore the zero modes of the fields ${Q}_{\pm}(z)$
\begin{align*}
&{\rm Res}_{z=0}{Q}_{+}(z){\rm d}z={Q}_{+}\ :\ F_{1,k}\rightarrow F_{-1,k},\ \ k\in\Z\\
&{\rm Res}_{z=0}{Q}_{-}(z){\rm d}z={Q}_{-}\ :\ F_{k,1}\rightarrow F_{k,-1},\ \ k\in\Z
\end{align*}
commute with every Virasoro mode. These zero modes are called screening operators. 

For $r,s\geq 1$, we introduce more complicated screening currents 
\begin{align*}
&{Q}^{[r]}_{+}(z)\in {\rm Hom}_{\C}(F_{r,k},F_{-r,k})[[z,z^{-1}]],\ \ r\geq 1,k\in\Z,\\
&{Q}^{[s]}_{-}(z)\in {\rm Hom}_{\C}(F_{k,s},F_{k,-s})[[z,z^{-1}]],\ \ s\geq 1,k\in\Z,
\end{align*}
constructed by \cite{TK} as follows
\begin{equation}
\label{Tsuchiya-Kanie}
\begin{split}
&{Q}^{[r]}_{+}(z)=\int_{\overline{\Gamma}_{r}(\kappa_+)}Q_+(z)Q_+(zx_{1})Q_+(zx_2)\cdots Q_+(zx_{r-1})z^{r-1}{\rm d}x_1\cdots{\rm d}x_{r-1},\\
&{Q}^{[s]}_{-}(z)=\int_{\overline{\Gamma}_{s}(\kappa_-)}Q_-(z)Q_-(zx_{1})Q_-(zx_2)\cdots Q_-(zx_{s-1})z^{s-1}{\rm d}x_1\cdots{\rm d}x_{s-1},
\end{split}
\end{equation}
where $\overline{\Gamma}_n(\kappa_\pm)$ is a certain regularized cycle constructed from the simplex
\begin{align*}
\Delta_{n-1}=\{\ (x_1,\dots,x_{n-1})\in \mathbb{R}^{n-1}\ |\ 1>x_1>\cdots>x_{n-1}>0\ \}.
\end{align*}
These fields satisfy the following operator product expansion \cite{FF,TK,TW}
\begin{align*}
&T(z){Q}^{[r]}_{+}(w)=\frac{{Q}^{[r]}_{+}(w)}{(z-w)^2}+\frac{\partial_{w}{Q}^{[r]}_{+}(w)}{z-w}+\cdots,\\
&T(z){Q}^{[s]}_{-}(w)=\frac{{Q}^{[s]}_{-}(w)}{(z-w)^2}+\frac{\partial_{w}{Q}^{[s]}_{-}(w)}{z-w}+\cdots.
\end{align*}
In particular the zero modes
\begin{align*}
&{\rm Res}_{z=0}{Q}^{[r]}_{+}(z){\rm d}z=Q^{[r]}_+ \in {\rm Hom}_{\C}(F_{r,k},F_{-r,k}),\ \ r\geq 2,k\in\Z,\\
&{\rm Res}_{z=0}{Q}^{[s]}_{-}(z){\rm d}z=Q^{[s]}_- \in {\rm Hom}_{\C}(F_{k,s},F_{k,-s}),\ \ s\geq 2,k\in\Z
\end{align*}
commute with every Virasoro mode of $\mathcal{F}_{\alpha_0}\mathchar`-{\rm Mod}$. These zero modes are also called screening operators.

See \cite{FF,Felder,IK,TW} for the detailed structure of the images and kernels of the screening operators $Q^{[\bullet]}_\pm$.

\section{The triplet $W$-algebra $\W_{p_+,p_-}$ and the abelian category $\mathcal{C}_{p_+,p_-}$}
\label{secWabel}
In this section we introduce some basic properties of the abelian category $\mathcal{C}_{p_+,p_-}$ of $\W_{p_+,p_-}$-modules.
In Subsections \ref{Walgebra} and \ref{IrreducibleWalgebra}, we introduce the lattice vertex operator algebra $\mathcal{V}_{[p_+,p_-]}$ and the triplet $W$-algebra $\W_{p_+,p_-}$, and review the structure of simple $\W_{p_+,p_-}$-modules in accordance with \cite{AMW2p,AMW3p,TW}. In Subsections \ref{BlockWalgebra} and \ref{LogarithmicWalgebra}, we introduce the block decomposition of the abelian category of $\W_{p_+,p_-}$-modules and logarithmic $\W_{p_+,p_-}$-modules $\mathcal{Q}(\X^\pm_{r,s})_{\bullet,\bullet}$, $\mathcal{P}^\pm_{r,s}$ with $L_0$-nilpotent rank two or three, in accordance with \cite{AK, GRW0,GRW,Nakano,R}. From this subsection, we identify any $\W_{p_+,p_-}$-modules that are isomorphic among each other.

\subsection{The lattice vertex operator algebra and the vertex operator algebra $\W_{p_+,p_-}$}
\label{Walgebra}
In this subsection, we introduce the triplet $W$-algebra $\W_{p_+,p_-}$ given by \cite{FF2}. This vertex operator algebra is a generalization of the triplet $W$-algebra $\W_p$ \cite{AM,FF3,FF4,Ka}.

Let us define a lattice vertex operator algebra associated to the integer lattice $\mathbb{Z}\sqrt{2p_+p_-}$. 
\begin{dfn}
The {\rm lattice vertex operator algebra} $\mathcal{V}_{[p_+,p_-]}$ is the tuple 
\begin{align*}
(\V^{+}_{1,1}, \ket{0}, \frac{1}{2}(a^2_{-1}-\alpha_0a_{-2})\ket{0}, Y),
\end{align*}
where the underlying vector space of $\mathcal{V}_{[p_+,p_-]}$ is given by
\begin{align*}
\V^{+}_{1,1}=\bigoplus_{n\in\Z}F_{1,1;2n}=\bigoplus_{n\in\Z}F_{n\sqrt{2p_+p_-}},
\end{align*}
and the vertex operator $Y$ satisfies $Y(\ket{\alpha_{1,1;2n}},z)=V_{\alpha_{1,1;2n}}(z)$ for $n\in\Z$.
\end{dfn}
It is a known fact that simple $\mathcal{V}_{[p_+,p_-]}$-modules are given by the following $2p_+p_-$ direct sum of Fock modules
\begin{align*}
&\V^{+}_{r,s}=\bigoplus_{n\in\Z}F_{r,s;2n},
&\V^{-}_{r,s}=\bigoplus_{n\in\Z}F_{r,s;2n+1},
\end{align*}
where $1\leq r\leq p_+,1\leq s\leq p_-$.

Note that the two screening operators $Q_+$ and $Q_-$ act on $\V^+_{1,1}$. We define the following vector subspace of $\V^+_{1,1}$:
\begin{align*}
\K_{1,1}={\rm ker}Q_+\cap{\rm ker}Q_-\subset \V^+_{1,1}.
\end{align*}
\begin{dfn}
The {\rm triplet $W$-algebra} is defined by the following vertex operator algebra
\begin{align*}
\W_{p_+,p_-}:=(\K_{1,1},\ket{0},T,Y),
\end{align*}
where the vacuum vector, conformal vector and vertex operator map are those of $\mathcal{V}_{[p_+,p_-]}$.
\end{dfn}
We define
\begin{align*}
&W^+:=Q^{[p_--1]}_-\ket{\alpha_{1,p_--1;3}},
&W^-&:=Q^{[p_+-1]}_+\ket{\alpha_{p_+-1,1;-3}},\\ 
&W^0:=Q^{[2p_+-1]}_+\ket{\alpha_{p_+-1,1;-3}}.
\end{align*}
From the results in \cite{FF,Felder,TW}, we see that these vectors are non-trivial Virasoro singular vectors whose $L_0$ weights are $h_{4p_+-1,1}$.
\begin{prop}[\cite{AMW2p,AMW3p,TW}]
\label{genW}
$\W_{p_+,p_-}$ is strongly  generated by the fields $T(z),Y(W^{\pm},z),Y(W^0,z)$.
\end{prop}
By the following theorem, any $\W_{p_+,p_-}$-module has finite length and the abelian category of $\mathcal{W}_{p_+,p_-}$-modules has the structure of braided tensor category (cf.\cite{HLZ1,HLZ2,HLZ3,HLZ4,HLZ5,HLZ6,HLZ7,HLZ8}).
\begin{thm}[\cite{AMW2p,AMW3p,TW}]
$\W_{p_+,p_-}$ is $C_2$-cofinite.
\end{thm}
\subsection{Simple $\W_{p_+,p_-}$-modules}
\label{IrreducibleWalgebra}
In this subsection, we introduce the structure of the simple $\W_{p_+,p_-}$-modules which was first given by \cite{FF2} and proved mathmatically by \cite{AMW2p,AMW3p,TW}.
For $1\leq r\leq p_+-1$, $1\leq s\leq p_--1$, we introduce the following symbols
\begin{align*}
&r^{\vee}(p_+):=p_+-r,
&s^{\vee}(p_-):=p_--s
\end{align*}
and use the abbreviations 
\begin{align}
\label{notrs}
&r^\vee=r^{\vee}(p_+),
&s^\vee=s^{\vee}(p_-).
\end{align}
For $1\leq r\leq p_+,\ 1\leq s\leq p_-$, let $\X^{\pm}_{r,s}$ be the following vector subspace of $\V^\pm_{r,s}$: 
\begin{enumerate}
\item For $1\leq r\leq p_+-1,\ 1\leq s\leq p_--1$,
\begin{align*}
&\X^+_{r,s}=Q^{[r^\vee]}_+(\V^-_{r^\vee,s})\cap Q^{[s^\vee]}_-(\V^-_{r,s^\vee}),
&\X^-_{r,s}=Q^{[r^\vee]}_+(\V^+_{r^\vee,s})\cap Q^{[s^\vee]}_-(\V^+_{r,s^\vee}).
\end{align*}
\item For $1\leq r\leq p_+-1,\ s=p_-$, 
\begin{align*}
&\X^+_{r,p_-}=Q^{[r^\vee]}_+(\V^-_{r^\vee,p_-}),
&\X^-_{r,p_-}=Q^{[r^\vee]}_+(\V^+_{r^\vee,p_-}).
\end{align*}
\item For $r=p_+,\ 1\leq s\leq p_--1$,
\begin{align*}
&\X^+_{p_+,s}=Q^{[s^\vee]}_-(\V^-_{p_+,s^\vee}),
&\X^-_{p_+,s}=Q^{[s^\vee]}_-(\V^+_{p_+,s^\vee}).
\end{align*}
\item For $r=p_+,\ s=p_-$, 
\begin{align*}
&\X^+_{p_+,p_-}=\V^+_{p_+,p_-},
&\X^-_{p_+,p_-}=\V^-_{p_+,p_-}.
\end{align*}
\end{enumerate}

We define the interior Kac table $\mathcal{T}$ as the following quotient set
\begin{align*}
\mathcal{T}=\{(r,s)|\ 1\leq r<p_+,1\leq s<p_-\}/\sim
\end{align*}
where $(r,s)\sim (r',s')$ if and only if $r'=r^\vee, s'=s^\vee$. 

\begin{thm}[\cite{AMW2p,AMW3p,TW}]
\label{simpleclass}
Each of the $\frac{(p_+-1)(p_--1)}{2}+2p_+p_-$ vector spaces
\begin{align*}
L(h_{k,l}),\ (k,l)\in\mathcal{T},\ \ \ \ \X^{\pm}_{r,s},\ 1\leq r\leq p_+,\ 1\leq s\leq p_-
\end{align*} 
becomes a simple $\W_{p_+,p_-}$-module. Furthermore, all simple modules are given by these $\frac{(p_+-1)(p_--1)}{2}+2p_+p_-$ simple modules. 
\end{thm}

\begin{prop}[\cite{AMW2p,AMW3p,TW}]
\label{socleV}
Each $2p_+p_-$ simple $\mathcal{V}_{[p_+,p_-]}$-module becomes a $\W_{p_+,p_-}$-module by restriction, and has the following socle series:
\begin{enumerate}
\item For $1\leq r<p_+,\ 1\leq s<p_-$, $\V^+_{r,s}$ has the following socle series
\begin{align*}
{\rm Soc}_1(\V^+_{r,s})\subsetneq {\rm Soc}_2(\V^{+}_{r,s})\subsetneq {\rm Soc}_3(\V^{+}_{r,s})=\V^{+}_{r,s}
\end{align*}
with
\begin{align*}
{\rm Soc}_1=\X^+_{r,s},
&&{\rm Soc}_2/{\rm Soc}_1=\X^-_{r^{\vee},s}\oplus \X^{-}_{r,s^{\vee}}\oplus L(h_{r,s}),
&&&{\rm Soc}_3/{\rm Soc}_2=\X^+_{r^{\vee},s^{\vee}}.
\end{align*}
\item
For $1\leq r<p_+,\ 1\leq s<p_-$, $\V^-_{r,s}$ has the following socle series
\begin{align*}
{\rm Soc}_1(\V^-_{r,s})\subsetneq {\rm Soc}_2(\V^{-}_{r,s})\subsetneq {\rm Soc}_3(\V^{-}_{r,s})=\V^{-}_{r,s}
\end{align*}
with
\begin{align*}
{\rm Soc}_1=\X^-_{r,s},
&&{\rm Soc}_2/{\rm Soc}_1=\X^+_{r^{\vee},s}\oplus \X^{+}_{r,s^{\vee}},
&&&{\rm Soc}_3/{\rm Soc}_2=\X^-_{r^{\vee},s^{\vee}}.
\end{align*}
\item 
For $1\leq r<p_+$, $\V^+_{r,p_-}$ and $\V^-_{r^\vee,p_-}$ have the following socle series
\begin{align*}
\V^+_{r,p_-}/\X^+_{r,p_-}\simeq \X^-_{r^\vee,p_-},\ \ \ \ \V^-_{r^\vee,p_-}/\X^-_{r^\vee,p_-}\simeq \X^+_{r,p_-}.
\end{align*}
\item
For $1\leq s<p_-$, $\V^+_{p_+,s}$ and $\V^-_{p_-,s^\vee}$ have the following socle series
\begin{align*}
\V^+_{p_+,s}/\X^+_{p_+,s}\simeq \X^-_{p_+,s^\vee},\ \ \ \ \V^-_{p_+,s^\vee}/\X^-_{p_+,s^\vee}\simeq \X^+_{p_+,s}.
\end{align*}
\item For $r=p_+$, $s=p_-$,
\begin{align*}
\V^+_{p_+,p_-}=\X^+_{p_+,p_-},\ \ \ \ \V^-_{p_+,p_-}=\X^-_{p_+,p_-}.
\end{align*}

\end{enumerate}
\end{prop}
\begin{dfn}
\label{dfn0830}
For $1\leq r\leq p_+$, $1\leq s\leq p_-$ and $n\geq 0$, we define the following vectors $\{w^{(n)}_{i}(r,s)\}_{i=-n}^{n}$ in $\mathcal{X}^+_{r,s}$:
\begin{equation*}
w^{(n)}_{i}(r,s):=
\begin{cases}
Q^{[(n+i)p_+]}_+\ket{\alpha_{p_+,p_-;-2n}}& r=p_+,\ s=p_-\\
Q^{[(n+i)p_+]}_+\ket{\alpha_{p_+,s;-2n}} & r=p_+,\ s\neq p_-\\
Q^{[(n+i+1)p_+-r]}_+\ket{\alpha_{r^\vee,p_-;-2n-1}} & r\neq p_+,\ s=p_-\\
Q^{[(n+i+1)p_+-r]}_+\ket{\alpha_{r^\vee,s;-2n-1}} & r\neq p_+,\ s\neq p_-
\end{cases}
\end{equation*}
where $Q^{[0]}_+={\rm id}$.
For $1\leq r\leq p_+$, $1\leq s\leq p_-$ and $n\geq 0$, we define the following vectors $\{v^{(n)}_{\frac{2i-1}{2}}(r,s)\}_{i=-n}^{n+1}$ in $\mathcal{X}^-_{r,s}$:
\begin{equation*}
v^{(n)}_{\frac{2i-1}{2}}(r,s):=
\begin{cases}
Q^{[(n+i)p_+]}_+\ket{\alpha_{p_+,p_-;-2n-1}}& r=p_+,\ s=p_-\\
Q^{[(n+i)p_+]}_+\ket{\alpha_{p_+,s;-2n-1}} & r=p_+,\ s\neq p_-\\
Q^{[(n+i+1)p_+-r]}_+\ket{\alpha_{r^\vee,p_-;-2n-2}} & r\neq p_+,\ s=p_-\\
Q^{[(n+i+1)p_+-r]}_+\ket{\alpha_{r^\vee,s;-2n-2}} & r\neq p_+,\ s\neq p_-
\end{cases}
.
\end{equation*}
\end{dfn}
In the following, we use the abbreviations $w^{(n)}_{i}=w^{(n)}_{i}(r,s)$ and $v^{(n)}_{\frac{2i-1}{2}}=v^{(n)}_{\frac{2i-1}{2}}(r,s)$.
Following \cite{TW}, for $1\leq r\leq p_+$, $1\leq s\leq p_-$ and $n\geq 0$, we define the following symbols
\begin{equation*}
\Delta^+_{r,s;n}=
\begin{cases}
h_{r^\vee,s;-2n-1}&r\neq p_+,s\neq p_-\\
h_{p_+,s;-2n}&r=p_+,s\neq p_-\\
h_{r,p_-;2n}&r\neq p_+,s=p_-\\
h_{p_+,p_-;-2n}&r=p_+,s=p_-
\end{cases}
,\ \ 
\Delta^-_{r,s;n}=
\begin{cases}
h_{r^\vee,s;-2n-2}&r\neq p_+,s\neq p_-\\
h_{p_+,s;-2n-1}&r=p_+,s\neq p_-\\
h_{r,p_-;2n+1}&r\neq p_+,s=p_-\\
h_{p_+,p_-;-2n-1}&r=p_+,s=p_-
\end{cases}
,
\end{equation*}
where the symbol $h_{r,s;n}$ is defined by (\ref{h_alphars}).
From the structure of the images and kernels of the screening operators $Q^{[\bullet]}_\pm$ \cite{FF,Felder,TW}, we can see that each $\X^\pm_{r,s}$ decomposes as the direct sum of  the simple Virasoro modules
\begin{align}
&\X^+_{r,s}=\bigoplus_{n\geq 0}(2n+1)L(\Delta^+_{r,s;n}),
&\X^-_{r,s}=\bigoplus_{n\geq 0}(2n+2)L(\Delta^-_{r,s;n}),
\label{decomp}
\end{align}
and the sets of vectors $\{w^{(n)}_{i}\}_{i=-n}^{n}$ and $\{v^{(n)}_{\frac{2i-1}{2}}\}_{i=-n}^{n+1}$ give a basis of the highest weight space of $(2n+1)L(\Delta^+_{r,s;n})\subset \X^+_{r,s}$ and $(2n+2)L(\Delta^-_{r,s;n})\subset \X^-_{r,s}$, respectively. 
For $W=W^\pm,W^0$ and $n\in \mathbb{Z}$, we define the $n$-th mode of the field $Y(W,z)$ as follows
\begin{align*}
W[n]=\oint_{z=0}Y(W,z)z^{h_{4p_+-1,1}+n-1}{\rm d}z.
\end{align*}
The transitive $\W_{p_+,p_-}$-actions on the Virasoro highest weight spaces of (\ref{decomp}) are given by the following proposition.
\begin{prop}[\cite{AMW2p,AMW3p,TW}]
\mbox{}
\begin{enumerate}
\item The Virasoro highest weight vectors $w^{(n)}_{i}$ $(n\geq 0,-n\leq i\leq n)$ satisfy
\begin{align*}
W^\epsilon[-h]w^{(0)}_{0}=0,\ \ \ h<\Delta^+_{r,s;1}-\Delta^+_{r,s;0},\ \epsilon=\pm,
\end{align*}
and
\begin{align*}
W^\pm[0]w^{(n)}_{i}&\in \C^\times w^{(n)}_{i\pm1}+\sum_{k=0}^{n-1}U(\mathcal{L}).w^{(k)}_{i\pm 1},\\
w^{(n+1)}_{i\pm 1}&\in \C^\times W^\pm[\Delta^+_{r,s;n}-\Delta^+_{r,s;n+1}]w^{(n)}_{i} +\sum_{k=0}^{n}U(\mathcal{L}).w^{(k)}_{i\pm 1},\\
w^{(n+1)}_{i}&\in \C^\times W^0[\Delta^+_{r,s;n}-\Delta^+_{r,s;n+1}]w^{(n)}_{i}+\sum_{k=0}^{n}U(\mathcal{L}).w^{(k)}_{i},
\end{align*}
where $w^{(n)}_{n+1}=w^{(n)}_{-n-1}=w^{(-1)}_{i}=0$.
\item  The Virasoro highest weight vectors $v^{(n)}_{\frac{2i+1}{2}}$ $(n\geq 0,-n-1\leq i\leq n)$ satisfy
\begin{equation*}
W^{\epsilon}[-h]v^{(0)}_{\frac{\pm1}{2}}
\begin{cases}
=0 &\epsilon=\pm \\
\in  U(\mathcal{L}).v^{(0)}_{\frac{\pm1}{2}} &\epsilon=0\\
\in U(\mathcal{L}).v^{(0)}_{\frac{\mp1}{2}}& \epsilon=\mp
\end{cases}
\ \ \ h<\Delta^-_{r,s;1}-\Delta^-_{r,s;0},
\end{equation*}
and
\begin{align*}
W^\pm[0]v^{(n)}_{\frac{2i+1}{2}}&\in \C^\times v^{(n)}_{\frac{2i+1\pm 2}{2}}+\sum_{k=0}^{n-1}U(\mathcal{L}).v^{(k)}_{\frac{2i+1}{2}\pm 1},\\
v^{(n+1)}_{\frac{2i+1}{2}\pm 1}&\in \C^\times W^\pm[\Delta^-_{r,s;n}-\Delta^-_{r,s;n+1}]v^{(n)}_{\frac{2i+1}{2}}+\sum_{k=0}^{n}U(\mathcal{L}).v^{(k)}_{\frac{2i+1}{2}\pm 1},\\
v^{(n+1)}_{\frac{2i+1}{2}}&\in \C^\times W^0[\Delta^-_{r,s;n}-\Delta^-_{r,s;n+1}]v^{(n)}_{\frac{2i+1}{2}}+\sum_{k=0}^{n}U(\mathcal{L}).v^{(k)}_{\frac{2i+1}{2}},
\end{align*}
where $v^{(n)}_{\frac{-2n-3}{2}}=v^{(n)}_{\frac{2n+3}{2}}=v^{(-1)}_{\frac{2i+1}{2}}=0$.
\end{enumerate}
\label{sl2action2}
\end{prop}

\subsection{The block decomposition of $\mathcal{C}_{p_+,p_-}$}
\label{BlockWalgebra}
Let $\mathcal{C}_{p_+,p_-}$ be the abelian category of grading restricted generalised $\W_{p_+,p_-}$-modules. Since $\W_{p_+,p_-}$ is $C_2$-cofinite, any object $M$ in $\mathcal{C}_{p_+,p_-}$ has finite length. For any $M$ in $\mathcal{C}_{p_+,p_-}$, let $M^*$ be the contragredient of $M$. Note that $\mathcal{C}_{p_+,p_-}$ is closed under contragredient.

\begin{dfn}
We define the following $\frac{(p_+-1)(p_--1)}{2}$ thick blocks, $p_++p_--2$ thin blocks and two semi-simple blocks:
\begin{enumerate}
\item For each $(r,s)\in\mathcal{T}$, we denote by $C^{thick}_{r,s}=C^{thick}_{p_+-r,p_--s}$ the full abelian subcategory of $\mathcal{C}_{p_+,p_-}$ such that
\begin{align*}
&M\in C^{thick}_{r,s}\\ 
&\Leftrightarrow\ {\rm all\ composition\ factors\ of}\ M\ {\rm are\ given\ by}\ \mathcal{X}^{+}_{r,s},\mathcal{X}^{+}_{r^{\vee},s^{\vee}},\\
&\ \ \ \ \X^-_{r^{\vee},s},\X^-_{r,s^{\vee}}\ {\rm and}\ L(h_{r,s}),
\end{align*}
where we use the notation $r^\vee=p_+-r$, $s^\vee=p_--s$ defined by (\ref{notrs}).
\item For each $1\leq s\leq p_--1$, we denote by $C^{thin}_{p_+,s}$ the full abelian subcategory of $\mathcal{C}_{p_+,p_-}$ such that
\begin{align*}
&M\in C^{thin}_{p_+,s}\\ 
&\Leftrightarrow\ {\rm all\ composition\ factors\ of}\ M\ {\rm are\ given\ by}\ \mathcal{X}^{+}_{p_+,s}\ {\rm and }\ \mathcal{X}^{-}_{p_+,s^{\vee}}.
\end{align*}
\item For each $1\leq r\leq p_+-1$, we denote by $C^{thin}_{r,p_-}$ the full abelian subcategory of $\mathcal{C}_{p_+,p_-}$ such that
\begin{align*}
&M\in C^{thin}_{r,p_-}\\ 
&\Leftrightarrow\ {\rm all\ composition\ factors\ of}\ M\ {\rm are\ given\ by}\ \mathcal{X}^{+}_{r,p_-}\ {\rm and }\ \mathcal{X}^{-}_{r^{\vee},p_-}.
\end{align*}
\item We denote by $C^{\pm}_{p_+,p_-}$ the full abelian subcategory of $\mathcal{C}_{p_+,p_-}$ such that
\begin{align*}
&M\in C^{\pm}_{p_+,p_-}\\ 
&\Leftrightarrow\ {\rm all\ composition\ factors\ of}\ M\ {\rm are\ given\ by}\ \mathcal{X}^{\pm}_{p_+,p_-}.
\end{align*}
\end{enumerate}
\end{dfn}
The following proposition can be proved by the same argument as \cite[Theorem 4.4]{AM}, using Theorem \ref{simpleclass} and the theory of Virasoro semisimple category $\mathcal{O}$ (cf. \cite{BNW,IK}). We omit the proof.
\begin{prop}
The abelian category $\mathcal{C}_{p_+,p_-}$ has the following block decomposition
\begin{align*}
\mathcal{C}_{p_+,p_-}=\bigoplus_{(r,s)\in\mathcal{T}} C^{thick}_{r,s}\oplus \bigoplus_{r=1}^{p_+-1}C^{thin}_{r,p_-}\oplus\bigoplus_{s=1}^{p_--1}C^{thin}_{p_+,s}\oplus C^+_{p_+,p_-}\oplus C^-_{p_+,p_-}.
\end{align*}
\end{prop}

The ${\rm Ext}^1$-groups between all simple modules are given by the following proposition.
\begin{prop}[\cite{Nakano}]
\mbox{}
\begin{enumerate}
\item In each thick block $C^{thick}_{r,s}$, we have 
\begin{align*}
{\rm Ext}^1(\X^{\pm},\X^{\mp})=\C^2,\ \ \ \ \ {\rm Ext}^1(L(h_{r,s}),\X^+)={\rm Ext}^1(\X^+,L(h_{r,s}))=\C,
\end{align*}
where $\X^+=\X^+_{r,s}$ or $\X^+_{r^\vee,s^\vee}$ and $\X^-=\X^-_{r^\vee,s}$ or $\X^-_{r,s^\vee}$. The other extensions between the simple modules in this block are trivial.

\item In each thin block $C^{thin}_{r,p_-}$, we have 
\begin{align*}
{\rm Ext}^1(\X^+_{r,p_-},\X^-_{r^{\vee},p_-})={\rm Ext}^1(\X^-_{r^{\vee},p_-},\X^+_{r,p_-})=\C^2.
\end{align*}
The other extensions between the simple modules in this block are trivial.
\item In each thin block $C^{thin}_{p_+,s}$, we have 
\begin{align*}
{\rm Ext}^1(\X^+_{p_+,s},\X^-_{p_+,s^\vee})={\rm Ext}^1(\X^-_{p_+,s^\vee},\X^+_{p_+,s})=\C^2.
\end{align*}
The other extensions between the simple modules in this block are trivial.
\end{enumerate}
\label{Ext}
\end{prop}

\subsection{Logarithmic $\W_{p_+,p_-}$-modules}
\label{LogarithmicWalgebra}
In this subsection, we introduce certain rank two and rank three logarithmic $\W_{p_+,p_-}$-modules.
The detailed structure of these rank two and rank three logarithmic $\W_{p_+,p_-}$-modules is studied from the boundary conformal field theory in \cite{GRW0,GRW,R} and the lattice construction in \cite{AK}. 
In our paper \cite{Nakano}, we investigated ${\rm Ext}^1$-group properties of these logarithmic $\W_{p_+,p_-}$-modules, and determined the structure of the projective covers of all simple $\W_{p_+,p_-}$-modules $\X^\pm_{r,s}$. In this subsection, we briefly review the structure of certain logarithmic modules in accordance with the notation in \cite{Nakano}. 

First let us consider the thick blocks. In each thick block $C^{thick}_{r,s}$, let $\PP^+_{r^\vee,s^\vee}$, $\PP^-_{r^\vee,s}$ and $\PP^-_{r,s^\vee}$ be the projective covers of the simple modules $\X^+_{r,s}$, $\X^+_{r^\vee,s^\vee}$, $\X^-_{r^\vee,s}$ and $\X^-_{r,s^\vee}$, respectively,
where we use the notation $r^\vee=p_+-r$, $s^\vee=p_--s$ defined by (\ref{notrs}). 
These $\W_{p_+,p_-}$-modules are logarithmic modules that have $L_0$ nilpotent rank three. The structure of socle series of these projective covers are given by the following proposition.

\begin{prop}[\cite{Nakano}]
\label{logarithmic2}
Fix any thick block $C^{thick}_{r,s}$ and let $(a,b)=(r,s)$ or $(r^\vee,s^\vee)$.
Then each of the logarithmic modules $\mathcal{P}^+_{a,b},\mathcal{P}^-_{a,b^\vee}\in C^{thick}_{r,s}$ is generated from the top composition factor and has the following length five socle series:
\mbox{}
\begin{enumerate}
\item
For the socle series of $\mathcal{P}^+_{a,b}$, we have
\begin{align*}
&S_1=\X^+_{a,b},\\
&S_2/S_1=\X^{-}_{a,b^\vee}\oplus\X^{-}_{a,b^\vee}\oplus L(h_{r,s})\oplus\X^-_{a^\vee,b}\oplus\X^-_{a^\vee,b},\\
&S_3/S_2=\X^+_{a,b}\oplus\X^+_{a^{\vee},b^{\vee}}\oplus\X^{+}_{a^{\vee},b^{\vee}}\oplus\X^+_{a^\vee,b^\vee}\oplus\X^+_{a^\vee,b^\vee}\oplus\X^+_{a,b},\\
&S_4/S_3=\X^-_{a^{\vee},b}\oplus\X^-_{a^\vee,b}\oplus L(h_{r,s})\oplus\X^{-}_{a,b^{\vee}}\oplus\X^{-}_{a,b^{\vee}},\\
&\mathcal{P}^+_{a,b}/S_4=\X^+_{a,b},
\end{align*}
where ${S}_i={\rm Soc}_i$.
\item
For the socle series of $\mathcal{P}^-_{a,b^{\vee}}$, we have
\begin{align*}
&S_1=\X^-_{a,b^\vee},\\
&S_2/S_1=\X^{+}_{a,b}\oplus\X^{+}_{a,b}\oplus\X^+_{a^\vee,b^\vee}\oplus\X^+_{a^\vee,b^\vee},\\
&S_3/S_2=\X^-_{a,b^\vee}\oplus\X^-_{a^\vee,b}\oplus\X^{-}_{a^\vee,b}\oplus L(h_{r,s})\oplus L(h_{r,s})\oplus\X^-_{a^\vee,b}\oplus\X^-_{a^\vee,b}\oplus\X^-_{a,b^\vee},\\
&S_4/S_3=\X^+_{a^\vee,b^\vee}\oplus\X^+_{a^\vee,b^\vee}\oplus \X^{+}_{a,b}\oplus\X^{+}_{a,b},\\
&\mathcal{P}^-_{a,b^{\vee}}/S_4=\X^-_{a,b^\vee}.
\end{align*}
\end{enumerate}
\end{prop}


For $1\leq r\leq p_+-1$ and $1\leq s\leq p_--1$, we set
\begin{equation}
\label{Qthick}
\begin{split}
{Q}^{thick}_{r,s}=&\bigl\{(+,r,s,r^\vee,s),(+,r,s,r,s^\vee),(+,r^\vee,s^\vee,r^\vee,s),(+,r^\vee,s^\vee,r,s^\vee),\\
&(-,r^\vee,s,r,s),(-,r^\vee,s,r^\vee,s^\vee),(-,r,s^\vee,r,s),(-,r,s^\vee,r^\vee,s^\vee)\bigr\}.
\end{split}
\end{equation}

\begin{prop}[\cite{Nakano}]
\label{logarithmic3}
Fix any thick block $C^{thick}_{r,s}$. By taking quotients of $\mathcal{P}^+_{r,s}$, $\mathcal{P}^+_{r^\vee,s^\vee}$, $\mathcal{P}^-_{r^\vee,s}$ and $\mathcal{P}^-_{r,s^\vee}$, we obtain eight indecomposable modules 
\begin{align*}
\mathcal{Q}(\X^\epsilon_{a,b})_{b,c},\ \ (\epsilon,a,b,c,d)\in {Q}^{thick}_{r,s}
\end{align*}
whose socle length are three and the socle series are given by:
\begin{enumerate}
\item For $\mathcal{Q}(\X^+_{a,b})_{c,d}$,
\begin{align*}
&S_1=\X^+_{a,b},&&S_2/S_1=\X^{-}_{c,d}\oplus L(h_{a,b})\oplus\X^-_{c,d},&&&\mathcal{Q}(\X^+_{a,b})_{c,d}/S_2=\X^+_{a,b}.
\end{align*}
\item For $\mathcal{Q}(\X^-_{a,b})_{c,d}$,  
\begin{align*}
&S_1=\X^-_{a,b},&&S_2/S_1=\X^{+}_{c,d}\oplus\X^+_{c,d},&&&\mathcal{Q}(\X^-_{a,b})_{c,d}/S_2=\X^-_{a,b}.
\end{align*}
\end{enumerate}
\end{prop}

The ${\rm Ext}^1$-groups between the indecomposable modules $\mathcal{Q}(\X^\pm_{\bullet,\bullet})_{\bullet,\bullet}\in C^{thick}_{r,s}$ and the simple modules are given by the following proposition.
\begin{prop}[\cite{Nakano}]
\label{ExtQ}
Fix any thick block $C^{thick}_{r,s}$. For any $(\epsilon,a,b,c,d)\in {Q}^{thick}_{r,s}$, we have
\begin{align*}
&{\rm Ext}^1(\mathcal{Q}(\X^\epsilon_{a,b})_{c,d},L(h_{r,s}))=0,\ \ \ {\rm Ext}^1(\mathcal{Q}(\X^\epsilon_{a,b})_{c,d},\X^{\epsilon}_{a,b})=0,\\
&{\rm Ext}^1(\mathcal{Q}(\X^\epsilon_{a,b})_{c,d},\X^{\epsilon}_{a^\vee,b^\vee})=0,\ \ \ \ {\rm Ext}^1(\mathcal{Q}(\X^\epsilon_{a,b})_{c,d},\X^{-\epsilon}_{c,d})=0.\\
\end{align*}
\end{prop}
\vspace{2mm}
Next let us consider the thin blocks. In each of $C^{thin}_{r,p_-}$ and $C^{thin}_{p_+,s}$, let $\mathcal{Q}(\X^+_{r,p_-})_{r^\vee,p_-},\mathcal{Q}(\X^-_{r^\vee,p_-})_{r,p_-}\in C^{thin}_{r,p_-}$ and $\mathcal{Q}(\X^+_{p_+,s})_{p_+,s^\vee},\mathcal{Q}(\X^-_{p_+,s^\vee})_{p_+,s}\in C^{thin}_{p_+,s}$ be the projective covers of $\X^+_{r,p_-}$, $\X^-_{r^\vee,p_-}$, $\X^+_{p_+,s}$ and $\X^-_{p_+,s^\vee}$, respectively. These $\W_{p_+,p_-}$-modules have $L_0$ nilpotent rank two.
For $1\leq r\leq p_+-1$ and $1\leq s\leq p_--1$, we set
\begin{equation}
\label{Qthin}
\begin{split}
&{Q}^{thin}_{r,p_-}=\bigl\{(+,r,p_-,r^\vee,p_-),(-,r^\vee,p_-,r,p_-)\bigr\},\\
&{Q}^{thin}_{p_+,s}=\bigl\{(+,p_+,s,p_+,s^\vee),(-,p_+,s^\vee,p_+,s)\bigr\}.
\end{split}
\end{equation}

\begin{prop}[\cite{Nakano}]
\label{logarithmic4}
Fix any $(r,s)\in \mathbb{Z}^2_{\geq 1}$ such that $1\leq r<p_+$ and $1\leq s<p_-$. For any $(\epsilon,a,b,c,d)\in {Q}^{thin}_{r,p_-}\sqcup{Q}^{thin}_{p_+,s}$, the socle length of $\mathcal{Q}(\mathcal{X}^\epsilon_{a,b})_{c,d}$ is three and the structure of socle series of $\mathcal{Q}(\mathcal{X}^\epsilon_{a,b})_{c,d}$ is given by
\begin{align*}
&S_1(\mathcal{Q}(\mathcal{X}^\epsilon_{a,b})_{c,d})=\mathcal{X}^\epsilon_{a,b},\\
&S_2(\mathcal{Q}(\mathcal{X}^\epsilon_{a,b})_{c,d})/S_1(\mathcal{Q}(\mathcal{X}^\epsilon_{a,b})_{c,d})=\mathcal{X}^{-\epsilon}_{c,d}\oplus\mathcal{X}^{-\epsilon}_{c,d},\\
&\mathcal{Q}(\mathcal{X}^\epsilon_{a,b})_{c,d}/S_2(\mathcal{Q}(\mathcal{X}^\epsilon_{a,b})_{c,d})=\mathcal{X}^\epsilon_{a,b}.
\end{align*}
\end{prop}

\section{Non-semisimple fusion rules}
\label{FusionRules}
Since the triplet $W$ algebra $\W_{p_+,p_-}$ is $C_2$-cofinite, Theorem 4.13 in \cite{H} show that $\mathcal{C}_{p_+,p_-}$ has braided tensor category structure as developed in the series of papers \cite{HLZ1,HLZ2,HLZ3,HLZ4,HLZ5,HLZ6,HLZ7,HLZ8}. We denote by ($\mathcal{C}_{p_+,p_-},\boxtimes,\mathcal{K}_{1,1})$ this vertex tensor category, where $\boxtimes$ is the tensor product (see Definition \ref{ten0819}) and $\K_{1,1}:=\W_{p_+,p_-}.\ket{\alpha_{1,1}}$ is the unit object. 
Note that the tensor product $\boxtimes$ of $(\mathcal{C}_{p_+,p_-},\boxtimes,\mathcal{K}_{1,1})$ is right exact (see \cite[Proposition 4.26]{HLZ3}).
The tensor category ($\mathcal{C}_{p_+,p_-},\boxtimes,\mathcal{K}_{1,1}$) is not rigid as shown in \cite{GRW,GRW0,W}. 
This makes it more difficult to study the structure of the tensor category ($\mathcal{C}_{p_+,p_-},\boxtimes,\mathcal{K}_{1,1}$) compared to the triplet $W$-algebra $\W_{p}$ (cf.\cite{McRae},\cite{TWFusion}). 

In this section, 
we will introduce certain indecomposable modules $\K_{r,s}$ following \cite{FF2,MS,W}, and will show the rigidity for $\K_{1,2}$ and $\K_{2,1}$ in Theorems \ref{rigid12ori} and \ref{rigid21}, using the vertex tensor categorical methods detailed in \cite{CMY,MS,McRae,TWFusion}. 
Using the rigidity for $\K_{1,2}$ and $\K_{2,1}$, we show that all indecomposable modules $\K_{r,s}$, $\mathcal{Q}(\X^\pm_{r,s})_{\bullet,\bullet}$ and $\mathcal{P}^\pm_{r,s}$ appear in the repeated tensor products of $\K_{1,2}$ and $\K_{2,1}$. As a result we see that all simple modules in the thin blocks, all indecomposable modules $\K_{r,s}$, $\mathcal{Q}(\X^\pm_{r,s})_{\bullet,\bullet}$ and $\mathcal{P}^\pm_{r,s}$ are rigid and self-dual objects. The self-duality for these indecomposable modules was conjectured in \cite{GRW0}. 
In Subsection \ref{simple0817}, we also determine the tensor product between all simple modules. 

For basic properties of general tensor categories, see \cite{Etingof,JS} and \cite[Appendix A]{KL}.

\subsection{Tensor product $\boxtimes$ and $P(w)$-intertwining operators}
\label{Pint2023}

In this section, we review the definition of the tensor product $\boxtimes$ and $P(w)$-intertwining operators in accordance with \cite{AW,HLZ3,Kanade} and derive some identities known as the Nahm-Gaberdiel-Kausch fusion algorithm(cf. \cite{GK}). 
\begin{dfn}
\label{ten0819}
Let $V$ be a vertex operator algebra and let $\mathcal{C}$ be a category of grading-restricted generalized $V$-modules. A {\rm tensor product} $($or {\rm fusion product}$)$ of $M_1$ and $M_2$ in $\mathcal{C}$ is a pair $(M_1\boxtimes M_2,\mathcal{Y}_{\boxtimes})$, with $M_1\boxtimes M_2$ and $\mathcal{Y}_{\boxtimes}$ an intertwining operator of type 
{\scriptsize{$\begin{pmatrix}
   M_1\boxtimes M_2  \\
   M_1\ M_2
\end{pmatrix}
$}}, 
which satisfies the following universal property: For any $M_3\in\mathcal{C}$ and intertwining operator $\mathcal{Y}$ of type 
{\scriptsize{$\begin{pmatrix}
   M_3  \\
   M_1\ M_2
\end{pmatrix}
$}}, there is a unique $V$-module homomorphism $f:M_1\boxtimes M_2\rightarrow M_3$ such that $\mathcal{Y}=f\circ \mathcal{Y}_{\boxtimes}$.
\end{dfn}
In the paper \cite{HLZ3}, the notion of $P(w)$-intertwining operators and the $P(w)$-tensor product are introduced. The definitions are as follows. 
\begin{dfn}[\cite{HLZ3}]
\label{P(w)}
Fix $w\in \C^\times$. Let $V$ be a vertex operator algebra and let $\mathcal{C}$ be a category of grading-restricted generalized $V$-modules. Given $M_1$, $M_2$ and $M_3$ in $\mathcal{C}$, a $P(w)$-{\rm intertwining operator} $I$ {\rm of type}
{\scriptsize{$\begin{pmatrix}
   \ M_3  \\
   M_1\ M_2
\end{pmatrix}
$}}
is a bilinear map $I$ $:$ $M_1\otimes M_2\rightarrow M_3$ that satisfies the following properties:
\begin{enumerate}
\item For any $\psi_1\in M_1$ and $\psi_2\in M_2$, $\pi_h(I[\psi_1\otimes \psi_2])=0$ for all $h\ll0$, where $\pi_h$ denotes the projection onto the generalized eigenspace $M_3[h]$ of $L_0$-eigenvalue $h$.
\item For any $\psi_1\in M_1$, $\psi_2\in M_2$, $\psi^*_3\in M^*_3$ and $A\in V$, the three point functions
\begin{align*}
\langle \psi^*_3, Y_3(A,z)I[\psi_1\otimes \psi_2]\rangle,\ 
\langle \psi^*_3, I[Y_1(A,z-w)\psi_1\otimes \psi_2]\rangle,\  
\langle \psi^*_3, I[\psi_1\otimes Y_2(A,z)\psi_2]\rangle
\end{align*}
are absolutely convergent in the regions $|z|>|w|>0$, $|w|>|z-w|>0$, $|w|>|z|>0$, respectively, where $Y_i$ is the action of $V$-module.
\item Given any $f(t)\in R_{P(w)}:=\C[t,t^{-1},(t-w)^{-1}]$, we have the following identity
\begin{equation}
\label{Pw-com}
\begin{split}
&\oint_{0,w}f(z)\langle \psi^*_3, Y_3(A,z)I[\psi_1\otimes \psi_2]\rangle\frac{{\rm d}z}{2\pi i}\\
&=\oint_{w}f(z)\langle\psi^*_3, I[Y_1(A,z-w)\psi_1\otimes \psi_2]\rangle\frac{{\rm d}z}{2\pi i}\\
&\ \ \ +\oint_{0}f(z)\langle \psi^*_3, I[\psi_1\otimes Y_2(A,z)\psi_2]\rangle\frac{{\rm d}z}{2\pi i}.
\end{split}
\end{equation}
\end{enumerate}
\end{dfn}

\begin{dfn}[\cite{HLZ3}]
Let $V$ be a vertex operator algebra and let $\mathcal{C}$ be a category of grading-restricted generalized $V$-modules. A $P(w)$-{\rm tensor product} of $M_1$ and $M_2$ in $\mathcal{C}$ is a pair $(M_1\boxtimes_{P(w)} M_2,{\boxtimes_{P(w)}})$, with $M_1\boxtimes_{P(w)} M_2$ and ${\boxtimes_{P(w)}}$ a $P(w)$-intertwining operator of type 
{\scriptsize{$\begin{pmatrix}
   M_1\boxtimes_{P(w)} M_2  \\
   M_1\ \ M_2
\end{pmatrix}
$}}, 
which satisfies the following universal property: For any $M_3\in\mathcal{C}$ and $P(w)$-intertwining operator $I$ of type 
{\scriptsize{$\begin{pmatrix}
   M_3  \\
   M_1\ M_2
\end{pmatrix}
$}}, there is a unique $V$-module homomorphism $\eta:M_1\boxtimes_{P(w)} M_2\rightarrow M_3$ such that 
\begin{align*}
\overline{\eta}\circ {\boxtimes_{P(w)}}[\psi_1\otimes \psi_2]=I[\psi_1\otimes \psi_2]
\end{align*}
for all $\psi_1\in M_1$ and $\psi_2\in M_2$, where $\overline{\eta}$ denotes the extension of $\eta$ to a map between the completions of $M_1\boxtimes_{P(w)} M_2$ and $M_3$.
\end{dfn}

\begin{remark}
It is known that the definition $P(w)$-tensor product $\boxtimes_{P(w)}$ does not depend on the choice of $w\in \C^\times$. See \cite[Remark 4.22]{HLZ3}.
\end{remark}
\begin{prop}[\cite{HLZ3}]
Let $V$ be a vertex operator algebra and let $\mathcal{C}$ be a category of grading-restricted generalized $V$-modules.
Given $M_1$, $M_2$ and $M_3$ in $\mathcal{C}$, there exists a linear isomorphism from the space of intertwining operators of type {\scriptsize{$\begin{pmatrix}
   \ M_3  \\
   M_1\ M_2
\end{pmatrix}
$}} to the space of $P(w)$-intertwining operators of the same type.
\end{prop}
By this proposition, the structure of the space of intertwining operators can be determined from the structure of $P(w)$-intertwining operators of the same type.
In the following, we will introduce some formulas derived from the $P(w)$-compatibility condition (\ref{Pw-com}).
Before we do that, let us define the following symbols. 
\begin{dfn}
\begin{enumerate}
\item Let $V$ be a vertex operator algebra. For any $L_0$ homogeneous element $A\in V$, whose $L_0$ weight is $h$, we define
\begin{align}
\label{A0819}
A[n]:=\oint_{z=0}Y(A,z)z^{n+h-1}{\rm d}z.
\end{align}
\item We define a translation map
\begin{align*}
T_1:\C(t)\rightarrow \C(t),\ \ \ \ \ \ \ {\rm by}\ \ \ \ \ \ \ f(t)\mapsto f(t+1),
\end{align*}
and a expansion map 
\begin{align*}
\iota_+:\C(t)\hookrightarrow \C((t))
\end{align*}
that expands a given rational function in $t$ as a power series around $t=0$.
\end{enumerate}
\end{dfn}
The following definition is derived from the condition (\ref{Pw-com}) \cite{AW,Kanade}.
\begin{dfn}
Let $V$ be a vertex operator algebra and let $\mathcal{C}$ be a category of grading-restricted generalized $V$-modules. Given $M_1, M_2, M_3\in \mathcal{C}$ and a $P(1)$-intertwining operator $I$ of type
{\scriptsize{$\begin{pmatrix}
   \ M_3  \\
   M_1\ M_2
\end{pmatrix}
$}}, we define the action of $V\otimes \C[t,t^{-1},(t-1)^{-1}]$ or $V\otimes \C((t))$ on $M^*_3$ as 
\begin{equation}
\begin{split}
&\langle Af(t)\psi^*_3, I[\psi_1\otimes\psi_1]\rangle=\langle A\iota_+(f(t))\psi^*_3, I[\psi_1\otimes\psi_1]\rangle\\
&=\langle \psi^*_3,I[\iota_+\circ T_1\bigl(A^{{\rm opp}}f(t^{-1})\bigl)\psi_1\otimes\psi_2]\rangle\\
&\ \ \ \ \ \ \ +\langle\psi^*_3,I[\psi_1\otimes\iota_+\bigl(A^{{\rm opp}}f(t^{-1})\bigl)\psi_2]\rangle
\end{split}  
\label{NGK0}
\end{equation}
where $\psi^*_3\in M^*_3$, $\psi_i\in M_i(i=1,2)$, 
\begin{align*}
A^{{\rm opp}}:=e^{t^{-1}L_{1}}(-t^2)^{L_0}At^{-2},
\end{align*}
and (assuming $A\in V$ has $L_0$ weight $h$) $At^n\psi_i=A[{n-h+1}]\psi_i$.
\end{dfn}
\begin{lem}
\label{NGK01}
Let $V$ be a vertex operator algebra and let $\mathcal{C}$ be a category of grading-restricted generalized $V$-modules. Let $A\in V$ be a non-zero Virasoro primary vector with $L_0$ conformal weight $h$, that is 
\begin{align*} 
&L_0A=hA,
&L_nA=0,\ \ \ n\geq 1.
\end{align*}
Then, given $M_1, M_2, M_3\in \mathcal{C}$ and a $P(1)$-{\rm intertwining operator} $I$ {\rm of type}
{\scriptsize{$\begin{pmatrix}
   \ M_3  \\
   M_1\ M_2
\end{pmatrix}
$}}, we have the following identities:
\begin{equation*}
\begin{split}
&\langle A[n]\psi^*_3, I[\psi_1\otimes\psi_2]\rangle\\
&=\sum_{i=0}^\infty\binom{h-n-1}{i}\langle\psi^*_3,I[\bigl(A[{i-h+1}]\psi_1\bigr)\otimes\psi_2]\rangle+\langle\psi^*_3,I[\psi_1\otimes\bigl(A[{-n}]\psi_2\bigr)]\rangle
\end{split}
\end{equation*}
and 
\begin{equation*}
\begin{split}
&\sum_{m=0}^\infty\binom{n+h-2}{m}(-1)^m\langle A[m-n]\psi^*_3,I[\psi_1\otimes\psi_2]\rangle\\
&=\langle\psi^*_3,I[\bigl((A[{n-1}]+A[{n}])\psi_1\bigr)\otimes\psi_2]\rangle\\
&\ \ \ \ \ \ \ \ \ +\sum_{i=0}^\infty\binom{n+h-2}{i}(-1)^{n-i+h-2}\langle\psi^*_3,I[\psi_1\otimes\bigl(A[{i-h+2}]\psi_2\bigr)]\rangle,
\end{split}
\end{equation*}
where $\psi^*_3\in M^*_3$, $\psi_i\in M_i$ $(i=1,2)$ and $n\in \mathbb{Z}$.
\begin{proof}
For the first identity, let $f(t)=t^{n+h-1}$ in (\ref{NGK0}), and for the second identity, let $f(t)=t^{2h-3}(t^{-1}-1)^{n+h-2}$ in (\ref{NGK0}).
\end{proof}
\end{lem}


Hereafter we omit $P(1)$-intertwining operators and use the abbreviation as 
\begin{align*}
 \langle \psi^*_3,\psi_1\otimes\psi_2\rangle= \langle \psi^*_3,I[\psi_1\otimes\psi_2]\rangle
\end{align*}
unless otherwise noted.\\

\subsection{Tensor products $L(h_{r,s})\boxtimes \bullet$}
In this subsection, we determine the tensor products $L(h_{r,s})\boxtimes M$ $(M\in\mathcal{C}_{p_+,p_-})$. First we introduce the tensor products between the minimal simple modules $L(h_{r,s})$. Recall that the triplet algebra $\mathcal{W}_{p_+,p_-}$ satisfies the exact sequence
\begin{align*}
0\rightarrow \mathcal{X}^+_{1,1}\rightarrow \mathcal{W}_{p_+,p_-}\rightarrow L(c_{p_+,p_-},0)\rightarrow 0,
\end{align*}
where $L(c_{p_+,p_-},0)$ is the Virasoro minimal vertex operator algebra.
Note the following proposition.
\begin{prop}[\cite{AMW2p,AMW3p,TW}]
\label{prop0830}
The maximal ideal $\X^+_{1,1}$ of $\W_{p_+,p_-}$ acts trivially on the minimal simple modules $L(h_{r,s})$ $(r,s)\in\mathcal{T}$.
\end{prop}
From this proposition, we have the following BPZ fusion rules.
\begin{prop}[\cite{BPZ}]
\label{minimal fusion}
For $1\leq r,r'\leq p_+-1$, $1\leq s,s'\leq p_--1$, we have
\begin{align*}
L(h_{r,s})\boxtimes L(h_{r',s'})=\bigoplus_{\substack{i=1+|r-r'| \\ i+r+r'=1\ {\rm mod}\ 2}}^{{\min}\{r+r'-1,2p_+-r-r'-1\}}   \bigoplus_{\substack{j=1+|s-s'|\\j+s+s'=1\ {\rm mod}\ 2}}^{{\rm min}\{ s+s'-1,2p_--s-s'-1\}}L(h_{i,j}).
\end{align*}
\begin{proof}
By Proposition \ref{prop0830}, we see that the fusion product $L(h_{r,s})\boxtimes L(h_{r',s'})$ in $(\mathcal{C}_{p_+,p_-},\boxtimes,\mathcal{K}_{1,1})$ can be considered the fusion product in the tensor category $(L(c_{p_+,p_-}),\boxtimes, L(0))$. Therefore, the explicit formula for $L(h_{r,s})\boxtimes L(h_{r',s'})$ is obtained from the results of \cite{BPZ}.
\end{proof}
\end{prop}

From this subsection, we will use the following notation and propositions frequently.
\begin{dfn}
\label{dfn0830gogo}
Let $V$ be a vertex operator algebra.
\begin{enumerate}
\item For any $M\in V\mathchar`-{\rm Mod}$, we define the following vector space 
\begin{align*}
A_0(M)=\{\psi\in M\ |\ \psi\neq 0,\ A[n]\psi=0,\ {}^\forall A\in V,\ n>0\},
\end{align*}
where $A[n]$ is defined by (\ref{A0819}).
\item For any $M\in V\mathchar`-{\rm Mod}$, we define {\rm the top composition factors} of $M$ as the maximal semisimple quotient of $M$. We denote by ${\rm top}(M)$ the top composition factors of $M$. In the case of $V=\mathcal{W}_{p_+,p_-}$, we have
\begin{align*}
{\rm top}(M)={\rm Socle}(M^*),
\end{align*}
where $M^*$ is the contragredient of $M$.
\item Given $V$-modules $M_1,M_2,M_3$, we denote by
\begin{equation*}
I
\begin{pmatrix}
\ M_3 \\
M_2\ \ M_1
\end{pmatrix}
\end{equation*}
the vector space of $V$-intertwining operators of type {\footnotesize{$\begin{pmatrix}
\ M_3 \\
M_2\ \ M_1
\end{pmatrix}$}}.
\end{enumerate}
\end{dfn}
Let $M(h,c_{p_+,p_-})$ be the Verma module of the Virasoro algebra whose highest weight is $h\in\mathbb{C}$ and the central charge $C=c_{p_+,p_-}\cdot {\rm id}$. Let ${\mid}h\rangle$ be the highest weight vector of $M(h,c_{p_+,p_-})$.
Note that, for $r,s\geq 1$, $M(h_{r,s},c_{p_+,p_-})$ has the singular vector whose $L_0$-weight is $h_{r,s}+rs$. Let $S_{r,s}\in U(\mathcal{L})$ be the Shapovalov element corresponding to this singular vector, normalized as
\begin{align*}
S_{r,s}{\mid}h_{r,s}\rangle=(L^{rs}_{-1}+\cdots){\mid}h_{r,s}\rangle.
\end{align*}
Let $\sigma$ be the anti-involution of $U(\mathcal{L})$ defined by $\sigma(L_{n})=L_{-n}$ $(n\in \mathbb{Z})$, and set $S^*_{r,s}=\sigma(S_{r,s})$. We have the following propositions for non-logarithmic Virasoro intertwining operators.
\begin{prop}[\cite{FF,IK,Milas,Lin}]
\label{FusionInt}
Let $M_1$, $M_2$ and $M_3$ be Virasoro highest weight modules whose $L_0$ weights are $h_{r_1,s_1}$, $h_{r_2,s_2}$ and $h_{r_3,s_3}$, respectively, where $r_i,s_i\geq 1$ $(i=1,2,3)$. Let $v_i(i=1,2,3)$ be the highest weight vectors of $M_i$.
Assume that there exists a non-logarithmic intertwining operator $\mathcal{Y}$ of type {\scriptsize{$\begin{pmatrix}
   M^*_3  \\
   M_1\ M_2
\end{pmatrix}
$}}.
Then we have
\begin{align*}
&\langle v_3,\mathcal{Y}(v_1,z)S_{r_2,s_2}v_2\rangle=\prod_{i=1}^{r_2}\prod_{j=1}^{s_2}(h_{r_1,s_1}-h_{r_2+r_3-2i+1,s_2+s_3-2j+1})\langle v_3,\mathcal{Y}(v_1,z)v_2\rangle,\\
&\langle v_3,S^*_{r_3,s_3}\mathcal{Y}(v_1,z)v_2\rangle=\prod_{i=1}^{r_3}\prod_{j=1}^{s_3}(h_{r_1,s_1}-h_{r_2+r_3-2i+1,s_2+s_3-2j+1})\langle v_3,\mathcal{Y}(v_1,z)v_2\rangle.
\end{align*}
\end{prop}

\begin{prop}[\cite{FF,IK,Milas,Lin}]
For $h\in\C$, $1\leq r_1,r_2<p_+$, $1\leq s_1,s_2<p_-$ and $n_1,n_2\in\Z_{\geq 0}$, we have
\begin{equation*}
\mathcal{N}^{L(h)}_{L(h_{r_1,s_1;n_1}),L(h_{r_2,s_2;n_2})}:={\rm dim}_{\mathbb{C}}I\begin{pmatrix}
   L(h) \\
   L(h_{r_1,s_1;n_1})\ L(h_{r_2,s_2;n_2})
\end{pmatrix}\leq 1,
\end{equation*}
If $\mathcal{N}^{L(h)}_{L(h_{r_1,s_1;n_1}),L(h_{r_2,s_2;n_2})}\neq 0$, then $h$ is the common solution of the following equations
\begin{align*}
\prod_{i=1}^{r_1}\prod_{j=1}^{s_1+n_1p_-}&(h-h_{r_1+r_2-2i+1,s_1+s_2-2j+1;n_1+n_2})=0,\\
\prod_{i=1}^{(n_1+1)p_+-r_1}\prod_{j=1}^{p_--s_1}&(h-h_{2p_+-r_1-r_2-2i+1,2p_--s_1-s_2-2j+1;-n_1-n_2})=0,\\
\prod_{i=1}^{r_2}\prod_{j=1}^{s_2+n_2p_-}&(h-h_{r_1+r_2-2i+1,s_1+s_2-2j+1;n_1+n_2})=0,\\
\prod_{i=1}^{(n_2+1)p_+-r_2}\prod_{j=1}^{p_--s_2}&(h-h_{2p_+-r_1-r_2-2i+1,2p_--s_1-s_2-2j+1;-n_1-n_2})=0.
\end{align*}
\label{VirasoroFusion}
\end{prop}

\begin{prop}
\label{null11rs}
\begin{align*}
\X^\pm_{r,s}\boxtimes L(h_{1,1}) =0,\ \ \ 1\leq r\leq p_+,\ 1\leq s\leq p_-.
\end{align*}
\begin{proof}
Assume $\X^+_{r,s}\boxtimes L(h_{1,1})\neq 0$. Fix any non-zero vector 
\begin{align*}
\psi^*_3\in A_0((\X^+_{r,s}\boxtimes L(h_{1,1}))^*), 
\end{align*}
and let $\psi_1$ and $\psi_2$ be the highest weight vectors of $\X^+_{r,s}$ and $L(h_{1,1})$, respectively. Let $\phi_1$ and $\phi_2$ be arbitrary vectors of $\mathcal{X}^+_{1,1}$ and $L(h_{1,1})$ such that $\langle \psi^*_3, \phi_1\otimes \phi_2\rangle\neq 0$.
Note that the maximal ideal $\X^+_{1,1}$ of $\W_{p_+,p_-}$ acts trivially on any minimal simple modules $L(h_{r,s})((r,s)\in \mathcal{T})$, and $\psi_2$ satisfies $L_{-1}\psi_2=0$.
Then, noting Propositions \ref{genW} and \ref{sl2action2}, and using the formulas in Lemma \ref{NGK01}, we obtain 
\begin{align*}
\langle \psi^*_3, \phi_1\otimes \phi_2\rangle=c \langle \psi^*_3,\psi_1\otimes\psi_2\rangle,
\end{align*}
where $c$ is a non-zero constant.
Using the formulas in Lemma \ref{NGK01}, we have
\begin{align*}
\langle L_0\psi^*_3,\psi_1\otimes\psi_2\rangle
&=\langle \psi^*_3,L_{-1}\psi_1\otimes\psi_2\rangle+\langle \psi^*_3,L_0\psi_1\otimes\psi_2\rangle+\langle \psi^*_3,\psi_1\otimes L_0\psi_2\rangle\\
&=\langle \psi^*_3,L_0\psi_1\otimes \psi_2\rangle.
\end{align*}
Thus, the $L_0$-eigenvalue of $\psi^*_3$ is the same as that of $\psi_1$. Then, by restricting the action of $\W_{p_+,p_-}$ to the Virasoro action in the intertwining operator $\mathcal{Y}_{\boxtimes}$, we have a non-trivial non-logarithmic Virasoro intertwining operator of type
\begin{equation*}
\begin{pmatrix}
\ L(\Delta^+_{r,s;0}) \\
L(\Delta^+_{r,s;0})\ \ L(h_{1,1})
\end{pmatrix}
.
\end{equation*}
Note that $S_{p_+-1,p_--1}\psi_2=0$. Then by using Proposition \ref{FusionInt}, we have
\begin{align*}
0&=\langle \psi^*_3,\psi_1\otimes S_{p_+-1,p_--1}\psi_2\rangle\\
&=\prod_{i=1}^{p_+-1}\prod_{j=1}^{p_--1}(h_{r^\vee+p_+,s}-h_{p_+-1+r^\vee+p_+-2i+1,p_--1+s-2i+1})\langle \psi^*_3,\psi_1\otimes \psi_2\rangle.
\end{align*}
We see that the coefficient of the above equation is non-zero. Thus we obtain $\langle \psi^*_3,\psi_1\otimes \psi_2\rangle=0$. But this contradicts the assumption. Similarly, we can prove $L(h_{1,1})\boxtimes \X^-_{r,s}=0$.

\end{proof}
\end{prop}

By Propositions \ref{minimal fusion} and \ref{null11rs}, we obtain the following proposition.
\begin{prop}
\label{minimalnull}
For any simple module $\X^{\pm}_{r,s}(1\leq r\leq p_+,1\leq s\leq p_-)$, we have
\begin{align*}
\X_{r,s}\boxtimes L(h_{r',s'})=0,\ \ \ 1\leq r'\leq p_+-1,\ 1\leq s'\leq p_--1.
\end{align*}
\end{prop}
\begin{corollary}
\label{202210180}
For $1\leq r,r'\leq p_+-1$, $1\leq s,s'\leq p_--1$, we have
\begin{align*}
(\W_{p_+,p_-}.\ket{\alpha_{r,s}})\boxtimes L(h_{r',s'})\simeq L(h_{r,s})\boxtimes L(h_{r',s'}).
\end{align*}
\begin{proof}
Note that, for $0<r<p_+,0<s<p_-$, $\W_{p_+,p_-}.\ket{\alpha_{r,s}}\subset \mathcal{V}^+_{r,s}$. Thus, by Proposition \ref{socleV}, $\W_{p_+,p_-}.\ket{\alpha_{r,s}}$ satisfies the exact sequence
\begin{align*}
0\rightarrow \mathcal{X}^+_{r,s}\rightarrow \W_{p_+,p_-}.\ket{\alpha_{r,s}}\rightarrow L(h_{r,s})\rightarrow 0.
\end{align*}
Therefore, from Proposition \ref{minimalnull}, we obtain the equality in the claim.
\end{proof}
\end{corollary}
\begin{corollary}
\label{directsumL}
For any $(r,s)\in \mathcal{T}$ and $M\in \W_{p_+,p_-}$, $L(h_{r,s})\boxtimes M$ becomes a direct sum of minimal simple modules $L(h_{k,l})$ $(k,l)\in \mathcal{T}$.
\end{corollary}
\subsection{Self-dual objects $\K_{r,s}$}
We define the following indecomposable modules following \cite{FF2,MS,W}.
\begin{dfn}
We define the following $\W_{p_+,p_-}$-modules
\begin{enumerate}
\item For $1\leq r\leq p_+-1$, $1\leq s\leq p_--1$, 
\begin{align*}
\K_{r,s}:=\W_{p_+,p_-}.\ket{\alpha_{r,s}}.
\end{align*}
\item For $1\leq r\leq p_+$, $1\leq s\leq p_-$,
\begin{align*}
&\K_{r,p_-}:=\X^+_{r,p_-},
&\K_{p_+,s}:=\X^+_{p_+,s}.
\end{align*}
\end{enumerate}
\label{W(beta)}
\end{dfn}

In this subsection, we compute some tensor products $\K_{1,2}\boxtimes \bullet $ and $\K_{2,1}\boxtimes \bullet$, and show that the indecomposable modules $\K_{r,s}$ are rigid and self-dual. 
\vspace{2mm}

For any highest weight $\mathcal{W}_{p_+,p_-}$-module $M$, we denote by $H(M)$ the highest weight of $M$. 
\begin{prop}
\label{NGK2}
Assume that
\begin{align*}
&\mathcal{K}_{1,2}\boxtimes \mathcal{K}_{r,s}\neq 0,
&\mathcal{K}_{1,2}\boxtimes \mathcal{X}^\pm_{r,s}\neq 0
\end{align*}
for all $r,s\geq 1$.
Then the following holds.
\begin{enumerate}
\item For $1\leq r\leq p_+$, $2\leq s< p_-$, 
the dimension of the vector space $A_0((\K_{1,2}\boxtimes\K_{r,s})^*)$ is at most two dimensional. $L_0$ acts semisimply on this vector space and the $L_0$ eigenvalues are contained in 
\begin{align*}
\{{H}(\K_{r,s+1}),{H}(\K_{r,s-1})\}=\{h_{r,s-1},h_{r,s+1}\}.
\end{align*}
\item For $2\leq r\leq p_+$, 
the dimension of the vector space $A_0((\K_{1,2}\boxtimes\K_{r,1})^*)$ is at most one dimensional. $L_0$ acts semisimply on this vector space and the $L_0$ eigenvalue is given by ${H}(\K_{r,2})=h_{r,2}$.

\item For $1\leq r\leq p_+$, $2\leq s< p_-$,
the dimension of the vector space $A_0((\K_{1,2}\boxtimes\X^+_{r,s})^*)$ is at most two dimensional. $L_0$ acts semisimply on this vector space and the $L_0$ eigenvalues are contained in 
\begin{align*}
\{H(\X^+_{r,s-1}), H(\X^+_{r,s+1})\}=\{\Delta^+_{r,s-1;0},\Delta^+_{r,s+1;0}\}.
\end{align*}
\item For $1\leq r\leq p_+$, $2\leq s< p_-$,
the dimension of the vector space $A_0((\K_{1,2}\boxtimes\X^-_{r,s})^*)$ is at most four dimensional. $L_0$ acts semisimply on this vector space and the $L_0$ eigenvalues are contained in 
\begin{align*}
\{H(\X^-_{r,s-1}),H(\X^-_{r,s+1})\}=\{\Delta^-_{r,s-1;0},\Delta^-_{r,s+1;0}\}.
\end{align*}
\item For $1\leq r\leq p_+$, $s=p_-$, the vector space $A_0((\K_{1,2}\boxtimes\X^+_{r,p_-})^*)$ is at most two dimensional.  The generalized $L_0$ eigenvalues of this vector space are contained in 
\begin{align*}
\{H(L(h_{r,p_--1})), H(\X^+_{r,p_--1})\}=\{h_{r,p_--1},\Delta^+_{r,p_--1;0}\}.
\end{align*}
\item For $1\leq r\leq p_+$, $s=p_-$, the vector space $A_0((\K_{1,2}\boxtimes\X^-_{r,p_-})^*)$ is at most four dimensional. $L_0$ acts semisimply on this vector space and the $L_0$ eigenvalues are contained in 
\begin{align*}
\{H(\X^+_{r,1}),H(\X^-_{r,p_--1})\}=\{\Delta^+_{r,1;0},\Delta^-_{r,p_--1;0}\}.
\end{align*}
\end{enumerate}
\begin{proof}
We only prove the first and second cases.
The other cases can be proved in the same way.

Let us prove the first case. Assume $\K_{1,2}\boxtimes\K_{r,s}\neq 0$.
Let $\psi^*$ be an arbitrary non-zero vector of $A_0((\K_{1,2}\boxtimes\K_{r,s})^*)$, where the notation $A_0(\bullet)$ is given in Definition \ref{dfn0830gogo}.
Let $\phi_1$ and $\phi_2$ be arbitrary vectors of $\K_{1,2}$ and $\K_{r,s}$ such that $\langle\psi^*,\phi_1\otimes \phi_2\rangle\neq 0$.
For $n\geq 1$, let $\{w^{(n)}_i\}_{i=-n}^n$ be the Virasoro highest weight vectors of the vector subspace $(2n+1)L(\Delta^+_{r,s;n})\subset \K_{r,s}$. 

Let us consider the values
\begin{align}
\label{ax}
\langle\psi^*,U(\mathcal{L}).\ket{\alpha_{1,2}}{\otimes}x\rangle,
\end{align}
where $x=\ket{\alpha_{1,2}}$ or $w^{(n)}_i$.
Note that, from Proposition \ref{FockSocle}, the ground state $\ket{\alpha_{1,2}}$ satisfies
\begin{align}
\label{relsing}
\bigl( L^2_{-1}-\frac{p_+}{p_-}L_{-2} \bigr)\ket{\alpha_{1,2}}=0.
\end{align}
Noting (\ref{relsing}) and applying the formulas in Lemma \ref{NGK01} as $A=T$, we see that the value (\ref{ax}) is determined by the numbers
\begin{align*}
&\langle \psi^*,\ket{\alpha_{1,2}}\otimes x\rangle,
&\langle \psi^*,(L_{-1}\ket{\alpha_{1,2}})\otimes x\rangle.
\end{align*}
Let $h(x)$ be the $L_0$ weight of $x$. Using the formulas in Lemma \ref{NGK01}, we have
\begin{align*}
\langle L_0\psi^*,\ket{\alpha_{1,2}}\otimes x\rangle=(h_{1,2}+h(x))\langle \psi^*,\ket{\alpha_{1,2}}\otimes x\rangle+\langle \psi^*,(L_{-1}\ket{\alpha_{1,2}})\otimes x\rangle
\end{align*}
and 
\begin{align*}
\langle L_0\psi^*,(L_{-1}\ket{\alpha_{1,2}})\otimes x\rangle&=(h_{1,2}+h(x)+1)\langle \psi^*,(L_{-1}\ket{\alpha_{1,2}})\otimes x\rangle\\
&\ \ +\frac{p_+}{p_-}\langle \psi^*,(L_{-2}\ket{\alpha_{1,2}})\otimes x\rangle\\
&=\frac{p_+}{p_-}h(x)\langle \psi^*,\ket{\alpha_{1,2}}\otimes x\rangle\\
&\ \ +(h_{1,2}+h(x)+1-\frac{p_+}{p_-})\langle \psi^*,(L_{-1}\ket{\alpha_{1,2}})\otimes x\rangle,
\end{align*}
Consequently, we have
\begin{equation*}
\begin{pmatrix}
\langle L_0\psi^*,\ket{\alpha_{1,2}}\otimes x\rangle  \\
\langle L_0\psi^*,(L_{-1}\ket{\alpha_{1,2}})\otimes x\rangle \\
\end{pmatrix}
=N(x)
\begin{pmatrix}
\langle \psi^*,\ket{\alpha_{1,2}}\otimes x\rangle  \\
\langle \psi^*,(L_{-1}\ket{\alpha_{1,2}})\otimes x\rangle \\
\end{pmatrix}
,
\end{equation*}
where
\begin{equation*}
N(x)=
\begin{pmatrix}
 h_{1,2}+h(x) & \frac{p_+}{p_-}h(x)\\
1 & h_{1,2}+h(x)+1-\frac{p_+}{p_-} 
\end{pmatrix}
.
\end{equation*}
We see that $N(\ket{\alpha_{1,2}})$ and $N(w^{(n)}_i)$ are diagonalizable and the eigenvalues are given by $\{h_{r,s-1},h_{r,s-1}\}$ and $\{\Delta^+_{r,s+1;n},\Delta^+_{r,s-1;n}\}$, respectively. 
Note that the eigenvalues of $N(w^{(n)}_i)$ do not correspond to any $L_0$ eigenvalues of the highest weight space of the simple $\W_{p_+,p_-}$-modules (see Theorem \ref{simpleclass}). Thus we have
\begin{align}
\label{noteW}
\langle\psi^*,U(\mathcal{L}).\ket{\alpha_{1,2}}\otimes w^{(n)}_i\rangle=0
\end{align}
for any $n\geq 1$ and $i$.
Noting (\ref{noteW}), from Proposition \ref{genW}, Proposition \ref{sl2action2} and the formulas in Lemma \ref{NGK01}, we see that $\langle\psi^*,\phi_1\otimes \phi_2\rangle$ is determined by the numbers
\begin{align*}
&\langle \psi^*,\ket{\alpha_{1,2}}\otimes\ket{\alpha_{r,s}}\rangle,
&\langle \psi^*,(L_{-1}\ket{\alpha_{1,2}})\otimes\ket{\alpha_{r,s}}\rangle.
\end{align*}
Therefore the $L_0$ eigenvalue of $\psi^*$ is given by $h_{r,s+1}$ or $h_{r,s-1}$ and
\begin{align*}
{\rm dim}_{\mathbb{C}}A_0((\K_{1,2}\boxtimes\K_{r,s})^*)\leq 2.
\end{align*}

Next let us prove the second case. Assume $\K_{1,2}\boxtimes\K_{r,1}\neq 0$. Let $\chi^*$, $\upsilon_1$ and $\upsilon_2$ be arbitrary vectors of $A_0((\K_{1,2}\boxtimes\K_{r,1})^*)$, $\K_{1,2}$ and $\K_{r,1}$ such that $\langle\chi^*,\upsilon_1\otimes \upsilon_2\rangle\neq 0$. 
Similar to the first case, we see that $\langle\chi^*,\upsilon_1\otimes \upsilon_2\rangle$ is determined by the numbers
\begin{align*}
&\langle \chi^*,\ket{\alpha_{1,2}}\otimes\ket{\alpha_{r,1}}\rangle,
&\langle \chi^*,(L_{-1}\ket{\alpha_{1,2}})\otimes\ket{\alpha_{r,1}}\rangle,
\end{align*}
and we have
\begin{equation*}
\begin{pmatrix}
\langle L_0\chi^*,\ket{\alpha_{1,2}}\otimes \ket{\alpha_{r,1}}\rangle  \\
\langle L_0\chi^*,(L_{-1}\ket{\alpha_{1,2}})\otimes \ket{\alpha_{r,1}}\rangle \\
\end{pmatrix}
=N
\begin{pmatrix}
\langle \chi^*,\ket{\alpha_{1,2}}\otimes \ket{\alpha_{r,1}}\rangle  \\
\langle \chi^*,(L_{-1}\ket{\alpha_{1,2}})\otimes \ket{\alpha_{r,1}}\rangle \\
\end{pmatrix}
,
\end{equation*}
where
\begin{equation*}
N=
\begin{pmatrix}
 h_{1,2}+h_{r,1} & \frac{p_+}{p_-}h_{r,1}\\
1 & h_{1,2}+h_{r,1}+1-\frac{p_+}{p_-} 
\end{pmatrix}
.
\end{equation*}
We see that $N$ is diagonalizable and the eigenvalues are given by $\{h_{r,2},h_{r^\vee,p_-}\}$. In particular $L_0$ acts semisimply on $\chi^*$.

Assume that the $L_0$ weight of $\chi^*$ is $h_{r^\vee,p_-}$. Since $S_{r,1}\ket{\alpha_{r,1}}=0$ (see Proposition \ref{FockSocle}), applying Proposition \ref{FusionInt} to the three point function $\langle \chi^*,\ket{\alpha_{1,2}}\otimes\ket{\alpha_{r,1}}\rangle$, we obtain
\begin{align*}
\prod_{i=1}^{r}(h_{1,2}-h_{r^\vee+r-2i+1,p_-})=0.
\end{align*}
But since $p_->2$, we have a contradiction. Thus the $L_0$ weight of $\chi^*$ is $h_{r,2}$. Note that 
$
\footnotesize{\begin{pmatrix}
     \frac{2}{1-r} \\
     1
  \end{pmatrix}}
$ is the eigenvector of ${}^tN$ with the eigenvalue $h_{r^\vee,p_-}$, and
\begin{align*}
&\frac{2}{1-r}\langle L_0\chi^*,\ket{\alpha_{1,2}}\otimes \ket{\alpha_{r,1}}\rangle+\langle L_0\chi^*,(L_{-1}\ket{\alpha_{1,2}})\otimes \ket{\alpha_{r,1}}\rangle\\
&=\bigl(\frac{2}{1-r},1\bigr)\begin{pmatrix}
 h_{1,2}+h_{r,1} & \frac{p_+}{p_-}h_{r,1}\\
1 & h_{1,2}+h_{r,1}+1-\frac{p_+}{p_-} 
\end{pmatrix}
\begin{pmatrix}
     \langle \chi^*,\ket{\alpha_{1,2}}\otimes \ket{\alpha_{r,1}} \\
     \langle \chi^*,(L_{-1}\ket{\alpha_{1,2}})\otimes \ket{\alpha_{r,1}}\rangle
  \end{pmatrix}\\
&=h_{r^\vee,p_-}\bigl(\frac{2}{1-r},1\bigr)
\begin{pmatrix}
     \langle \chi^*,\ket{\alpha_{1,2}}\otimes \ket{\alpha_{r,1}} \\
     \langle \chi^*,(L_{-1}\ket{\alpha_{1,2}})\otimes \ket{\alpha_{r,1}}\rangle
  \end{pmatrix}
.  
\end{align*}
Thus we have
\begin{align}
\label{Fromone}
\langle \chi^*,\ket{\alpha_{1,2}}\otimes \ket{\alpha_{r,1}}\rangle+\frac{1-r}{2}\langle \chi^*,(L_{-1}\ket{\alpha_{1,2}})\otimes \ket{\alpha_{r,1}}\rangle=0.
\end{align}
From (\ref{Fromone}), we see that $A_0((\K_{1,2}\boxtimes\K_{r,1})^*)$ is one dimensional.
\end{proof}
\end{prop}

Similar to Proposition \ref{NGK2}, we obtain the following proposition.
\begin{prop}
\label{NGK22}
Assume that
\begin{align*}
&\mathcal{K}_{2,1}\boxtimes \mathcal{K}_{r,s}\neq 0,
&\mathcal{K}_{2,1}\boxtimes \mathcal{X}^\pm_{r,s}\neq 0
\end{align*}
for all $r,s\geq 1$.
Then the following holds.
\begin{itemize}
\item For $p_+\geq 3$, we have
\begin{enumerate}
\item For $2\leq r< p_+$, $1\leq s\leq p_-$, 
the vector space $A_0((\K_{2,1}\boxtimes\K_{r,s})^*)$ is at most two dimensional. $L_0$ acts semisimply on this vector space and the $L_0$ eigenvalues are contained in 
\begin{align*}
\{H(\K_{r-1,s}),H(\K_{r+1,s})\}=\{h_{r-1,s},h_{r+1,s}\}.
\end{align*}
\item For $2\leq s\leq p_-$, 
the dimension of the vector space $A_0((\K_{2,1}\boxtimes\K_{1,s})^*)$ is at most one dimensional. $L_0$ acts semisimply on this vector space and the $L_0$ eigenvalue is given by $H(\K_{2,s})=h_{2,s}$.

\item For $2\leq r< p_+$, $1\leq s\leq p_-$, 
the vector space $A_0((\K_{2,1}\boxtimes\X^+_{r,s})^*)$ is at most two dimensional. $L_0$ acts semisimply on this vector space and the $L_0$ eigenvalues are contained in 
\begin{align*}
\{H(\X^+_{r-1,s}),H(\X^+_{r+1,s})\}=\{\Delta^+_{r-1,s;0},\Delta^+_{r+1,s;0}\}.
\end{align*}
\item For $2\leq r< p_+$, $1\leq s\leq p_-$, the dimension of the vector space $A_0((\K_{2,1}\boxtimes\X^-_{r,s})^*)$ is at most four dimensional. $L_0$ acts semisimply on this vector space and the $L_0$ eigenvalues are contained in 
\begin{align*}
\{H(\X^-_{r-1,s}),H(\X^-_{r+1,s})\}=\{\Delta^-_{r-1,s;0},\Delta^-_{r+1,s;0}\}.
\end{align*}
\item For $r=p_+$, $1\leq s\leq p_-$, the vector space $A_0((\K_{2,1}\boxtimes\X^+_{p_+,s})^*)$ is at most two dimensional. The generalized $L_0$ eigenvalues of this vector space are contained in 
\begin{align*}
\{H(\mathcal{K}_{p_+-1,s}),H(\X^+_{p_+-1,s})\}=\{h_{p_+-1,s},\Delta^+_{p_+-1,s;0}\}.
\end{align*}
\item For $r=p_+$, $1\leq s\leq p_-$, the vector space $A_0((\K_{2,1}\boxtimes\X^-_{p_+,s})^*)$ is at most four dimensional. $L_0$ acts semisimply on this vector space and the $L_0$ eigenvalues of this vector space are contained in 
\begin{align*}
\{H(\X^+_{1,s}),H(\X^-_{p_+-1,s})\}=\{\Delta^+_{1,s;0},\Delta^-_{p_+-1,s;0}\}.
\end{align*}
\end{enumerate}
\item For $p_+=2$, we have
\begin{enumerate}
\item For $1\leq s\leq p_-$, the vector space $A_0((\K_{2,1}\boxtimes\X^+_{2,s})^*)$ is at most two dimensional. The generalized $L_0$ eigenvalues of this vector space are contained in 
\begin{align*}
\{H(\mathcal{K}_{1,s}),H(\X^+_{1,s})\}=\{h_{1,s},\Delta^+_{1,s;0}\}.
\end{align*}
\item For $1\leq s\leq p_-$, the vector space $A_0((\K_{2,1}\boxtimes\X^-_{2,s})^*)$ is at most four dimensional. $L_0$ acts semisimply on this vector space and the $L_0$ eigenvalues are contained in 
\begin{align*}
\{H(\X^+_{1,s}),H(\X^-_{1,s})\}=\{\Delta^+_{1,s;0},\Delta^-_{1,s;0}\}.
\end{align*}
\end{enumerate}
\end{itemize}
\end{prop}
We set 
$
A_{p_+,p_-}=\{\ \alpha_{r,s;n}\ |\ r,s,n\in\Z \},
$
where the symbol $\alpha_{r,s;n}$ is defined by (\ref{alpha_{r,s;n}}).
For any $\alpha\in A_{p_+,p_-}$, let 
\begin{align*}
\mathbb{V}_{\alpha}=\bigoplus_{n\in\Z}F_{\alpha+n\sqrt{2p_+p_-}}
\end{align*}
be the simple $\mathcal{V}_{[p_+,p_-]}$-module.
For any $\alpha,\alpha'\in A_{p_+,p_-}$, it can be proved easily that there are no $\mathcal{V}_{[p_+,p_-]}$-module intertwining operators of type 
{\footnotesize{$\begin{pmatrix}
\ \mathbb{V}_{\alpha''} \\
\mathbb{V}_{\alpha'}\ \ \mathbb{V}_{\alpha}
\end{pmatrix}$}}
unless $\alpha''\equiv\alpha'+\alpha\ {\rm mod}\ \Z\sqrt{2p_+p_-}$, and {\footnotesize{${\rm dim}_{\C}I\begin{pmatrix}
\ \mathbb{V}_{\alpha'+\alpha} \\
\mathbb{V}_{\alpha'}\ \ \mathbb{V}_{\alpha}
\end{pmatrix}=1$}}.
\begin{dfn}
\label{dfnY}
For $\alpha,\alpha'\in A_{p_+,p_-}$, let $Y$ be the non-zero $\mathcal{V}_{[p_+,p_-]}$-module intertwining operator of type {\footnotesize{$\begin{pmatrix}
\ \mathbb{V}_{\alpha'+\alpha} \\
\mathbb{V}_{\alpha'}\ \ \mathbb{V}_{\alpha}
\end{pmatrix}$}}. Then, by restricting the action of $\mathcal{V}_{[p_+,p_-]}$ to $\W_{p_+,p_-}$, $Y$ defines a $\W_{p_+,p_-}$-module intertwining operator of type {\footnotesize{$\begin{pmatrix}
\ \mathbb{V}_{\alpha'+\alpha} \\
\mathbb{V}_{\alpha'}\ \ \mathbb{V}_{\alpha}
\end{pmatrix}$}}.
We denote by $Y_{\alpha',\alpha}$ this $\W_{p_+,p_-}$-module intertwining operator.
\end{dfn}

\begin{lem}
\label{Intertwiningori}
For $1\leq r\leq p_+$, $2\leq s\leq p_+-1$, we have 
\begin{equation*}
I
\begin{pmatrix}
\ \X^+_{r,s-1} \\
\K_{1,2}\ \ \X^+_{r,s}
\end{pmatrix}
\neq 0
,\ \ \ \ \ \ 
I
\begin{pmatrix}
\ \X^+_{r,s+1} \\
\K_{1,2}\ \ \X^+_{r,s}
\end{pmatrix}
\neq 0
.
\end{equation*}
\begin{proof}
We only prove the first inequality. 
The second case can be proved in the same way.

Let us consider the $\W_{p_+,p_-}$-module intertwining operator $Y=Y_{\alpha_1,\alpha_2}$ given in Definition \ref{dfnY}, where $\alpha_1=\alpha_{1,2}$ and $\alpha_2=\alpha_{r,s^\vee;1}$. Then we have
\begin{align*}
\bra{\alpha_{r,s^\vee+1;1}}Y(\ket{\alpha_{1,2}},z)\ket{\alpha_{r,s^\vee;1}}\neq 0.
\end{align*}
Thus we have 
\begin{equation}
I
\begin{pmatrix}
\ \X^+_{r,s-1} \\
\K_{1,2}\ \ \W_{p_+,p_-}.\ket{\alpha_{r,s^\vee;1}}
\end{pmatrix}
\neq 0.
\label{Int202212130}
\end{equation}
Note that, from Proposition \ref{socleV}, $\W_{p_+,p_-}.\ket{\alpha_{r,s^\vee;1}}$ satisfies the exact sequence
\begin{align*}
0\rightarrow \X^-_{r,s^\vee}\rightarrow \W_{p_+,p_-}.\ket{\alpha_{r,s^\vee;1}}\rightarrow \X^+_{r,s}\rightarrow 0.
\end{align*}
Then, by the exact sequence
\begin{align*}
\K_{1,2}\boxtimes\X^-_{r,s^\vee}\rightarrow \K_{1,2}\boxtimes\W_{p_+,p_-}.\ket{\alpha_{r,s^\vee;1}}\rightarrow \K_{1,2}\boxtimes\X^+_{r,s}\rightarrow 0,
\end{align*}
we have the following exact sequence
\begin{align}
0&\rightarrow{\rm Hom}_{\mathcal{C}_{p_+,p_-}}(\K_{1,2}\boxtimes\X^+_{r,s},\X^+_{r,s-1})\rightarrow {\rm Hom}_{\mathcal{C}_{p_+,p_-}}(\K_{1,2}\boxtimes\W_{p_+,p_-}.\ket{\alpha_{r,s^\vee;1}},\X^+_{r,s-1})\nonumber\\
&\rightarrow {\rm Hom}_{\mathcal{C}_{p_+,p_-}}(\K_{1,2}\boxtimes\X^-_{r,s^\vee},\X^+_{r,s-1}).
\label{exact202212130}
\end{align}
By Proposition \ref{NGK2}, we have $ {\rm Hom}_{\mathcal{C}_{p_+,p_-}}(\K_{1,2}\boxtimes\X^-_{r,s^\vee},\X^+_{r,s-1})=0$.
Thus, by (\ref{Int202212130}) and (\ref{exact202212130}), we see that $\K_{1,2}\boxtimes\X^+_{r,s}$ has a quotient isomorphic to $\X^+_{r,s-1}$. Thus we obtain the first inequality.
\end{proof}
\end{lem}

\begin{lem}
\label{ichiyou1}
For $1\leq r\leq p_+$, $1\leq s\leq p_+-1$, we have 
\begin{equation*}
I
\begin{pmatrix}
\ \X^+_{r,s} \\
\K_{r,s}\ \X^+_{1,1}
\end{pmatrix}
\neq 0
.
\end{equation*}
\begin{proof}
By the exact sequence
\begin{align*}
0\rightarrow \X^+_{1,1}\rightarrow \K_{1,1}\rightarrow L(h_{1,1})\rightarrow 0
\end{align*}
and by Corollary \ref{202210180}, we have the following exact sequence
\begin{align*}
\K_{r,s}\boxtimes\X^+_{1,1}\rightarrow \K_{r,s}\rightarrow L(h_{r,s})\rightarrow 0.
\end{align*}
Thus, by this exact sequence, we obtain the claim of the lemma.
\end{proof}
\end{lem}

For the following lemma, see \cite[Proposition 6.29, Proposition 6.30]{Nakano}.
\begin{lem}[\cite{Nakano}]
\label{Extcoro1}
For $1\leq r<p_+$, $1\leq s<p_-$, we have
\begin{align*}
{\rm Ext}^1(\K_{r,s},L(h_{r,s}))={\rm Ext}^1(\K_{r,s},\X^-_{r^\vee,s})={\rm Ext}^1(\K_{r,s},\X^-_{r,s^\vee})=0.
\end{align*}
\end{lem}
\begin{prop}
\label{Fusion12W}
\mbox{}
\begin{enumerate}
\item We have
\begin{align*}
\K_{1,2}\boxtimes\K_{r,s}=\K_{r,s-1}\oplus\Gamma(\K_{r,s+1}),
\end{align*}
where ${\Gamma}(\K_{r,s+1})$ is a highest weight module with ${\rm top}(\Gamma(\K_{r,s+1}))=\X^+_{r,p_-}$.
\item For $p_+\geq 3$, $2\leq r\leq p_+-1$, $1\leq s\leq p_--1$, we have
\begin{align*}
\K_{2,1}\boxtimes\K_{r,s}=\K_{r-1,s}\oplus\Gamma(\K_{r+1,s}),
\end{align*}
where $\Gamma(\K_{r+1,s})$ is a highest weight module with ${\rm top}(\Gamma(\K_{r+1,s}))=\X^+_{p_+,s}$. 
\end{enumerate}
\begin{proof}

We will prove 
\begin{align}
\label{K0826}
\K_{1,2}\boxtimes\K_{r,s}=\K_{r,s-1}\oplus\K_{r,s+1}
\end{align}
in the case of $p_-\geq 4$, $1\leq r\leq p_+-1$ and $2\leq s\leq p_--2$. 
The other cases can be proved in a similar way.

By the exact sequence
\begin{align*}
0\rightarrow \X^+_{r,s}\rightarrow \K_{r,s}\rightarrow L(h_{r,s})\rightarrow 0,
\end{align*}
we have the following exact sequence
\begin{align}
\label{toro}
\K_{1,2}\boxtimes\X^+_{r,s}\xrightarrow{\eta} \K_{1,2}\boxtimes\K_{r,s}\rightarrow L(h_{r,s-1})\oplus L(h_{r,s+1})\rightarrow 0,
\end{align}
and by Lemma \ref{ichiyou1}, we have the following exact sequence
\begin{align}
\label{toro2}
(\K_{1,2}\boxtimes\K_{r,s})\boxtimes\X^+_{1,1}\rightarrow \K_{1,2}\boxtimes\X^+_{r,s}\rightarrow 0.
\end{align}
Then, by two exact sequences (\ref{toro}),(\ref{toro2}) and by Lemma \ref{Intertwiningori}, we see that $\K_{1,2}\boxtimes\K_{r,s}$ has $\X^+_{r,s-1}$ and $\X^+_{r,s+1}$ as composition factors. In particular, we have $\eta(\K_{1,2}\boxtimes \mathcal{X}^+_{r,s})\neq 0$, where $\eta$ is defined by (\ref{toro}).
Note that from Proposition \ref{NGK2} we have
\begin{equation}
\label{top2022102400}
\begin{split}
{\rm top}(\K_{1,2}\boxtimes \K_{r,s})&=L(h_{r,s-1})\oplus L(h_{r,s+1}),\\
{\rm top}(\K_{1,2}\boxtimes \mathcal{X}^+_{r,s})&={\rm top}(\eta(\K_{1,2}\boxtimes \mathcal{X}^+_{r,s}))=\X^+_{r,s-1}\oplus \X^+_{r,s+1}.
\end{split}
\end{equation}
Thus, by Lemma \ref{Extcoro1}, 
to prove (\ref{K0826}),
it is sufficient to show that $\K_{1,2}\boxtimes\K_{r,s}$ does not contain any $\X^+_{r^\vee,s^\vee+1}$ and $\X^+_{r^\vee,s^\vee-1}$ as composition factors.

Assume that $\K_{1,2}\boxtimes\K_{r,s}$ contains $\X=\X^+_{r^\vee,s^\vee+1}\ {\rm or}\ \X^+_{r^\vee,s^\vee-1}$ as a composition factor. Then, we see that the composition factor $\X$ of $\K_{1,2}\boxtimes\K_{r,s}$ is contained in the submodule $\eta(\K_{1,2}\boxtimes \mathcal{X}^+_{r,s})\subset \K_{1,2}\boxtimes \K_{r,s}$.
Note that from Propositions \ref{Ext} and \ref{ExtQ}, we have 
\begin{equation}
\label{ntt0816}
\begin{split}
&{\rm Ext}^1(\X^+_{r,s-1}\oplus \X^+_{r,s+1},\X^+_{r^\vee,s^\vee+1}\oplus \X^+_{r^\vee,s^\vee-1})=0,\\
&{\rm Ext}^1(\mathcal{K}^*_{r,s-1}\oplus \mathcal{K}^*_{r,s+1},\X^+_{r^\vee,s^\vee+1}\oplus \X^+_{r^\vee,s^\vee-1})=0.
\end{split}
\end{equation}
From (\ref{ntt0816}), we see that $\eta(\K_{1,2}\boxtimes \mathcal{X}^+_{r,s})$ has an indecomposable subquotient in ${\rm Ext}^1(\X^-_{r^\vee,s-1},\mathcal{X})$ or ${\rm Ext}^1(\X^-_{r,s^\vee+1},\mathcal{X})$.
Thus, noting (\ref{top2022102400}),
as a quotient of $\K_{1,2}\boxtimes\K_{r,s}$ we have a non-trivial extension in ${\rm Ext}^1(\K_{r,s-1},\X^-_{r^\vee,s-1})$ or ${\rm Ext}^1(\K_{r,s-1},\X^-_{r,s^\vee+1})$. But this contradicts Lemma \ref{Extcoro1}. 

\end{proof}
\end{prop}

\begin{thm}
\label{rigid12ori}
$\K_{1,2}$ is rigid and self-dual in $(\mathcal{C}_{p_+,p_-},\boxtimes,\mathcal{K}_{1,1})$.
\begin{proof}
We show the rigidity of $\K_{1,2}$ using the methods detailed in \cite{CMY,MS,McRae,TWFusion}. By Proposition \ref{Fusion12W}, we have homomorphisms
\begin{align*}
&i_{1}:\K_{1,1}\rightarrow \K_{1,2}\boxtimes\K_{1,2},
&p_1&:\K_{1,2}\boxtimes\K_{1,2}\rightarrow\K_{1,1},\\
&i_{3}:\Gamma(\K_{1,3})\rightarrow \K_{1,2}\boxtimes\K_{1,2},
&p_3&:\K_{1,2}\boxtimes\K_{1,2}\rightarrow \Gamma(\K_{1,3})
\end{align*}
such that
\begin{align*}
p_1\circ i_1={\rm id}_{\K_{1,1}},\ \ \ \ \ p_3\circ i_3={\rm id}_{\Gamma(\K_{1,3})},&&
i_1\circ p_1+i_3\circ p_3={\rm id}_{\K_{1,2}\boxtimes\K_{1,2}},
\end{align*}
where $\Gamma(\K_{1,3})$ is the indecomposable module defined by Proposition \ref{Fusion12W}.

To prove that $\K_{1,2}$ is rigid and self-dual, it is sufficient to prove that the homomorphisms $f,g:\K_{1,2}\rightarrow \K_{1,2}$ defined by the commutative diagrams
\begin{equation*}
 \xymatrix{
    \K_{1,2} \ar[r]^-{r^{-1}} \ar[d]_{f} & \K_{1,2}\boxtimes\K_{1,1} \ar[r]^-{{\rm id}\boxtimes i_1}& \K_{1,2}\boxtimes(\K_{1,2}\boxtimes\K_{1,2})\ar[d]^{\mathcal{A}}& \\
    \K_{1,2} & \K_{1,1}\boxtimes\K_{1,2}\ar[l]_-l&(\K_{1,2}\boxtimes\K_{1,2})\boxtimes\K_{1,2}\ar[l]_-{p_1\boxtimes{\rm id}}
  }
\end{equation*}
and
\begin{equation*}
 \xymatrix{
    \K_{1,2} \ar[r]^-{l^{-1}} \ar[d]_{g} & \K_{1,1}\boxtimes\K_{1,2} \ar[r]^-{i_1\boxtimes{\rm id}}& (\K_{1,2}\boxtimes\K_{1,2})\boxtimes\K_{1,2}\ar[d]^{\mathcal{A}^{-1}}& \\
    \K_{1,2} & \K_{1,2}\boxtimes\K_{1,1}\ar[l]_-r&\K_{1,2}\boxtimes(\K_{1,2}\boxtimes\K_{1,2})\ar[l]_-{{\rm id}\boxtimes p_1}
  }
\end{equation*}
are non-zero multiples of the identity and $f=g$, where $\mathcal{A}$ is the associativity isomorphism and $l$ and $r$ are the left and right unit isomorphisms. Since ${\rm Hom}(\K_{1,2},\K_{1,2})\simeq \C$, it is sufficient to show that $f=g\neq 0$. 


Let $\mathcal{Y}_{2\boxtimes {2}}$, $\mathcal{Y}_{(2\boxtimes 2)\boxtimes {2}}$ and $\mathcal{Y}_{2\boxtimes (2\boxtimes 2)}$ be the non-zero intertwining operators of type
\begin{equation*}
\begin{pmatrix}
\ \K_{1,2}\boxtimes\K_{1,2} \\
\K_{1,2}\ \ \K_{1,2}
\end{pmatrix}
,\ \ \ 
\begin{pmatrix}
\ (\K_{1,2}\boxtimes\K_{1,2})\boxtimes \K_{1,2} \\
\K_{1,2}\boxtimes\K_{1,2}\ \ \K_{1,2}
\end{pmatrix}
,\ \ \ \begin{pmatrix}
\ \K_{1,2}\boxtimes(\K_{1,2}\boxtimes \K_{1,2}) \\
\K_{1,2}\ \ \K_{1,2}\boxtimes \K_{1,2}
\end{pmatrix}
,
\end{equation*}
respectively. 

We introduce the intertwining operators
\begin{align*}
&\mathcal{Y}^2_{21}=l_{\K_{1,2}}\circ (p_1\boxtimes {\rm id}_{\K_{1,2}})\circ \mathcal{A}_{\K_{1,2},\K_{1,2},\K_{1,2}}\circ \mathcal{Y}_{2\boxtimes (2\boxtimes 2)}\circ ({\rm id}_{\K_{1,2}}\otimes i_1),\\
&\mathcal{Y}^2_{12}=r_{\K_{1,2}}\circ ({\rm id}_{\K_{1,2}}\boxtimes p_1)\circ \mathcal{A}^{-1}_{\K_{1,2},\K_{1,2},\K_{1,2}}\circ \mathcal{Y}_{(2\boxtimes 2)\boxtimes 2}\circ (i_1\otimes {\rm id}_{\K_{1,2}}) 
\end{align*} 
which correspond to $f$ and $g$, respectively.
To prove $f=g\neq 0$, it is sufficient to show that 
\begin{align}
\label{show00}
\langle v^*, \mathcal{Y}^2_{21}(v,1)\ket{0}\rangle=\langle v^*, \mathcal{Y}^2_{12}(\ket{0},1)v\rangle\neq 0.
\end{align}

Define the following intertwining operators
\begin{align*}
&\mathcal{Y}^2_{23}=l_{\K_{1,2}}\circ (p_3\boxtimes {\rm id}_{\K_{1,2}})\circ \mathcal{A}_{\K_{1,2},\K_{1,2},\K_{1,2}}\circ \mathcal{Y}_{2\boxtimes (2\boxtimes 2)}\circ ({\rm id}_{\K_{1,2}}\otimes i_3),\\
&\mathcal{Y}^2_{32}=r_{\K_{1,2}}\circ ({\rm id}_{\K_{1,2}}\boxtimes p_3)\circ \mathcal{A}^{-1}_{\K_{1,2},\K_{1,2},\K_{1,2}}\circ \mathcal{Y}_{(2\boxtimes 2)\boxtimes 2}\circ (i_3\otimes {\rm id}_{\K_{1,2}}).
\end{align*}
Then, for highest weight vectors $v\in\K_{1,2}[h_{1,2}]$, $v^*\in \K_{1,2}^*[h_{1,2}]$, and for some $x\in\mathbb{R}$ such that $1>x>1-x>0$, we have
\begin{equation}
\label{F}
\begin{split}
&\langle v^*,\mathcal{Y}^2_{21}(v,1)(p_1\circ\mathcal{Y}_{2\boxtimes{2}})(v,x)v\rangle+\langle v^*,\mathcal{Y}^2_{23}(v,1)(p_3\circ\mathcal{Y}_{2\boxtimes{2}})(v,x)v\rangle\\
&=\langle v^*, \overline{l_{\K_{1,2}}\circ(p_1\boxtimes {\rm id}_{\K_{1,2}})\circ \mathcal{A}_{\K_{1,2},\K_{1,2},\K_{1,2}} }\bigl(\mathcal{Y}_{2\boxtimes(2\boxtimes {2})}(v,1)\mathcal{Y}_{2\boxtimes {2}}(v,x)v\bigr)\rangle\\
&=\langle v^*,  \overline{l_{\K_{1,2}}\circ(p_1\boxtimes {\rm id}_{\K_{1,2}}) } \bigl(\mathcal{Y}_{(2\boxtimes 2)\boxtimes{2}}(\mathcal{Y}_{2\boxtimes 2}(v,1-x)v,x)v\bigr)\rangle\\
&=\langle v^*, \overline{l_{\K_{1,2}}}\bigl(\mathcal{Y}_{1\boxtimes {2}}((p_1\circ\mathcal{Y}_{2\boxtimes 2})(v,1-x)v,x)v\bigr) \rangle,\\
&=\langle v^*, Y_{\K_{1,2}}\bigl((p_1\circ\mathcal{Y}_{2\boxtimes 2})(v,1-x)v,x\bigr)v\rangle,
\end{split}
\end{equation}
where $\mathcal{Y}_{1\boxtimes {2}}$ is the non-zero intertwining operator of type
{\scriptsize{$\begin{pmatrix}
   \ \K_{1,2}  \\
   \K_{1,1}\ \K_{1,2}
\end{pmatrix}
$}}.
Similarly we can show
\begin{equation}
\label{G}
\begin{split}
&\langle v^*,\mathcal{Y}^2_{12}\bigl((p_1\circ\mathcal{Y}_{2\boxtimes{2}})(v,x)v,1-x\bigr)v\rangle+\langle v^*,\mathcal{Y}^2_{32}\bigl((p_3\circ\mathcal{Y}_{2\boxtimes{2}})(v,x)v,1-x\bigr)v\rangle\\
&=\langle v^*, \Omega(Y_{\K_{1,2}})(v,1)(p_1\circ\mathcal{Y}_{2\boxtimes 2})(v,1-x)v\rangle,
\end{split}
\end{equation}
where $\Omega$ represents the skew-symmetry operation on vertex operators defined by
\begin{align*}
\Omega(Y_{\K_{1,2}})(v,z)w=e^{zL_{-1}}Y_{\K_{1,2}}(w,-z)v
\end{align*}
for $w\in \K_{1,1}$. Note that the four point functions (\ref{F}) and (\ref{G}) are non-zero.

From the relation (\ref{relsing}), as in \cite{BPZ,CMY,McRae,TWFusion}, we see that the four point functions (\ref{F}) and (\ref{G}) satisfy the following Fuchsian differential equation
\begin{equation}
\label{solF}
\begin{split}
\phi''(x)&+\frac{p_+}{p_-}\Bigl(\frac{1}{x-1}+\frac{1}{x}\Bigr)\phi'(x)\\
&-\frac{p_+h_{1,2}}{p_-}\Bigl\{ \frac{1}{(x-1)^2}+\frac{1}{x^2}-2\Bigl(\frac{1}{x-1}-\frac{1}{x}\Bigr)\Bigr\}\phi(x)=0
\end{split}
\end{equation}
with the Riemann scheme
\begin{equation*}
\begin{split}
\begin{bmatrix}
0 & 1 & \infty \\
\lambda_+ & \mu_+ &\nu_+\\
\lambda_- & \mu_-&\nu_-
\end{bmatrix}
=
\begin{bmatrix}
0 & 1 & \infty \\
\frac{p_+}{2p_-} &  \frac{p_+}{2p_-} &0\\
1-\frac{3p_+}{2p_-} & 1-\frac{3p_+}{2p_-}&\frac{2p_+}{p_-}-1
\end{bmatrix}
,
\end{split}
\end{equation*}
where
\begin{align*}
&h_{1,1}-2h_{1,2}=1-\frac{3p_+}{2p_-},
&h_{1,3}-2h_{1,2}=\frac{p_+}{2p_-}.
\end{align*} 
Note that
\begin{align}
\label{nt0715}
\langle v^*, Y_{\K_{1,2}}(\ket{0},1)v\rangle=\langle v^*, \Omega(Y_{\K_{1,2}})(v,1)\ket{0}\rangle,
\end{align}
$p_1\circ\mathcal{Y}_{2\boxtimes 2}$ is a non-logarithmic intertwining operator of type
{\scriptsize{$\begin{pmatrix}
   \ \mathcal{K}_{1,1}  \\
   \K_{2,1}\ \ \K_{2,1}
\end{pmatrix}
$}} and
\begin{align}
\label{nt0812}
\langle v^*_{\mathcal{K}_{1,1}},(p_1\circ\mathcal{Y}_{2\boxtimes 2})(v,x)v\rangle\neq 0,
\end{align}
where $v^*_{\mathcal{K}_{1,1}}$ is the highest weight vector of $\mathcal{K}^*_{1,1}$.
From (\ref{nt0715}) and (\ref{nt0812}), we obtain
\begin{equation}
\label{fourpointori}
\langle v^*, Y_{\K_{1,2}}\bigl((p_1\circ\mathcal{Y}_{2\boxtimes 2})(v,1-x)v,x)\bigr)v\rangle=
\langle v^*, \Omega(Y_{\K_{1,2}})(v,1)(p_1\circ\mathcal{Y}_{2\boxtimes 2})(v,1-x)v\rangle.
\end{equation}

Let $u_+$ and $u_-$ be the fundamental solutions of (\ref{solF}) near $x=0$ whose characteristic exponents are $\lambda_+,\lambda_-$, respectively, and $v_+$ and $v_-$ the fundamental solutions of (\ref{solF}) near $x=1$ whose characteristic exponents are $\mu_+,\mu_-$, respectively. These two bases of solutions are related by the connection matrix:
\begin{equation*}
\begin{pmatrix}
u_+  \\
u_- \\
\end{pmatrix}
=
\begin{pmatrix}
F_{++} & F_{-+} \\
F_{+-} & F_{--} \\
\end{pmatrix}
\begin{pmatrix}
v_+  \\
v_- \\
\end{pmatrix}
,
\end{equation*}
where 
\begin{align*}
F_{\epsilon\epsilon'}=\frac{\Gamma(-\epsilon(\mu_+-\mu_-))\Gamma(\epsilon'(\lambda_+-\lambda_-)+1)}{\Gamma(\lambda_{\epsilon'}+\mu_{-\epsilon}+\nu_+)\Gamma(\lambda_{\epsilon'}+\mu_{-\epsilon}+\nu_-)},\ \ \ \ \epsilon,\epsilon'=\pm.
\end{align*}
We can see that this connection matrix is regular. Furthermore, expressing $v_-$ as a linear combination of $u_+$ and $u_-$, we see that the coefficient of $u_-$ is non-zero. 
Thus, from (\ref{F}) and (\ref{G}), we see that the coefficients of the leading terms $x^{-2h_{1,2}}$ in
\begin{align*}
\langle v^*,\mathcal{Y}^2_{21}(v,1)(p_1\circ\mathcal{Y}_{2\boxtimes{2}})(v,x)v\rangle&=\langle v^*, \mathcal{Y}^2_{21}(v,1)\ket{0}\rangle x^{-2h_{1,2}}(1+x\mathbb{C}[[x]]),\\
\langle v^*,\mathcal{Y}^2_{12}\bigl((p_1\circ\mathcal{Y}_{2\boxtimes{2}})(v,x)v,1-x\bigr)v\rangle&=\langle v^*, \mathcal{Y}^2_{12}(\ket{0},1)v\rangle x^{-2h_{1,2}}(1+x\mathbb{C}[[x]])
\end{align*}
are non-zero.
Therefore, from (\ref{fourpointori}), we obtain (\ref{show00}).
\end{proof}
\end{thm}

\begin{lem}
\label{Intertwining lemma w-}
For $1\leq r\leq p_+$, $2\leq s \leq p_--1$, we have
\begin{equation*}
I
\begin{pmatrix}
\ \X^-_{r,s-1} \\
\K_{1,2}\ \ \X^-_{r,s}
\end{pmatrix}
\neq 0
,\ \ \ \ \ \ 
I
\begin{pmatrix}
\ \X^-_{r,s+1} \\
\K_{1,2}\ \ \X^-_{r,s}
\end{pmatrix}
\neq 0
.
\end{equation*}
\begin{proof}
We will only prove the first inequality. The second inequality can be proved in the same way.

Let us consider the $\W_{p_+,p_-}$-module intertwining operator $Y_1=Y_{\alpha_1,\alpha_2}$ given by Definition \ref{dfnY}, where $\alpha_1=\alpha_{1,2}$, $\alpha_2=\alpha_{r^\vee,s;-2}$.
Note that
\begin{align*}
\bra{\alpha_{r^\vee,s+1;-2}}Y_1(\ket{\alpha_{1,2}},z)\ket{\alpha_{r^\vee,s;-2}}\neq 0.
\end{align*}
Thus we have 
\begin{equation}
I
\begin{pmatrix}
\ \X^-_{r,s-1} \\
\K_{1,2}\ \ \W_{p_+,p_-}.\ket{\alpha_2}
\end{pmatrix}
\neq 0.
\label{Int12132}
\end{equation}
Note that, from Proposition \ref{socleV}, $\W_{p_+,p_-}.\ket{\alpha_{r^\vee,s;-2}}$ satisfies the exact sequence
\begin{align*}
0\rightarrow \X^+_{r^\vee,s}\rightarrow\W_{p_+,p_-}.\ket{\alpha_{r^\vee,s;-2}}\rightarrow \X^-_{r,s}\rightarrow 0.
\end{align*}
Then, from this exact sequence, we have the following exact sequence
\begin{align*}
\K_{1,2}\boxtimes\X^+_{r^\vee,s}\rightarrow\K_{1,2}\boxtimes\W_{p_+,p_-}.\ket{\alpha_{r^\vee,s;-2}}\rightarrow \K_{1,2}\boxtimes\X^-_{r,s}\rightarrow 0.
\end{align*}
Thus we obtain the exact sequence
\begin{align}
0&\rightarrow{\rm Hom}_{\mathcal{C}_{p_+,p_-}}(\K_{1,2}\boxtimes\X^-_{r,s},\X^-_{r,s-1})\rightarrow {\rm Hom}_{\mathcal{C}_{p_+,p_-}}(\W_{p_+,p_-}.\ket{\alpha_{r^\vee,s;-2}},\X^-_{r,s-1})\nonumber\\
&\rightarrow {\rm Hom}_{\mathcal{C}_{p_+,p_-}}(\K_{1,2}\boxtimes\X^+_{r^\vee,s},\X^-_{r,s-1}).
\label{exact12131}
\end{align}
By Proposition \ref{NGK2}, we see that
\begin{align*}
{\rm Hom}_{\mathcal{C}_{p_+,p_-}}(\K_{1,2}\boxtimes\X^+_{r^\vee,s},\X^-_{r,s-1})=0.
\end{align*}
Thus, by (\ref{Int12132}) and (\ref{exact12131}), we see that $\K_{1,2}\boxtimes\X^-_{r,s}$ has a quotient isomorphic to $\X^-_{r,s-1}$. Thus we obtain the first inequality.
\end{proof}
\end{lem}

By Proposition \ref{NGK2} and Lemma \ref{Intertwining lemma w-}, we obtain the following proposition.
\begin{prop}
\label{fusion12-}
\mbox{}
\begin{enumerate}
\item For $1\leq r\leq p_+$, $2\leq s\leq p_--1$, we have
\begin{align*}
\K_{1,2}\boxtimes \X^-_{r,s}=\X^-_{r,s-1}\oplus \X^-_{r,s+1}.
\end{align*}
\item For $1\leq r\leq p_+$, we have
\begin{align*}
\K_{1,2}\boxtimes \X^-_{r,1}=\X^-_{r,2}.
\end{align*}
\item  For $p_+\geq 3$, $2\leq r\leq p_+-1$, $1\leq s\leq p_-$, we have
\begin{align*}
\K_{2,1}\boxtimes \X^-_{r,s}=\X^-_{r-1,s}\oplus \X^-_{r+1,s}.
\end{align*}
\item For $1\leq s\leq p_-$, we have
\begin{align*}
\K_{2,1}\boxtimes \X^-_{1,s}=\X^-_{2,s}.
\end{align*}
\end{enumerate}
\end{prop}

For the following lemma, see \cite{Etingof,JS,KL} and the references therein.
\begin{lem}
\label{Etingof}
Let $(\mathcal{C},\otimes)$ be a tensor category and let $V$ be a rigid object in $\mathcal{C}$. Then there is a natural adjunction isomorphism
\begin{align*}
{\rm Hom}_{\mathcal{C}}(U\otimes V,W)\simeq {\rm Hom}_{\mathcal{C}}(U,W\otimes V^\vee),
\end{align*}
where $U,W$ are any objects in $\mathcal{C}$ and $V^\vee$ is the dual object of $V$. 
\end{lem}

\begin{lem}
\label{Fusion12Wkai0simple}
\mbox{}
\begin{enumerate}
\item For $1\leq r\leq p_+$, $2\leq s\leq p_--1$, we have
\begin{align*}
\K_{1,2}\boxtimes\X^+_{r,s}=\X^+_{r,s-1}\oplus\X^+_{r,s+1}.
\end{align*}
\item For $1\leq r\leq p_+$, we have
\begin{align*}
\K_{1,2}\boxtimes\X^+_{r,1}=\X^+_{r,2}.
\end{align*}
\end{enumerate}
\begin{proof}
We will only prove 
\begin{align}
\label{0627x+}
\K_{1,2}\boxtimes\X^+_{r,s}=\X^+_{r,s-1}\oplus \X^+_{r,s+1}
\end{align}
for $1\leq r\leq p_+-1$, $2\leq s\leq p_+-2$. The other cases can be proved in a similar way.

In this case, from Proposition \ref{NGK2}, we have $\K_{1,2}\boxtimes\X^+_{r,s}\in C^{thick}_{r,s-1}\oplus C^{thick}_{r,s+1}$.
By Proposition \ref{fusion12-}, Lemma \ref{Etingof} and the self-duality of $\K_{1,2}$, we see that
\begin{align*}
{\rm Hom}_{\mathcal{C}_{p_+,p_-}}(\X,\K_{1,2}\boxtimes\X^+_{r,s})=0,
\end{align*}
where $\X=\X^-_{r^\vee,s-1},\X^-_{r,s^\vee+1},\X^-_{r^\vee,s+1},\X^-_{r,s^\vee-1}$. Thus, by Proposition \ref{NGK2} and Lemma \ref{Intertwiningori}, we obtain (\ref{0627x+}).
\end{proof}
\end{lem}

By Proposition \ref{Fusion12W} and Lemma \ref{Fusion12Wkai0simple}, we obtain the following proposition.
\begin{prop}
\label{Fusion12Wkai0}
\mbox{}
\begin{enumerate}
\item For $1\leq r\leq p_+$, $2\leq s\leq p_--1$, we have
\begin{align*}
\K_{1,2}\boxtimes\K_{r,s}=\K_{r,s-1}\oplus\K_{r,s+1}.
\end{align*}
\item For $1\leq r\leq p_+$, we have
\begin{align*}
\K_{1,2}\boxtimes\K_{r,1}=\K_{r,2}.
\end{align*}
\end{enumerate}
\end{prop}

For the following Lemma, see \cite[Theorem 2.12]{AL} (cf. \cite[Proposition 3.46]{HLZ2}). 
\begin{lem}[\cite{AL}]
\label{UV}
For any $U,V\in \mathcal{C}_{p_+,p_-}$, we have a natural isomorphism.
\begin{align*}
{\rm Hom}_{\mathcal{C}_{p_+,p_-}}(U,V)\simeq {\rm Hom}_{\mathcal{C}_{p_+,p_-}}(U\boxtimes V^*,\K^*_{1,1}).
\end{align*}
\end{lem}

\begin{lem}
\label{lem&p_+=2}
For $1\leq r\leq p_+-1$, $1\leq s\leq p_--1$, we have surjective intertwining operators of types
\begin{align*}
&
\begin{pmatrix}
\  \K_{p_+-1,s}^*\\
\K_{2,1}\ \ \K_{p_+,s}
\end{pmatrix}
,
&
\begin{pmatrix}
\  \K^*_{r,p_--1}\\
\K_{1,2}\ \ \K_{r,p_-}
\end{pmatrix}
.
\end{align*}
\begin{proof}
We only prove the first case. The second case can be proved in the same way.
 
For $\alpha_1=\alpha_{2,1}$ and $\alpha_2=\alpha_{p_+,s^\vee;1}$, let us consider $\W_{p_+,p_-}$-module intertwining operator $Y=Y_{\alpha_1,\alpha_2}$ given by Definition \ref{dfnY}. Note that
\begin{align}
\label{refg}
\bra{\alpha_{1,s^\vee}}Y(\ket{\alpha_{2,1}},z)\ket{\alpha_{p_+,s^\vee;1}}\neq 0.
\end{align}
Then, by Proposition \ref{FusionInt}, we have
\begin{equation}
\label{refgg}
\bra{\alpha_{1,s^\vee}}S^*_{1,s^\vee}Y(\ket{\alpha_{2,1}},z)\ket{\alpha_{p_+,s^\vee;1}}\neq 0.
\end{equation}
Let $v$ be a cosingular vector in $F_{1,s^\vee}[s^\vee]$.  Then from Proposition \ref{FockSocle}, we have $S^*_{1,s^\vee}v\in \mathbb{C}^\times \ket{\alpha_{1,s^\vee}}$, and
from Proposition \ref{socleV}, $\mathcal{V}^+_{1,s}=\mathcal{W}_{p_+,p_-}.v$ has a quotient isomorphic to $\K_{p_+-1,s}^*$.
Thus by (\ref{refg}) and (\ref{refgg}) we obtain
\begin{align}
&I
\begin{pmatrix}
\  \K_{p_+-1,s}^*\\
\K_{2,1}\ \ \W_{p_+,p_-}.\ket{\alpha_{p_+,s^\vee;1}}
\end{pmatrix}
\neq 0
,
&I
\begin{pmatrix}
\  \X^+_{p_+-1,s}\\
\K_{2,1}\ \ \W_{p_+,p_-}.\ket{\alpha_{p_+,s^\vee;1}}
\end{pmatrix}
\neq 0.
\label{202208134}
\end{align}
Note that, from Proposition \ref{socleV}, $\W_{p_+,p_-}.\ket{\alpha_{p_+,s^\vee;1}}$ satisfies the following exact sequence
\begin{align*}
0\rightarrow \X^-_{p_+,s^\vee}\rightarrow \W_{p_+,p_-}.\ket{\alpha_{p_+,s^\vee;1}}\rightarrow \X^+_{p_+,s}(=\K_{p_+,s})\rightarrow 0.
\end{align*}
Then we have the following exact sequence
\begin{align*}
\K_{2,1}\boxtimes\X^-_{p_+,s^\vee}\rightarrow \K_{2,1}\boxtimes\W_{p_+,p_-}.\ket{\alpha_{p_+,s^\vee;1}}\rightarrow \K_{2,1}\boxtimes\K_{p_+,s}\rightarrow 0.
\end{align*}
By this exact sequence and ${\rm Hom}_{\mathcal{C}_{p_+,p_-}}(\K_{2,1}\boxtimes\X^-_{p_+,s^\vee},\K_{p_+-1,s}^*)=0$ which follows from Proposition \ref{NGK2}, we obtain 
\begin{equation}
\label{202212140}
{\rm Hom}_{\mathcal{C}_{p_+,p_-}}(\K_{2,1}\boxtimes\K_{p_+,s},\K_{p_+-1,s}^*)
\simeq {\rm Hom}_{\mathcal{C}_{p_+,p_-}}(\K_{2,1}\boxtimes\W_{p_+,p_-}.\ket{\alpha_{p_+,s^\vee;1}},\K_{p_+-1,s}^*).
\end{equation}
Since $L(h_{p_+-1,s})\boxtimes \K_{p_+,s}=0$, by Lemma \ref{UV}, we have
\begin{align*}
{\rm Hom}_{\mathcal{C}_{p_+,p_-}}(\K_{2,1}\boxtimes\K_{p_+,s},L(h_{p_+-1,s}))=0.
\end{align*}
Thus, by (\ref{202208134}) and (\ref{202212140}), we have a surjective module map from $\K_{2,1}\boxtimes\K_{p_+,s}$ to $\K_{p_+-1,s}^*$.\\
\end{proof}
\end{lem}
We have the following standard lemma for the product of rigid objects in any tensor category.
\begin{lem}
\label{Etingof0}
Let $(\mathcal{C},\otimes)$ be a tensor category. Let $V_1$ and $V_2$ be rigid object in $\mathcal{C}$. Then $V_1\otimes V_2$ is also rigid with dual $V^\vee_2\otimes V^\vee_1$, where $V^\vee_i$ be the dual of $V_i$.
\end{lem}

\begin{thm}
\label{rigid21}
$\K_{2,1}$ is rigid and self-dual in $(\mathcal{C}_{p_+,p_-},\boxtimes,\mathcal{K}_{1,1})$. In the case of $p_+=2$, we have
\begin{align*}
\K_{2,1}\boxtimes\K_{2,1}\simeq \mathcal{Q}(\X^+_{1,1})_{1,1}.
\end{align*}
\begin{proof}
In the case $p_+\geq 3$, the rigidity can be proved in the same way as in Theorem \ref{rigid12ori}. Therefore let $p_+=2$. 
We will use the method in \cite[Subsection 4.2.2]{CMY} to prove the rigidity.

Note that in this case $\K_{2,1}=\X^+_{2,1}$ from the definition. 
First let us show that $\K_{2,1}\boxtimes\K_{2,1}$ is isomorphic to $\mathcal{Q}(\X^+_{1,1})_{1,1}$ or a quotient of $\mathcal{Q}(\X^+_{1,1})_{1,1}$.
By Proposition \ref{NGK22} and by Lemma \ref{lem&p_+=2}, we have
\begin{align}
\label{top08200}
{\rm top}(\K_{2,1}\boxtimes\K_{2,1})=\X^+_{1,1}\in C^{thick}_{1,1},
\end{align}
and $\K_{2,1}\boxtimes\K_{2,1}$ has a quotient isomorphic to $\mathcal{K}^*_{1,1}$.
Note that $\K_{2,1}\boxtimes\K_{2,1}\in C^{thick}_{1,1}$ does not contain $\X^+_{1,p_--1},\X^-_{1,p_--1}\in C^{thick}_{1,1}$ as composition factors. In fact, assuming that $\K_{2,1}\boxtimes\K_{2,1}$ contains $\X^\pm_{1,p_--1}$, then from the rigidity of $\K_{1,2}$ and from Proposition \ref{fusion12-}, $\K_{1,2}\boxtimes(\K_{2,1}\boxtimes\K_{2,1})$ contains $\X^\pm_{1,p_-}\in C^{thin}_{1,p_-}$ as a composition factor, but since
\begin{align*}
 \K_{1,2}\boxtimes(\K_{2,1}\boxtimes\K_{2,1})=\K_{2,2}\boxtimes\K_{2,1}\in C^{thick}_{1,2}
\end{align*}
by Propositions \ref{NGK22} and \ref{Fusion12Wkai0}, so we have a contradiction. Thus, from Propositions \ref{Ext}, \ref{ExtQ} and (\ref{top08200}), we see that $\K_{2,1}\boxtimes\K_{2,1}$ is isomorphic to $\mathcal{Q}(\X^+_{1,1})_{1,1}$ or a quotient of $\mathcal{Q}(\X^+_{1,1})_{1,1}$.

Note that ${\rm dim}_{\mathbb{C}}{\rm Hom}_{\W_{p_+,p_-}}(\mathcal{K}_{1,1},\mathcal{K}_{2,1}\boxtimes \mathcal{K}_{2,1})=1$. Then we can define the non-trivial module map 
\begin{align*}
i_1:\K_{1,1}\rightarrow \mathcal{K}_{2,1}\boxtimes \mathcal{K}_{2,1}.
\end{align*}
By Lemma \ref{lem&p_+=2}, we can define the module map 
\begin{align*}
p_1:\K_{2,1}\boxtimes\K_{2,1}\rightarrow \K_{1,1}
\end{align*}
such that $p_1(\K_{2,1}\boxtimes\K_{2,1})\simeq \mathcal{X}^+_{1,1}$.

To prove that $\K_{2,1}$ is rigid and self-dual, it is sufficient to prove that the homomorphisms $f,g:\K_{2,1}\rightarrow \K_{2,1}$ defined by the commutative diagrams
\begin{equation*}
 \xymatrix{
    \K_{2,1} \ar[r]^-{r^{-1}} \ar[d]_{f} & \K_{2,1}\boxtimes\K_{1,1} \ar[r]^-{{\rm id}\boxtimes i_1}& \K_{2,1}\boxtimes(\K_{2,1}\boxtimes\K_{2,1})\ar[d]^{\mathcal{A}}& \\
    \K_{2,1} & \K_{1,1}\boxtimes\K_{2,1}\ar[l]_-l&(\K_{2,1}\boxtimes\K_{2,1})\boxtimes\K_{2,1}\ar[l]_-{p_1\boxtimes{\rm id}}
  }
\end{equation*}
and
\begin{equation*}
 \xymatrix{
    \K_{2,1} \ar[r]^-{l^{-1}} \ar[d]_{g} & \K_{1,1}\boxtimes\K_{2,1} \ar[r]^-{i_1\boxtimes{\rm id}}& (\K_{2,1}\boxtimes\K_{2,1})\boxtimes\K_{2,1}\ar[d]^{\mathcal{A}^{-1}}& \\
    \K_{2,1} & \K_{2,1}\boxtimes\K_{1,1}\ar[l]_-r&\K_{2,1}\boxtimes(\K_{2,1}\boxtimes\K_{2,1})\ar[l]_-{{\rm id}\boxtimes p_1}
  }
\end{equation*}
are non-zero multiples of the identity and $f=g$. Since ${\rm Hom}(\K_{2,1},\K_{2,1})\simeq \C$, it is sufficient to show that $f=g\neq 0$.


Let $\mathcal{Y}_{2\boxtimes {2}}$, $\mathcal{Y}_{(2\boxtimes 2)\boxtimes {2}}$ and $\mathcal{Y}_{2\boxtimes (2\boxtimes 2)}$ be the non-zero intertwining operators of type
\begin{equation*}
\begin{pmatrix}
\ \K_{2,1}\boxtimes\K_{2,1} \\
\K_{2,1}\ \ \K_{2,1}
\end{pmatrix}
,\ \ \ 
\begin{pmatrix}
\ (\K_{2,1}\boxtimes\K_{2,1})\boxtimes \K_{2,1} \\
\K_{2,1}\boxtimes\K_{2,1}\ \ \K_{2,1}
\end{pmatrix}
,\ \ \ \begin{pmatrix}
\ \K_{2,1}\boxtimes(\K_{2,1}\boxtimes \K_{2,1}) \\
\K_{2,1}\ \ \K_{2,1}\boxtimes \K_{2,1}
\end{pmatrix}
,
\end{equation*}
respectively. We define the intertwining operators 
\begin{equation}
\label{IZ}
\begin{split}
&\mathcal{Y}^{2}_{2\boxtimes 2,2}=r_{\K_{2,1}}\circ ({\rm id}_{\K_{2,1}}\boxtimes p_1)\circ \mathcal{A}^{-1}_{\K_{2,1},\K_{2,1},\K_{2,1}}\circ \mathcal{Y}_{(2\boxtimes 2)\boxtimes 2},\\
&\mathcal{Y}^2_{2,2\boxtimes 2}=l_{\K_{2,1}}\circ (p_1\boxtimes{\rm id}_{\K_{2,1}})\circ \mathcal{A}_{\K_{2,1},\K_{2,1},\K_{2,1}}\circ \mathcal{Y}_{2\boxtimes(2\boxtimes 2)}
\end{split}
\end{equation}
of types
{\scriptsize{$\begin{pmatrix}
   \ \K_{2,1}  \\
   \K_{2,1}\boxtimes \K_{2,1}\ \ \K_{2,1}
\end{pmatrix}
$}} and {\scriptsize{$\begin{pmatrix}
   \ \K_{2,1}  \\
   \K_{2,1}\ \ \K_{2,1}\boxtimes \K_{2,1}
\end{pmatrix}
$}}.

Let us consider the four point function
\begin{align}
\label{fourW21}
\langle v^*,\mathcal{Y}^2_{2\boxtimes 2,2}(\mathcal{Y}_{2\boxtimes 2}(v,1-x)v,x)v\rangle,
\end{align}
where $v$ and $v^*$ are the highest weight vectors of $\K_{2,1}$ and $\K_{2,1}^*$, respectively. Then, for some $x\in\mathbb{R}$ such that $1>x>1-x>0$, we have
\begin{equation}
\label{4point0}
\begin{split}
&\langle v^*,\mathcal{Y}^2_{2\boxtimes 2,2}(\mathcal{Y}_{2\boxtimes 2}(v,1-x)v,x)v\rangle\\
&=\langle v^*,\overline{r_{\K_{2,1}}\circ ({\rm id}_{\K_{2,1}}\boxtimes p_1)\mathcal{A}^{-1}_{\K_{2,1},\K_{2,1},\K_{2,1}}} \bigl(\mathcal{Y}_{(2\boxtimes 2)\boxtimes 2}(\mathcal{Y}_{2\boxtimes 2}(v,1-x)v,x)v\bigr)\rangle\\
&=\langle v^*,\overline{r_{\K_{2,1}}\circ ({\rm id}_{\K_{2,1}}\boxtimes p_1)} \bigl(\mathcal{Y}_{2\boxtimes (2\boxtimes 2)}(v,1)\mathcal{Y}_{2\boxtimes 2}(v,x)v \bigr)\rangle\\
&=\langle v^*, \overline{r_{\K_{2,1}}}\bigl(\mathcal{Y}_{2\boxtimes 1}(v,1)(p_1\circ\mathcal{Y}_{2\boxtimes 2})(v,x)v\bigr)\rangle\\
&=\langle v^*, \Omega(Y_{\K_{2,1}})(v,1)(p_1\circ\mathcal{Y}_{2\boxtimes 2})(v,x)v\rangle,
\end{split}
\end{equation}
where $\Omega$ represents the skew-symmetry operation on vertex operators defined by
\begin{align*}
\Omega(Y_{\K_{2,1}})(v,z)w=e^{zL_{-1}}Y_{\K_{2,1}}(w,-z)v
\end{align*}
for $w\in \K_{1,1}$.
Similarly, we can show that
\begin{align}
\label{4point02}
\langle v^*,\mathcal{Y}^2_{2,2\boxtimes 2}(v,1)\mathcal{Y}_{2\boxtimes 2}(v,1-x)v\rangle=\langle v^*, Y_{\K_{2,1}}\bigl((p_1\circ\mathcal{Y}_{2\boxtimes 2})(v,x)v,1-x)\bigr)v\rangle.
\end{align}
Then, by Lemma \ref{lem&p_+=2}, we see that the four point functions (\ref{4point0}) and (\ref{4point02}) are non-zero.

Note that from Proposition \ref{FockSocle}, we have
\begin{align}
\label{relsing2}
\bigl( L^2_{-1}-\frac{p_-}{2}L_{-2} \bigr)\ket{\alpha_{2,1}}=0.
\end{align}
Then, from the relation (\ref{relsing2}), as in \cite{BPZ,CMY,McRae,TWFusion}, we see that the four point functions (\ref{4point0}) and (\ref{4point02}) satisfy the following Fuchsian differential equation
\begin{equation}
\label{BPZ0714}
\begin{split}
\phi''(x)&+\frac{p_-}{2}\Bigl(\frac{1}{x-1}+\frac{1}{x}\Bigr)\phi'(x)\\
&-\frac{p_-h_{2,1}}{2}\Bigl\{ \frac{1}{(x-1)^2}+\frac{1}{x^2}-2\Bigl(\frac{1}{x-1}-\frac{1}{x}\Bigr)\Bigr\}\phi(x)=0.
\end{split}
\end{equation}
Note that $p_1\circ\mathcal{Y}_{2\boxtimes 2}$ is a non-logarithmic intertwining operator of type
{\scriptsize{$\begin{pmatrix}
   \ \mathcal{X}^+_{1,1}  \\
   \K_{2,1}\ \ \K_{2,1}
\end{pmatrix}
$}} and
\begin{align*}
\langle v^*_{\mathcal{X}^+_{1,1}},(p_1\circ\mathcal{Y}_{2\boxtimes 2})(v,x)v\rangle\in \mathbb{C}^\times x^{\frac{p_-}{4}},
\end{align*}
where $v^*_{\mathcal{X}^+_{1,1}}$ is the highest weight vector of $(\mathcal{X}^+_{1,1})^*$.
Then we have
\begin{align}
\label{eqFG}
\langle v^*, \Omega(Y_{\K_{2,1}})(v,1)(p_1\circ\mathcal{Y}_{2\boxtimes 2})(v,x)v\rangle=\langle v^*, Y_{\K_{2,1}}\bigl((p_1\circ\mathcal{Y}_{2\boxtimes 2})(v,x)v,1-x)\bigr)v\rangle
\end{align}
and
\begin{align*}
&\langle v^*, \Omega(Y_{\K_{2,1}})(v,1)(p_1\circ\mathcal{Y}_{2\boxtimes 2})(v,x)v\rangle\\
&=cx^{\frac{p_-}{4}}(1-x)^{\frac{p_-}{4}}{}_{2}F_{1}\bigl(\frac{p_-}{2},\frac{3p_-}{2}-1,p_-;x\bigr)\\
&=cx^{\frac{p_-}{4}}(1-x)^{-\frac{3p_-}{4}+1}{}_{2}F_{1}\bigl(1-\frac{p_-}{2},\frac{p_-}{2},p_-;x\bigr)
\end{align*}
where $c$ is a non-zero constant. By using Equation 15.8.10 in \cite{D}, we can see that 
\begin{align*}
{}_{2}F_{1}\bigl(1-\frac{p_-}{2},\frac{p_-}{2},p_-;x\bigr)\in \mathbb{C}^\times\bigl(1+(1-x)\C[[1-x]][{\rm log}(1-x)]\bigr).
\end{align*}
Thus we obtain
\begin{equation}
\label{koreda0716}
\begin{split}
&\langle v^*, \Omega(Y_{\K_{2,1}})(v,1)(p_1\circ\mathcal{Y}_{2\boxtimes 2})(v,x)v\rangle\\
&\in \mathbb{C}^\times(1-x)^{-2h_{2,1}}\bigl(1+(1-x)\C[[1-x]][{\rm log}(1-x)]\bigr),
\end{split}
\end{equation}
where $-2h_{2,1}=-\frac{3p_-}{4}+1$.


We introduce the intertwining operators
\begin{align*}
&\mathcal{Y}^2_{21}=l_{\K_{2,1}}\circ (p_1\boxtimes {\rm id}_{\K_{2,1}})\circ \mathcal{A}_{\K_{2,1},\K_{2,1},\K_{2,1}}\circ \mathcal{Y}_{2\boxtimes (2\boxtimes 2)}\circ ({\rm id}_{\K_{2,1}}\otimes i_1),\\
&\mathcal{Y}^2_{12}=r_{\K_{2,1}}\circ ({\rm id}_{\K_{2,1}}\boxtimes p_1)\circ \mathcal{A}^{-1}_{\K_{2,1},\K_{2,1},\K_{2,1}}\circ \mathcal{Y}_{(2\boxtimes 2)\boxtimes 2}\circ (i_1\otimes {\rm id}_{\K_{2,1}}) 
\end{align*} 
which correspond to $f$ and $g$, respectively. To prove $f=g\neq 0$, it is sufficient to show that 
\begin{align}
\label{show0}
\langle v^*, \mathcal{Y}^2_{21}(v,1)\ket{0}\rangle=\langle v^*, \mathcal{Y}^2_{12}(\ket{0},1)v\rangle\neq 0.
\end{align}
We set $c_f=\langle v^*, \mathcal{Y}^2_{21}(v,1)\ket{0}\rangle$ and $c_g=\langle v^*, \mathcal{Y}^2_{12}(\ket{0},1)v\rangle$.
Note that
\begin{equation}
\label{F0803}
\begin{split}
&(i_1\boxtimes{\rm id}_{\K_{2,1}})\circ l^{-1}_{\K_{2,1}}(v)\\
&\ \ =(i_1\boxtimes{\rm id}_{\K_{2,1}})(\mathcal{Y}_{1\boxtimes 2}(\ket{0},1)v)=\mathcal{Y}_{(2\boxtimes 2)\boxtimes {2}}(i_1(\ket{0}),1)v,\\
&({\rm id}_{\K_{2,1}}\boxtimes i_1)\circ r^{-1}_{\K_{2,1}}(v)\\
&\ \ =({\rm id}_{\K_{2,1}}\boxtimes i_1)(\mathcal{Y}_{2\boxtimes 1}(v,1)\ket{0})=\mathcal{Y}_{2\boxtimes(2\boxtimes 2)}(v,1)i_1(\ket{0})
\end{split}
\end{equation}
where $\mathcal{Y}_{1\boxtimes {2}}$ and $\mathcal{Y}_{2\boxtimes 1}$ are the non-zero intertwining operators of types
{\scriptsize{$\begin{pmatrix}
   \ \K_{2,1}  \\
   \K_{1,1}\ \K_{2,1}
\end{pmatrix}
$}} and {\scriptsize{$\begin{pmatrix}
   \ \K_{2,1}  \\
   \K_{2,1}\ \K_{1,1}
\end{pmatrix}
$}}, respectively.
Since $\K_{2,1}\boxtimes\K_{2,1}$ has a quotient isomorphic to $\mathcal{K}^*_{1,1}$, we see that $i_1(\ket{0})$ is the cofficient of $x^{-2h_{2,1}}$ in $\mathcal{Y}_{2\boxtimes 2}(v,x)v$.
Then, from (\ref{F0803}), we have
\begin{equation}
\label{eq07162}
\begin{split}
\langle v^*,\mathcal{Y}^2_{2,2\boxtimes 2}(&v,1)\mathcal{Y}_{2\boxtimes 2}(v,1-x)v\rangle\\
&\in c_f(1-x)^{-2h_{2,1}}\bigl(1+(1-x)\C[[1-x]][{\rm log}(1-x)]\bigr),\\
\langle v^*,\mathcal{Y}^2_{2\boxtimes 2,2}(&\mathcal{Y}_{2\boxtimes 2}(v,1-x)v,x)v\rangle\\
&\in c_g(1-x)^{-2h_{2,1}}\bigl(1+(1-x)\C[[1-x]][{\rm log}(1-x)]\bigr).
\end{split}
\end{equation}
Thus from (\ref{eqFG}), (\ref{koreda0716}) and (\ref{eq07162}), we obtain (\ref{show0}). Therefore $\mathcal{K}_{2,1}$ is self-dual.

Assume that $\K_{2,1}\boxtimes\K_{2,1}\not\simeq \mathcal{Q}(\X^+_{1,1})_{1,1}$. Since $\K_{2,1}\boxtimes\K_{2,1}$ is a quotient of $\mathcal{Q}(\X^+_{1,1})_{1,1}$ and has a quotient isomorphic to $\mathcal{K}^*_{1,1}$, we have
\begin{align}
\label{08032023}
{\rm Hom}_{\mathcal{C}_{p_+,p_-}}(L(h_{1,1}),\mathcal{K}_{2,1}\boxtimes \mathcal{K}_{2,1})\simeq \mathbb{C}.
\end{align}
On the other hand, from Proposition \ref{minimalnull}, Lemma \ref{Etingof0} and the self-duality of $\mathcal{K}_{2,1}$, we have
\begin{align*}
{\rm Hom}_{\mathcal{C}_{p_+,p_-}}(L(h_{1,1}),\mathcal{K}_{2,1}\boxtimes \mathcal{K}_{2,1})\simeq {\rm Hom}_{\mathcal{C}_{p_+,p_-}}(L(h_{1,1})\boxtimes\mathcal{K}_{2,1},\mathcal{K}_{2,1})=0.
\end{align*}
But this contradicts (\ref{08032023}). Thus we obtain $\K_{2,1}\boxtimes\K_{2,1}\simeq \mathcal{Q}(\X^+_{1,1})_{1,1}$.

\end{proof}
\end{thm}

Similar to Proposition \ref{Fusion12Wkai0}, we can prove the following proposition.
\begin{prop}
\label{Fusion12Wkai}
\mbox{}
\begin{enumerate}
\item For $p_+\geq 3$, $2\leq r\leq p_+-1$, $1\leq s\leq p_-$, we have
\begin{align*}
\K_{2,1}\boxtimes\K_{r,s}=\K_{r-1,s}\oplus\K_{r+1,s}.
\end{align*}
\item For $1\leq s\leq p_-$, we have
\begin{align*}
\K_{2,1}\boxtimes\K_{1,s}=\K_{2,s} .
\end{align*}
\end{enumerate}
\end{prop}

By Propositions \ref{Fusion12Wkai0}, \ref{Fusion12Wkai}, Theorems \ref{rigid12ori}, \ref{rigid21} and Lemma \ref{Etingof0}, we obtain the following theorem.
\begin{thm}
\label{Wrigid}
For $1\leq r\leq p_+,1\leq s\leq p_-$, the indecomposable modules $\K_{r,s}$ are rigid and self-dual.
\end{thm}

\subsection{Tensor products between simple modules $\X^\pm_{r,s}$}
\label{simple0817}
In this section, we will determine the tensor products between simple modules $\X^\pm_{r,s}$.
In the papers \cite{GRW0,GRW,W}, the explicit formulas for these tensor products are given with the Nahm-Gaberdiel-Kausch fusion algorithm \cite{Kanade,GK} when the central charge $c_{p_+,p_-}$ is special.
We will show that the structure of the tensor products can be determined without complex computation by using the self-duality of $\mathcal{K}_{r,s}$. 


First let us determine the tensor products $\mathcal{X}^+_{r,s}\boxtimes \mathcal{X}^+_{r',s'}$.
By Cororally \ref{202210180} and by Theorem \ref{Wrigid}, we obtain the following proposition.
\begin{prop}
\label{Fusion11+}
For $1\leq r\leq p_+$, $1\leq s\leq p_-$, we have
\begin{align*}
\X^+_{1,1}\boxtimes \K_{r,s}=\X^+_{r,s}.
\end{align*}
\end{prop}
Note that $\X^+_{1,1}$ is self-contragredient and $L(h_{1,1})\boxtimes \X^+_{1,1}=0$. Then, by Lemma \ref{UV}, we have
\begin{align*}
&{\rm Hom}_{\mathcal{C}_{p_+,p_-}}(\X^+_{1,1}\boxtimes \X^+_{1,1},\K^*_{1,1})=\C,
&{\rm Hom}_{\mathcal{C}_{p_+,p_-}}(\X^+_{1,1}\boxtimes \X^+_{1,1},L(h_{1,1}))=0.
\end{align*}
Thus we obtain the following lemma.
\begin{lem}
\label{lemma++}
$\X^+_{1,1}\boxtimes \X^+_{1,1}$ has a quotient isomorphic to $\K^*_{1,1}$.
\end{lem}
\begin{prop}
\label{++}
We have
\begin{align*}
\X^+_{1,1}\boxtimes\X^+_{1,1}=\K_{1,1}^*.
\end{align*}
\begin{proof}
By Lemma \ref{lemma++}, $\X^+_{1,1}\boxtimes\X^+_{1,1}\neq 0$. Let us show that
\begin{align}
\label{2022120600}
{\rm top}(\X^+_{1,1}\boxtimes\X^+_{1,1})=\X^+_{1,1}.
\end{align}
Let $\pi$ be the surjection from $\X^+_{1,1}\boxtimes \X^+_{1,1}$ to ${\rm top}(\X^+_{1,1}\boxtimes \X^+_{1,1})$. 
Let $\psi^*$ be an arbitrary non-zero vector of $A_0(({\rm top}(\X^+_{1,1}\boxtimes\X^+_{1,1}))^*)$.
Let $\phi_1$ and $\phi_2$ be arbitrary vectors of $\X^+_{1,1}$ such that
\begin{equation*}
\langle\psi^*,(\pi\circ \mathcal{Y}_{\boxtimes})(\phi_1,1)\phi_2\rangle\neq 0.
\end{equation*}
We abbreviate 
$
\langle\psi^*,\phi_1\otimes \phi_2\rangle=\langle\psi^*,(\pi\circ \mathcal{Y}_{\boxtimes})(\phi_1,1)\phi_2\rangle.
$
Note that, since $L_0$ acts semisimply on $\psi^*$, we have
\begin{align*}
\langle\psi^*,L_{-1}\phi_1\otimes \phi_2\rangle=(h_{\psi^*}-h_{\phi_1}-h_{\phi_2})\langle\psi^*,\phi_1\otimes \phi_2\rangle,
\end{align*}
where $h_{\psi^*}$, $h_{\phi_1}$ and $h_{\phi_2}$ are the $L_0$ weights of $\psi^*$, $\phi_1$ and $\phi_2$, respectively.

For $n\geq 0$, let $\{w^{(n)}_i\}_{i=-n}^{n}$ be the Virasoro highest weight vectors of the vector subspace $(2n+1)L(\Delta^+_{1,1;n})\subset \X^+_{1,1}$ given by Definition \ref{dfn0830}. 
From Proposition \ref{genW}, Proposition \ref{sl2action2} and Lemma \ref{NGK01}, we see that the value $\langle\psi^*,\phi_1\otimes \phi_2\rangle$ is determined by the numbers 
\begin{align*}
\langle\psi^*,w^{(0)}_0\otimes w^{(n)}_i\rangle,\ \ \ n\in\mathbb{Z}_{\geq 0},\ i\in\mathbb{Z}.
\end{align*}
Thus, there exists $n\in \mathbb{Z}_{\geq 0}$ and $i\in \mathbb{Z}$ such that $\langle\psi^*,w^{(0)}_0\otimes w^{(n)}_i\rangle\neq 0$. 
By Proposition \ref{VirasoroFusion}, $h_{\psi^*}$ must satisfies the following equations
\begin{align*}
&\prod_{k=1}^{2p_--1}(h_{\psi^*}-h_{1,(2n+4)p_--2-2k+1})=0,
&\prod_{l=1}^{2p_+-1}(h_{\psi^*}-h_{(2n+4)p_+-2-2l+1,1})=0.
\end{align*}
We see that $h_{\psi^*}$ satisfying these equations is given by $\Delta^+_{1,1;n}=h_{(2+2n)p_+-1,1}$. 
Thus, from Theorem \ref{simpleclass}, we obtain $h_{\psi^*}=\Delta^+_{1,1;0}$ and
\begin{align}
\label{van6}
\langle\psi^*,w^{(0)}_0\otimes w^{(n)}_i\rangle= 0,\ \ \ n\geq 1,\ i\in \mathbb{Z}.
\end{align}
Therefore by (\ref{van6}), $\langle\psi^*,\phi_1\otimes \phi_2\rangle$ is determined by 
$
\langle\psi^*,w^{(0)}_0\otimes w^{(0)}_0\rangle
$.
In particular,
we obtain ${\rm top}(\X^+_{1,1}\boxtimes\X^+_{1,1})=\X^+_{1,1}$.

Note that, by Lemma \ref{lemma++}, $\X^+_{1,1}\boxtimes\X^+_{1,1}$ has a quotient isomorphic to $\K^*_{1,1}$.
Assume that $\X^+_{1,1}\boxtimes\X^+_{1,1}\ncong\K_{1,1}^*$. 
Then, by (\ref{2022120600}) and Propositions \ref{Ext}, \ref{ExtQ}, we see that $\X^+_{1,1}\boxtimes\X^+_{1,1}$ contains $\X^-_{p_+-1,1}$ or $\X^-_{1,p_--1}$ as composition factors.
Then, by the rigidity of $\K_{2,1}$ and $\K_{1,2}$ and by Proposition \ref{fusion12-}, we see that  the direct sum
\begin{align*}
\K_{1,2}\boxtimes(\X^-_{1,1}\boxtimes\X^-_{1,1})\oplus \K_{2,1}\boxtimes(\X^-_{1,1}\boxtimes\X^-_{1,1})
\end{align*}
contains $\X^-_{1,p_-}$ or $\X^-_{p_+,1}$ as composition factors. On the other hand, by the associativity and Proposition \ref{Fusion11+}, we have
\begin{align*}
&\K_{1,2}\boxtimes(\X^+_{1,1}\boxtimes\X^+_{1,1})=(\K_{1,2}\boxtimes\X^+_{1,1})\boxtimes\X^+_{1,1}=\X^+_{1,2}\boxtimes\X^+_{1,1},\\
&\K_{2,1}\boxtimes(\X^+_{1,1}\boxtimes\X^+_{1,1})=(\K_{2,1}\boxtimes\X^+_{1,1})\boxtimes\X^+_{1,1}=\X^+_{2,1}\boxtimes\X^+_{1,1}.
\end{align*}
Similar to the case of $\X^+_{1,1}\boxtimes\X^+_{1,1}$, by using Proposition \ref{VirasoroFusion}, we can show that the top composition factors of $\X^+_{1,2}\boxtimes\X^+_{1,1}$ and $\X^+_{2,1}\boxtimes\X^+_{1,1}$ are $\X^+_{1,2}$ and $\X^+_{2,1}$, respectively. In particular, we have $\X^+_{1,2}\boxtimes\X^+_{1,1}\in C^{thick}_{1,2}$ and $\X^+_{2,1}\boxtimes\X^+_{1,1}\in C^{a}_{2,1}$, where $a=thick$ for $p_+\geq 3$ and $a=thin$ for $p_+=2$. But since $\X^-_{1,p_-}\in C^{thin}_{p_+-1,p_-}$ and $\X^-_{p_+,1}\in C^{thin}_{p_+,p_--1}$, we have a contradiction. 

\end{proof}

\end{prop}
By Propositions \ref{Fusion11+} and \ref{++}, we obtain the following proposition.
\begin{prop}
\label{12*}
For $1\leq r,r'\leq p_+$, $1\leq s,s'\leq p_-$, we have
\begin{align*}
\X^+_{r,s}\boxtimes \X^+_{r',s'}=(\K_{r,s}\boxtimes\K_{r',s'})\boxtimes\K^*_{1,1}.
\end{align*}
\end{prop}

\begin{prop}
\label{W^*W}
For $1\leq r\leq p_+$, $1\leq s\leq p_-$, we have
\begin{align*}
\K_{1,1}^*\boxtimes\K_{r,s}=\K_{r,s}^*.
\end{align*}
\begin{proof}
By Proposition \ref{++}, we have 
\begin{align}
\X^+_{1,1}\boxtimes\X^+_{1,1}=\K_{1,1}^*.
\label{20221011X1}
\end{align}
Multiplying both sides by $\K_{r,s}$ and using Proposition \ref{Fusion11+}, we have
\begin{align*}
\X^+_{1,1}\boxtimes\X^+_{r,s}=\K_{1,1}^*\boxtimes\K_{r,s}.
\end{align*}
for $1\leq r\leq p_+-1,1\leq s\leq p_--1$.
By the exact sequence
\begin{align*}
0\rightarrow L(h_{1,1})\rightarrow \K_{1,1}^*\rightarrow \X^+_{1,1}\rightarrow 0
\end{align*}
and the rigidity of $\K_{r,s}$, we have the following exact sequence
\begin{align}
\label{exact202210250}
0\rightarrow L(h_{r,s})\rightarrow \K_{1,1}^*\boxtimes\K_{r,s}\rightarrow \X^+_{r,s}\rightarrow 0.
\end{align}
By Theorem \ref{Wrigid} and Lemma \ref{Etingof}, we have 
\begin{align*}
{\rm Hom}_{\mathcal{C}_{p_+,p_-}}(\K_{1,1}^*\boxtimes\K_{r,s},L(h_{r,s}))\simeq{\rm Hom}_{\mathcal{C}_{p_+,p_-}}(\K_{1,1}^*,L(h_{r,s})\boxtimes\K_{r,s})=0.
\end{align*}
Thus the exact sequence (\ref{exact202210250}) does not split. Therefore, since ${\rm Ext}^1(\X^+_{r,s},L(h_{r,s}))=\C$ by Proposition \ref{Ext}, we obtain
\begin{align*}
\K_{1,1}^*\boxtimes\K_{r,s}\simeq \K_{r,s}^*.
\end{align*}
\end{proof}

\end{prop}
\begin{remark}
From Propositions \ref{12*} and \ref{W^*W}, we can compute the tensor products between the simple modules $\X^+_{r,s}$ by using the tensor product between the indecomposable modules $\K_{r,s}$ (see the formulas (\ref{simple0820}) in Section \ref{FusionRing}).
\end{remark}

Next let us determine the fusion formula of $\mathcal{X}^-_{r,s}\boxtimes \mathcal{X}^-_{r',s'}$. Similar to Lemma \ref{lemma++}, by using Lemma \ref{UV}, we obatin the following lemma.
\begin{lem}
\label{lemma--}
$\X^-_{1,1}\boxtimes \X^-_{1,1}$  has a quotient isomorphic to $\K^*_{1,1}$.
\end{lem}
\begin{prop}
\label{--}
We have
\begin{align*}
&\X^-_{1,1}\boxtimes\X^-_{1,1}=\K_{1,1}^*,
&\X^-_{1,1}\boxtimes\X^+_{1,1}=\X^-_{1,1}.
\end{align*}
\begin{proof}
We only prove $\X^-_{1,1}\boxtimes\X^-_{1,1}=\K_{1,1}^*$. The second equality can be proved in the same way.

By Lemma \ref{lemma--}, we see that $\X^-_{1,1}\boxtimes\X^-_{1,1}\neq 0$.
First, we will show 
\begin{align}
\label{20221105yoru}
{\rm top}(\X^-_{1,1}\boxtimes\X^-_{1,1})=\X^+_{1,1}.
\end{align} 
Let $\pi$ be the surjection from $\X^-_{1,1}\boxtimes\X^-_{1,1}$ to ${\rm top}(\X^-_{1,1}\boxtimes \X^-_{1,1})$.
Let $\psi^*$ be an arbitrary non-zero vector of $A_0(({\rm top}(\X^-_{1,1}\boxtimes \X^-_{1,1}))^*)$.
Let $\phi_1$ and $\phi_2$ be arbitrary vectors of $\X^-_{1,1}$ such that $\langle\psi^*,(\pi\circ \mathcal{Y}_{\boxtimes})(\phi_1,1)\phi_2\rangle\neq 0$. 
For $n\geq 0$, let $\bigl\{v^{(n)}_{\frac{2i+1}{2}},v^{(n)}_{\frac{-2i-1}{2}}\bigr\}_{i=0}^{n}$ be the Virasoro highest weight vectors of the vector subspace $(2n+2)L(\Delta^-_{1,1;n})\subset \X^-_{1,1}$ given by by Definition \ref{dfn0830}. 
We abbreviate 
\begin{align*}
\langle \psi^*,\phi_1\otimes \phi_2\rangle=\langle\psi^*,(\pi\circ \mathcal{Y}_{\boxtimes})(\phi_1,1)\phi_2\rangle
\end{align*}
and consider this value.

By Proposition \ref{genW}, Proposition \ref{sl2action2}, and Lemma \ref{NGK01}, we see that
$\langle\psi^*,\phi_1\otimes \phi_2\rangle$ is determined by the numbers
\begin{align*}
\langle\psi^*,v^{(0)}_{\pm\frac{1}{2}}\otimes v^{(n)}_{\frac{2i+1}{2}}\rangle,
\end{align*}
for $n\geq 0$ and $i\in\mathbb{Z}$.
Thus, there exists $n\in\mathbb{Z}_{\geq 0}$ and $i\in\mathbb{Z}$ such that $\langle \psi^*,v^{(0)}_{\pm\frac{1}{2}}\otimes v^{(n)}_{\frac{2i+1}{2}}\rangle\neq 0$. 
Let $h$ be the $L_0$ weight of $\psi^*$. 
By Proposition \ref{VirasoroFusion}, we see that $h$ must satisfies the following equations
\begin{align*}
&\prod_{k=1}^{3p_--1}(h-h_{1,(2n+4)p_--2-2k+1})=0,
&\prod_{l=1}^{3p_+-1}(h-h_{(2n+4)p_+-2-2l+1,1})=0.
\end{align*}
We see that $h$ satisfying these equations are given by $\Delta^+_{1,1;n}=h_{(2+2n)p_+-1,1}$ and $\Delta^+_{1,1;n-1}$. Thus, from Theorem \ref{simpleclass}, we have $h=\Delta^+_{1,1;0}$ and 
\begin{align}
\label{van7}
\langle \psi^*,v^{(0)}_{\pm\frac{1}{2}}\otimes v^{(n)}_{\frac{2i+1}{2}}\rangle=\langle \psi^*,v^{(n)}_{\frac{2i+1}{2}}\otimes v^{(0)}_{\pm\frac{1}{2}}\rangle=0,\ \ \ n\geq 2,\ i\in \mathbb{Z}.
\end{align}
Therefore by (\ref{van7}), $\langle\psi^*,\phi_1\otimes \phi_2\rangle$ is determined by the numbers
\begin{align*}
\langle\psi^*,v^{(0)}_{\epsilon\frac{1}{2}}\otimes v^{(0)}_{\epsilon'\frac{1}{2}}\rangle,&&\langle\psi^*,v^{(0)}_{\pm\frac{1}{2}}\otimes v^{(1)}_{\frac{2i+1}{2}}\rangle,
\end{align*}
for $\epsilon=\pm$, $\epsilon'=\pm$, and $i\in\mathbb{Z}$.

Note that $W^\pm[0]\psi^*=0$ and, by Proposition \ref{sl2action2}, we have
\begin{equation}
\label{van0828}
\begin{split}
W^\pm[0]v^{(0)}_{\mp\frac{1}{2}}&\in \mathbb{C}^\times v^{(0)}_{\pm\frac{1}{2}},\\
v^{(1)}_{\pm\frac{3}{2}}&\in \mathbb{C}^\times W^\pm[\Delta^-_{1,1;0}-\Delta^-_{1,1;1}]v^{(0)}_{\pm\frac{1}{2}},\\
v^{(1)}_{\pm\frac{1}{2}}&\in \mathbb{C}^\times W^0[\Delta^-_{1,1;0}-\Delta^-_{1,1;1}]v^{(0)}_{\pm\frac{1}{2}}+U(\mathcal{L}).v^{(0)}_{\pm\frac{1}{2}},
\end{split}
\end{equation}
and
\begin{equation}
\label{van0827}
W^{\epsilon}[-h]v^{(0)}_{\pm\frac{1}{2}}
\begin{cases}
=0 &\epsilon=\pm \\
\in  U(\mathcal{L}).v^{(0)}_{\pm\frac{1}{2}} &\epsilon=0\\
\in U(\mathcal{L}).v^{(0)}_{\mp\frac{1}{2}} & \epsilon=\mp
\end{cases}
\ \ \ h<\Delta^-_{1,1;1}-\Delta^-_{1,1;0}.
\end{equation}
Then, noting (\ref{van7})-(\ref{van0827}), and using Lemma \ref{NGK01}, we see that
\begin{align*}
\langle\psi^*,v^{(0)}_{\pm\frac{1}{2}}\otimes v^{(0)}_{\pm\frac{1}{2}}\rangle=0
\end{align*}
and the values
\begin{align*}
\langle\psi^*,v^{(0)}_{-\frac{1}{2}}\otimes v^{(0)}_{\frac{1}{2}}\rangle,&&\langle\psi^*,v^{(0)}_{\pm\frac{1}{2}}\otimes v^{(1)}_{\frac{2i+1}{2}}\rangle,\ \ i\in\mathbb{Z}
\end{align*}
are determined by 
$
\langle\psi^*,v^{(0)}_{\frac{1}{2}}\otimes v^{(0)}_{-\frac{1}{2}}\rangle
$.
Consequently,
$\langle\psi^*,\phi_1\otimes \phi_2\rangle$ is determined by
$
\langle\psi^*,v^{(0)}_{\frac{1}{2}}\otimes v^{(0)}_{-\frac{1}{2}}\rangle
$.
In particular we obtain (\ref{20221105yoru}).

Note that by Lemma \ref{lemma--} $\X^-_{1,1}\boxtimes\X^-_{1,1}$ has a quotient isomorphic to $\K^*_{1,1}$.
Assume $\X^-_{1,1}\boxtimes\X^-_{1,1}\ncong\K_{1,1}^*$. Then, noting (\ref{20221105yoru}), from Propositions \ref{Ext} and \ref{ExtQ}, we see that $\X^-_{1,1}\boxtimes\X^-_{1,1}$ contains $\X^-_{p_+-1,1}$ or $\X^-_{1,p_--1}$ as composition factors.
Hence, by the rigidity of $\K_{2,1}$ and $\K_{1,2}$ and by Proposition \ref{fusion12-}, we see that the direct sum
\begin{align*}
\K_{1,2}\boxtimes(\X^-_{1,1}\boxtimes\X^-_{1,1})\oplus \K_{2,1}\boxtimes(\X^-_{1,1}\boxtimes\X^-_{1,1})
\end{align*}
contains $\X^-_{1,p_-}$ or $\X^-_{p_+,1}$ as composition factors. On the other hand, by the associativity and Proposition \ref{fusion12-}, we have 
\begin{align*}
&\K_{1,2}\boxtimes(\X^-_{1,1}\boxtimes\X^-_{1,1})=(\K_{1,2}\boxtimes\X^-_{1,1})\boxtimes\X^-_{1,1}=\X^-_{1,2}\boxtimes\X^-_{1,1},\\
&\K_{2,1}\boxtimes(\X^-_{1,1}\boxtimes\X^-_{1,1})=(\K_{2,1}\boxtimes\X^-_{1,1})\boxtimes\X^-_{1,1}=\X^-_{2,1}\boxtimes\X^-_{1,1}.
\end{align*}
Similar to the case of ${\rm top}(\X^-_{1,1}\boxtimes\X^-_{1,1})=\mathcal{X}^+_{1,1}$,
we can see that the top composition factors of $\X^-_{1,2}\boxtimes\X^-_{1,1}$ and $\X^-_{2,1}\boxtimes\X^-_{1,1}$ are $\X^+_{1,2}$ and $\X^+_{2,1}$, respectively. In particular, we have $\X^-_{1,2}\boxtimes\X^-_{1,1}\in C^{thick}_{1,2}$ and $\X^-_{2,1}\boxtimes\X^-_{1,1}\in C^{a}_{2,1}$, where $a=thick$ for $p_+\geq 3$ and $a=thin$ for $p_+=2$. But since $\X^-_{1,p_-}\in C^{thin}_{p_+-1,p_-}$ and $\X^-_{p_+,1}\in C^{thin}_{p_+,p_--1}$, we have a contradiction.
\end{proof}
\end{prop}

By Propositions \ref{Fusion11+}, \ref{++} and \ref{--}, we obtain the following proposition.
\begin{prop}
\label{12*sai}
For $1\leq r,r'\leq p_+$, $1\leq s,s'\leq p_-$, we have
\begin{align*}
\X^-_{r,s}\boxtimes \X^-_{r',s'}=(\K_{r,s}\boxtimes\K_{r',s'})\boxtimes\K^*_{1,1}.
\end{align*}
\end{prop}
\begin{remark}
From Propositions \ref{W^*W} and \ref{12*sai}, we can compute the tensor products between the simple modules $\X^-_{r,s}$ by using the tensor products between the indecomposable modules $\K_{r,s}$ (see the formulas (\ref{simple0820}) in Section \ref{FusionRing}).
\end{remark}

By Propositions \ref{fusion12-}, \ref{W^*W} and \ref{--}, we obtain the following proposition.
\begin{prop}
\label{change}
For $1\leq r,r'\leq p_+$, $1\leq s,s'\leq p_-$, we have
\begin{align*}
\mathcal{X}^-_{1,1}\boxtimes \mathcal{K}_{r,s}=\mathcal{X}^-_{1,1}\boxtimes\mathcal{X}^+_{r,s}=\mathcal{X}^-_{r,s}.
\end{align*}
\end{prop}

\begin{prop}
\label{changefinal}
For $1\leq r,r'\leq p_+$, $1\leq s,s'\leq p_-$, we have
\begin{align*}
\X^+_{r,s}\boxtimes \X^-_{r',s'}=(\K_{r,s}\boxtimes\K_{r',s'})\boxtimes\mathcal{X}^-_{1,1}.
\end{align*}
\begin{proof}
By Propositions \ref{Fusion11+} and \ref{change}, we have
\begin{align*}
(\K_{r,s}\boxtimes\K_{r',s'})\boxtimes\mathcal{X}^-_{1,1}&=(\K_{r,s}\boxtimes\K_{r',s'})\boxtimes(\mathcal{X}^+_{1,1}\boxtimes\mathcal{X}^-_{1,1})\\
&=(\mathcal{K}_{r,s}\boxtimes \mathcal{X}^+_{1,1})\boxtimes(\mathcal{K}_{r',s'}\boxtimes\mathcal{X}^-_{1,1})\\
&=\X^+_{r,s}\boxtimes \X^-_{r',s'}.
\end{align*}
\end{proof}
\end{prop}

\subsection{Rigidity for $\mathcal{Q}(\X^\pm_{r,s})_{\bullet,\bullet}$ and $\mathcal{P}^\pm_{r,s}$}
In this section, we will compute the tensor products $\K\boxtimes \mathcal{Q}(\X^\pm_{r,s})_{\bullet,\bullet}$ and $\K\boxtimes\mathcal{P}^\pm_{r,s}$, where $\K=\K_{1,2}$ and $\K_{2,1}$. Then we will show that all simple modules in the thin blocks and all indecomposable modules $\mathcal{Q}(\X^\pm_{r,s})_{\bullet,\bullet}$ and $\mathcal{P}^\pm_{r,s}$ are rigid and self-dual objects.
The self-dualities of these indecomposable modules were conjectured by \cite{GRW0,GRW}.

Similar to Lemma \ref{lem&p_+=2}, we can show that the tensor products $\K_{1,2}\boxtimes \X^\pm_{p_+,p_-}$ and $\K_{2,1}\boxtimes \X^\pm_{p_+,p_-}$ are non-zero. 
Since $\mathcal{\X}^\pm_{p_+,p_-}$ is projective, by the rigidity of $\K_{1,2}$ and $\K_{2,1}$, $\K_{1,2}\boxtimes\mathcal{\X}^\pm_{p_+,p_-}$ and $\K_{2,1}\boxtimes\mathcal{\X}^\pm_{p_+,p_-}$ become projective modules. 
Thus, by Propositions \ref{NGK2} and \ref{NGK22}, we obtain the following proposition.
\begin{prop}
\label{projXX}
\begin{align*}
&\K_{1,2}\boxtimes\X^\pm_{p_+,p_-}=\mathcal{Q}(\X^\pm_{p_+,p_--1})_{p_+,1},\\
&\K_{2,1}\boxtimes\X^\pm_{p_+,p_-}=\mathcal{Q}(\X^\pm_{p_+-1,p_-})_{1,p_-}.
\end{align*}
\end{prop}

\begin{prop}
\label{fusionQ^+}
\mbox{}
\begin{enumerate}
\item For $1\leq r\leq p_+-1$, we have
\begin{align*}
\K_{1,2}\boxtimes\X^\pm_{r,p_-}=\mathcal{Q}(\X^\pm_{r,p_--1})_{r,1}.
\end{align*}
\item For $1\leq s\leq p_--1$, we have
\begin{align*}
\K_{2,1}\boxtimes\X^\pm_{p_+,s}=\mathcal{Q}(\X^\pm_{p_+-1,s})_{1,s}.
\end{align*}
\end{enumerate}
\begin{proof}
We only prove $\K_{1,2}\boxtimes\X^+_{r,p_-}=\mathcal{Q}(\X^+_{r,p_--1})_{r,1}$. The other equalities can be proved in the same way.

By Proposition \ref{NGK2} and Lemma \ref{lem&p_+=2}, we see that
\begin{align}
\label{top1019}
{\rm top}(\K_{1,2}\boxtimes\X^+_{r,p_-})=\X^+_{r,p_--1}
\end{align}
and $\K_{1,2}\boxtimes\X^+_{r,p_-}$ has a quotient isomorphic to $\K^*_{r,p_--1}$. 

Let $r=1$. Assume that $\K_{1,2}\boxtimes\X^+_{1,p_-}$ has $\X^-_{p_+-1,p_--1}$ as a composition factor. Then the tensor product
\begin{align*}
\K_{2,1}\boxtimes(\K_{1,2}\boxtimes\X^+_{1,p_-})\simeq \K_{1,2}\boxtimes \X^+_{2,p_-}
\end{align*}
has $\X^-_{p_+,p_--1}\in C^{thin}_{p_+,1}$ as a composition factor. But, by (\ref{top1019}) and Proposition \ref{projXX}, we have a contradiction because $\K_{1,2}\boxtimes \X^+_{2,p_-}\in C^{thick}_{2,p_--1}$ for $p_+\geq 3$ and $\K_{1,2}\boxtimes \X^+_{2,p_-}\in C^{thin}_{2,p_--1}$ for $p_+=2$. By induction on $r$, we see that $\K_{1,2}\boxtimes\X^+_{r,p_-}$ does not have $\X^-_{r^\vee,p_--1}$ as a composition factor. 

Thus, from Propositions \ref{Ext}, \ref{ExtQ} and (\ref{top1019}), we see that
$
\K_{1,2}\boxtimes\X^+_{r,p_-}
$
is isomorphic to $\mathcal{Q}(\X^+_{r,p_--1})_{r,1}$ or a quotient of $\mathcal{Q}(\X^+_{r,p_--1})_{r,1}$.
By Lemma \ref{Etingof} and the self-duality of $\K_{1,2}$, we have
\begin{align*}
{\rm Hom}_{\mathcal{C}_{p_+,p_-}}(\X^+_{r,p_--1},\K_{1,2}\boxtimes\X^+_{r,p_-})\simeq\C.
\end{align*}
Therefore, noting that $\K_{1,2}\boxtimes\X^+_{r,p_-}$ has a quotient isomorphic to $\K^*_{r,p_--1}$, we obtain $\K_{1,2}\boxtimes\X^+_{r,p_-}\simeq \mathcal{Q}(\X^+_{r,p_--1})_{r,1}$.
\end{proof}
\end{prop}

\begin{prop}
\label{fusionQ^+12}
The tensor products of $\K_{1,2}\boxtimes\mathcal{Q}(\X^\pm_{r,s})_{\bullet,\bullet}$ and $\K_{2,1}\boxtimes\mathcal{Q}(\X^\pm_{r,s})_{\bullet,\bullet}$ are given by :
\begin{enumerate}
\item For $1\leq r\leq p_+-1$, $2\leq s\leq p_--1$, 
\begin{align*}
\K_{1,2}\boxtimes \mathcal{Q}(\X^\pm_{r,s})_{r^\vee,s}=\mathcal{Q}(\X^\pm_{r,s-1})_{r^\vee,s-1}\oplus \mathcal{Q}(\X^\pm_{r,s+1})_{r^\vee,s+1}.
\end{align*}
\item For $1\leq r\leq p_+$, $2\leq s\leq p_--2$, 
\begin{align*}
\K_{1,2}\boxtimes \mathcal{Q}(\X^\pm_{r,s})_{r,s^\vee}=\mathcal{Q}(\X^\pm_{r,s-1})_{r,s^\vee+1}\oplus \mathcal{Q}(\X^\pm_{r,s+1})_{r,s^\vee-1}.
\end{align*}
\item For $p_+\geq 3$, $2\leq r\leq p_+-1$, $1\leq s\leq p_--1$, 
\begin{align*}
\K_{2,1}\boxtimes \mathcal{Q}(\X^\pm_{r,s})_{r,s^\vee}=\mathcal{Q}(\X^\pm_{r-1,s})_{r-1,s^\vee}\oplus\mathcal{Q}(\X^\pm_{r+1,s})_{r+1,s^\vee}. 
\end{align*}
\item For $p_+\geq 3$, $2\leq r\leq p_+-2$, $1\leq s\leq p_-$, 
\begin{align*}
\K_{2,1}\boxtimes \mathcal{Q}(\X^\pm_{r,s})_{r^\vee,s}=\mathcal{Q}(\X^\pm_{r-1,s})_{r^\vee+1,s}\oplus\mathcal{Q}(\X^\pm_{r+1,s})_{r^\vee-1,s}. 
\end{align*}
\item For $1\leq r\leq p_+-1$, 
\begin{align*}
\K_{1,2}\boxtimes \mathcal{Q}(\X^\pm_{r,1})_{r^\vee,1}=\mathcal{Q}(\X^\pm_{r,2})_{r^\vee,2}.
\end{align*}
\item For $1\leq r\leq p_+$, 
\begin{align*}
\K_{1,2}\boxtimes \mathcal{Q}(\X^\pm_{r,p_--1})_{r,1}=2\X^\pm_{r,p_-}\oplus\mathcal{Q}(\X^\pm_{r,p_--2})_{r,2}.
\end{align*}
\item For $1\leq s\leq p_--1$,
\begin{align*}
\K_{2,1}\boxtimes \mathcal{Q}(\X^\pm_{1,s})_{1,s^\vee}=\mathcal{Q}(\X^\pm_{2,s})_{2,s^\vee}.
\end{align*}
\item For $1\leq s\leq p_-$,
\begin{align*}
\K_{2,1}\boxtimes \mathcal{Q}(\X^\pm_{p_+-1,s})_{1,s}=2\X^\pm_{p_+,s}\oplus\mathcal{Q}(\X^\pm_{p_+-2,s})_{2,s}.
\end{align*}
\item For $1\leq r\leq p_+$, $1\leq s\leq p_-$,
\begin{equation}
\label{relQ0621}
\begin{split}
&\K_{1,2}\boxtimes\mathcal{Q}(\X^\pm_{r,1})_{r,p_--1}=2\X^\mp_{r,p_-}\oplus\mathcal{Q}(\X^\pm_{r,2})_{r,p_--2},\\
&\K_{2,1}\boxtimes\mathcal{Q}(\X^\pm_{1,s})_{p_+-1,s}=2\X^\mp_{p_+,s}\oplus\mathcal{Q}(\X^\pm_{2,s})_{p_+-2,s}.
\end{split}
\end{equation}
\end{enumerate}
\begin{proof}
We will only prove 
\begin{align}
\K_{1,2}\boxtimes \mathcal{Q}(\X^+_{r,s})_{r^\vee,s}=\mathcal{Q}(\X^+_{r,s-1})_{r^\vee,s-1}\oplus \mathcal{Q}(\X^+_{r,s+1})_{r^\vee,s+1}
\label{202210240}
\end{align}
for $1\leq r\leq p_+-1$, $2\leq s\leq p_--1$. The other equalities can be proved in a similar way. 

By the rigidity of $\K_{1,2}$ the composition factors of $\K_{1,2}\boxtimes \mathcal{Q}(\X^+_{r,s})_{r^\vee,s}$ is given by
\begin{align*}
2\X^+_{r,s-1},\ \ L(h_{r,s-1}),\ \ 2\X^-_{r^\vee,s-1},\ \ 2\X^+_{r,s+1},\ \ L(h_{r,s+1}),\ \ 2\X^-_{r^\vee,s+1}.
\end{align*}
By using Lemma \ref{Etingof} and the self-duality of $\K_{1,2}$, we can see that for any simple $\mathcal{W}_{p_+,p_-}$-module $M$
\begin{equation}
\begin{split}
&{\rm Hom}_{\mathcal{C}_{p_+,p_-}}(\K_{1,2}\boxtimes \mathcal{Q}(\X^+_{r,s})_{r^\vee,s},M)
=
\begin{cases}
\mathbb{C}&M=\X^+_{r,s-1}\ {\rm or}\ \X^+_{r,s+1}\\
0&\ {\rm otherwise}
\end{cases}
\\
&{\rm Hom}_{\mathcal{C}_{p_+,p_-}}(M,\K_{1,2}\boxtimes \mathcal{Q}(\X^+_{r,s})_{r^\vee,s})
=
\begin{cases}
\mathbb{C}&M=\X^+_{r,s-1}\ {\rm or}\ \X^+_{r,s+1}\\
0&\ {\rm otherwise}.
\end{cases}
\end{split}
\label{0618x}
\end{equation}
From (\ref{0618x}), we see that $\K_{1,2}\boxtimes \mathcal{Q}(\X^+_{r,s})_{r^\vee,s}$ is direct sums of two indecomposable modules whose composition factors are the same as those of $\mathcal{Q}(\X^+_{r,s-1})_{r^\vee,s-1}$ and $\mathcal{Q}(\X^+_{r,s+1})_{r^\vee,s+1}$. Thus, by Propositions \ref{Ext} and \ref{ExtQ}, we obtain (\ref{202210240}).

\end{proof}
\end{prop}

We set
\begin{align*}
\mathbb{Q}_{p_+,p_-}=\bigsqcup_{i=1}^{p_+-1} \bigsqcup_{j=1}^{p_--1} Q^{thick}_{i,j}\sqcup \bigsqcup_{k=1}^{p_+-1} Q^{thin}_{k,p_-}\sqcup\bigsqcup_{l=1}^{p_--1} Q^{thin}_{p_+,l}.
\end{align*}
where $Q^{thick}_{i,j}$, $Q^{thin}_{k,p_-}$ and $Q^{thin}_{p_+,l}$ are defined by (\ref{Qthick}) and (\ref{Qthin}).
\begin{thm}
\label{QThm}
\begin{enumerate}
\item All simple modules in the thin blocks and the semi-simple blocks are rigid and self-dual.
\item For any $(\epsilon,a,b,c,d)\in \mathbb{Q}_{p_+,p_-}$, $\mathcal{Q}(\mathcal{X}^\epsilon_{a,b})_{c,d}$ is rigid and self-dual.
\end{enumerate}
\begin{proof}
Noting the formula (\ref{relQ0621}), from Propositions \ref{projXX}, \ref{fusionQ^+} and \ref{fusionQ^+12}, we can see that each of
the simple modules in the thin and semi-simple blocks and the indecomposable modules $\mathcal{Q}(\mathcal{X}^\epsilon_{a,b})_{c,d} ((\epsilon,a,b,c,d)\in \mathbb{Q}_{p_+,p_-})$ is obtained as a direct summand of the repeated tensor products of $\K_{1,2}$ and $\K_{2,1}$. 
Thus these simple and indecomposable modules are rigid and self-dual.
\end{proof}
\end{thm}

\begin{prop}
\label{QX+}
For any $(\epsilon,a,b,c,d)\in \mathbb{Q}_{p_+,p_-}$, we have
\begin{align*}
\X^+_{1,1}\boxtimes\mathcal{Q}(\X^\epsilon_{a,b})_{c,d}=\mathcal{Q}(\X^\epsilon_{a,b})_{c,d}.
\end{align*}
\begin{proof}
Since $\X^+_{1,1}\boxtimes \X^\pm_{a,b}=\X^\pm_{a,b}$ for any simple module $\X^\pm_{a,b}$ in the thin blocks and the semi-simple blocks, by the associativity of the tensor product $\boxtimes$ and by Propositions \ref{projXX}, \ref{fusionQ^+} and \ref{fusionQ^+12}, we obtain the equality of the claim.
\end{proof}
\end{prop}

By Propositions \ref{change}, \ref{projXX}, \ref{fusionQ^+} and \ref{fusionQ^+12}, we obtain the following proposition.
\begin{prop}
\label{QX-}
For any $(\epsilon,a,b,c,d)\in \mathbb{Q}_{p_+,p_-}$, we have
\begin{align*}
\X^-_{1,1}\boxtimes\mathcal{Q}(\X^\epsilon_{a,b})_{c,d}=\mathcal{Q}(\X^{-\epsilon}_{a,b})_{c,d}.
\end{align*}
\begin{proof}
From Propositions \ref{change} and \ref{projXX}, we have
\begin{equation}
\label{chQ}
\begin{split}
&\X^-_{1,1}\boxtimes\mathcal{Q}(\X^\pm_{p_+,p_--1})_{p_+,1}=\mathcal{Q}(\X^\mp_{p_+,p_--1})_{p_+,1},\\
&\X^-_{1,1}\boxtimes\mathcal{Q}(\X^\pm_{p_+-1,p_-})_{1,p_-}=\mathcal{Q}(\X^\mp_{p_+-1,p_-})_{1,p_-}.
\end{split}
\end{equation}
From Propositions \ref{fusionQ^+}, \ref{fusionQ^+12} and (\ref{chQ}), we can prove the claim inductively.
\end{proof}
\end{prop}

\begin{prop}
\label{fusionP^+}
\mbox{}
\begin{enumerate}
\item For $1\leq r\leq p_+-1$, we have
\begin{align*}
\K_{1,2}\boxtimes\mathcal{Q}(\X^\pm_{r,p_-})_{r^\vee,p_-}=\PP^\pm_{r,p_--1}.
\end{align*}
\item For $1\leq s\leq p_--1$, we have
\begin{align*}
\K_{2,1}\boxtimes\mathcal{Q}(\X^\pm_{p_+,s})_{p_+,s^\vee}=\PP^\pm_{p_+-1,s}.
\end{align*}
\end{enumerate}
\begin{proof}
We only prove the first case. The second case can be proved in the same way. 

By using Lemma \ref{Etingof} and the self-duality of $\K_{1,2}$, we can see that for any simple $\mathcal{W}_{p_+,p_-}$-module $M$
\begin{equation}
\begin{split}
&{\rm Hom}_{\mathcal{C}_{p_+,p_-}}(\K_{1,2}\boxtimes\mathcal{Q}(\X^\pm_{r,p_-})_{r^\vee,p_-},M)
=
\begin{cases}
\mathbb{C}&M=\X^\pm_{r,p_--1}\\
0&\ {\rm otherwise}
\end{cases}
\\
&{\rm Hom}_{\mathcal{C}_{p_+,p_-}}(M,\K_{1,2}\boxtimes\mathcal{Q}(\X^\pm_{r,p_-})_{r^\vee,p_-})
=
\begin{cases}
\mathbb{C}&M=\X^\pm_{r,p_--1}\\
0&\ {\rm otherwise}
\end{cases}
\end{split}
\label{0618xx}
\end{equation}
Since $\mathcal{Q}(\X^\pm_{r,p_-})_{r^\vee,p_-}$ is projective, by the rigidity of $\mathcal{K}_{1,2}$, $\K_{1,2}\boxtimes\mathcal{Q}(\X^\pm_{r,p_-})_{r^\vee,p_-}$ is also projective. Therefore from (\ref{0618xx}), we obtain
\begin{align*}
\K_{1,2}\boxtimes\mathcal{Q}(\X^\pm_{r,p_-})_{r^\vee,p_-}\simeq\PP^\pm_{r,p_--1}.
\end{align*}
\end{proof}
\end{prop}
The following proposition can be proved in the same way as Proposition \ref{fusionP^+}, so we omit the proof.
\begin{prop}
\label{fusionP^+12}
The fusion products of $\K_{1,2}\boxtimes\PP^\pm_{r,s}$ and $\K_{2,1}\boxtimes\PP^\pm_{r,s}$ are given by :
\begin{enumerate}
\item For $1\leq r\leq p_+-1$, $2\leq s\leq p_--2$, 
\begin{align*}
\K_{1,2}\boxtimes\PP^\pm_{r,s}=\PP^\pm_{r,s-1}\oplus \PP^\pm_{r,s+1}.
\end{align*}
\item For $p_+\geq 3$, $2\leq r\leq p_+-2$, $1\leq s\leq p_--1$,
\begin{align*}
\K_{2,1}\boxtimes\PP^\pm_{r,s}=\PP^\pm_{r-1,s}\oplus\PP^\pm_{r+1,s}.
\end{align*}
\item For $1\leq r\leq p_+-1$, $s=1$,
\begin{align*}
\K_{1,2}\boxtimes\PP^\pm_{r,p_--1}=2\mathcal{Q}(\X^\pm_{r,p_-})_{r^\vee,p_-}\oplus \PP^\pm_{r,p_--2}.
\end{align*}
\item For $r=1$, $1\leq s\leq p_--1$,
\begin{align*}
\K_{2,1}\boxtimes\PP^\pm_{p_+-1,s}=2\mathcal{Q}(\X^\pm_{p_+,s})_{p_+,s^\vee}\oplus\PP^\pm_{p_+-2,s}.
\end{align*}
\item For $1\leq r\leq p_+-1$, $1\leq s\leq p_--1$,
\begin{align*}
&\K_{1,2}\boxtimes\PP^\pm_{r,1}=2\mathcal{Q}(\X^\mp_{r,p_-})_{p_+-r,p_-}\oplus\PP^\pm_{r,2},\\
&\K_{2,1}\boxtimes\PP^\pm_{1,s}=2\mathcal{Q}(\X^\mp_{p_+,s})_{p_+,p_--s}\oplus\PP^\pm_{2,s}.
\end{align*}
\end{enumerate} 
\end{prop}

Similar to the argument of Theorem \ref{QThm}, from Propositions \ref{fusionP^+} and \ref{fusionP^+12}, we obtain the following theorem.
\begin{thm}
All indecomposable modules $\mathcal{P}^\pm_{r,s}$ are rigid and self-dual.
\end{thm}

Similar to the arguments in the proofs of Propositions \ref{QX+} and \ref{QX-}, from Propositions \ref{change}, \ref{fusionP^+} and \ref{fusionP^+12}, we obtain the following propositions.
\begin{prop}
For $1\leq r\leq p_+-1$, $1\leq s\leq p_--1$, we have
\begin{align*}
\X^+_{1,1}\boxtimes\mathcal{P}^\pm_{r,s}=\mathcal{P}^\pm_{r,s}.
\end{align*}
\end{prop}
\begin{prop}
\label{X^+11P}
For $1\leq r\leq p_+-1$, $1\leq s\leq p_--1$, we have
\begin{align*}
\X^-_{1,1}\boxtimes\mathcal{P}^\pm_{r,s}=\mathcal{P}^\mp_{r,s}.
\end{align*}
\end{prop}


\section{Fusion rings}
\label{FusionRing}

In this section, following \cite[Subsection 5.3]{TWFusion}, we introduce certain fusion rings, and by using the results in Section \ref{FusionRules}, determine their ring structure.
\subsection{The ring structure of $P(\mathbb{I}_{p_+,p_-})$}
\label{secP}
We set
\begin{align*}
&\mathbb{I}_{p_+,p_-}\\
&=\bigl\{\mathcal{K}_{r,s},\mathcal{K}^*_{r,s},\mathcal{X}^-_{r,s},\mathcal{X}^\epsilon_{r,p_-},\mathcal{X}^\epsilon_{p_+,s},\mathcal{X}^{\epsilon}_{p_+,p_-},\mathcal{Q}(\X^\epsilon_{r,s})_{r^\vee,s},\mathcal{Q}(\X^\epsilon_{r,s})_{r,s^\vee},\\
&\ \ \ \ \ \mathcal{Q}(\X^{\epsilon}_{r,p_-})_{r^\vee,p_-},\mathcal{Q}(\X^{\epsilon}_{p_+,s})_{p_+,s^\vee},\PP^\epsilon_{r,s}\ |\ 1\leq r<p_+,1\leq s<p_-,\epsilon=\pm\bigr\}.
\end{align*}
Note that
\begin{align*}
|\mathbb{I}_{p_+,p_-}|=9p_+p_--5p_+-5p_-+3.
\end{align*}
Following \cite[Subsection 5.3]{TWFusion}, we introduce the following free abelian group $P(\mathbb{I}_{p_+,p_-})$ of rank $9p_+p_--5p_+-5p_-+3$ 
\begin{align*}
P(\mathbb{I}_{p_+,p_-})=&\bigoplus_{M\in \mathbb{I}_{p_+,p_-}}\Z[M]_P\\
=&\bigoplus_{r=1}^{p_+-1}\bigoplus_{s=1}^{p_--1}\Z[\mathcal{K}_{r,s}]_P\oplus\bigoplus_{r=1}^{p_+-1}\bigoplus_{s=1}^{p_--1}\Z[\mathcal{K}^*_{r,s}]_P\oplus\bigoplus_{r=1}^{p_+-1}\bigoplus_{s=1}^{p_--1}\Z[\mathcal{X}^-_{r,s}]_P\\
&\oplus\bigoplus_{r=1}^{p_+}\bigoplus_{\epsilon=\pm}\Z[\mathcal{X}^\epsilon_{r,p_-}]_P\oplus \bigoplus_{s=1}^{p_--1}\bigoplus_{\epsilon=\pm}\Z[\mathcal{X}^\epsilon_{p_+,s}]_P\\
&\oplus\bigoplus_{r=1}^{p_+-1}\bigoplus_{s=1}^{p_--1}\bigoplus_{\epsilon=\pm}\Z[\mathcal{Q}(\X^\epsilon_{r,s})_{r^\vee,s}]_P\oplus\bigoplus_{r=1}^{p_+-1}\bigoplus_{s=1}^{p_--1}\bigoplus_{\epsilon=\pm}\Z[\mathcal{Q}(\X^\epsilon_{r,s})_{r,s^\vee}]_P\\
&\oplus \bigoplus_{r=1}^{p_+-1}\bigoplus_{\epsilon=\pm}\Z[\mathcal{Q}(\X^{\epsilon}_{r,p_-})_{r^\vee,p_-}]_P\oplus\bigoplus_{s=1}^{p_--1}\bigoplus_{\epsilon=\pm}\Z[\mathcal{Q}(\X^{\epsilon}_{p_+,s})_{p_+,s^\vee}]_P\\
&\oplus\bigoplus_{r=1}^{p_+-1}\bigoplus_{s=1}^{p_--1}\bigoplus_{\epsilon=\pm}\Z[\PP^\epsilon_{r,s}]_P.
\end{align*}

From the results of Section \ref{FusionRules} we that $P(\mathbb{I}_{p_+,p_-})$ has the structure of a commutative ring. Precisely the definition is as follows.
\begin{dfn}
We define the structure of a commutative ring on $P(\mathbb{I}_{p_+,p_-})$ as follows:
the product as a ring is given by
\begin{align*}
[M_1]_P\cdot[M_2]_P=[M_1\boxtimes M_2]_P,
\end{align*}
where $M_1,M_2\in \mathbb{I}_{p_+,p_-}$ and we extend the symbol $[\bullet]_P$ as follows
\begin{align*}
\Bigl[\bigoplus^n_{i\geq 1}N_i\Bigr]_P=\bigoplus^n_{i\geq 1}[N_i]_P
\end{align*}
for any $N_i\in \mathbb{I}_{p_+,p_-}$ and any $n\in \Z_{\geq 1}$. 
\end{dfn}
From the results in Section \ref{FusionRules}, we see that the three operators
\begin{align*}
X=\mathcal{K}_{1,2}\boxtimes -,\ \ \ \ \ \ \ \ Y=\mathcal{K}_{2,1}\boxtimes -,\ \ \ \ \ \ \ \ Z=\X^-_{1,1}\boxtimes -
\end{align*}
define $\Z$-linear endomorphism of $P(\mathbb{I}_{p_+,p_-})$. Thus $P(\mathbb{I}_{p_+,p_-})$ is a module over $\Z[X,Y,Z]$.
We define the following $\Z[X,Y,Z]$-module map
\begin{align*}
\psi:\Z[X,Y,Z]&\rightarrow P(\mathbb{I}_{p_+,p_-}),\\
f(X,Y,Z)&\mapsto f(X,Y,Z)\cdot[\mathcal{K}_{1,1}]_P.
\end{align*}
Before examining the action of $\Z[X,Y,Z]$ on $P(\mathbb{I}_{p_+,p_-})$, we introduce the following Chebyshev polynomials.
\begin{dfn}
We define the {\rm Chebyshev polynomials} $U_n(A)$, $n=0,1,\dots\in \Z[A]$ recursively 
\begin{align*}
&U_0(A)=1,\ \ \ \ \ \ \ U_1(A)=A,\\
&U_{n+1}(A)=AU_n(A)-U_{n-1}(A).
\end{align*}
\end{dfn}
\begin{remark}
The coefficient of the leading term of any Chebyshev polynomial $U_n(A)$ is 1. Thus we have
\begin{align*}
\mathbb{Z}[A]=\bigoplus_{n=0}^\infty\mathbb{Z}U_n(A).
\end{align*}
\end{remark}
\vspace{2mm}

The goal of this section is to prove the following theorem.
\begin{thm}
\label{ideal}
The $\Z[X,Y,Z]$-module map $\psi$ is surjective and the kernel of $\psi$ is given by the following ideal
\begin{align*}
{\rm ker}\psi=
&\langle (Z^2-1)U_{p_--1}(X),(Z^2-1)U_{p_+-1}(Y),\\
&\ \ U_{2p_--1}(X)-2ZU_{p_--1}(X),U_{2p_+-1}(Y)-2ZU_{p_+-1}(Y) \rangle.
\end{align*}
\end{thm}

We will show some propositions to prove this theorem. 

From Propositions \ref{Fusion12Wkai0}, \ref{Fusion12Wkai}, \ref{W^*W} and \ref{change}, we obtain the following lemma.
\begin{lem}
\label{Q^+pqX}
\mbox{}
\begin{enumerate}
\item For $1\leq r\leq p_+$, $1\leq s\leq p_-$, we have
\begin{align*}
&[\mathcal{K}_{r,s}]_P=U_{s-1}(X)U_{r-1}(Y)[\mathcal{K}_{1,1}]_P,
&[\X^-_{r,s}]_P=ZU_{s-1}(X)U_{r-1}(Y)[\mathcal{K}_{1,1}]_P,
\end{align*}
where we use the abbreviations $\mathcal{K}_{r,p_-}=\mathcal{X}^+_{r,p_-}$ and $\mathcal{K}_{p_+,s}=\mathcal{X}^+_{p_+,s}$.
\item For $1\leq r\leq p_+-1$, $1\leq s\leq p_--1$, we have
\begin{align*}
[\mathcal{K}^*_{r,s}]_P=Z^2U_{s-1}(X)U_{r-1}(Y)[\mathcal{K}_{1,1}]_P.
\end{align*}
\end{enumerate}
\end{lem}
By Propositions \ref{projXX}, \ref{QX-} and Lemma \ref{Q^+pqX}, we have the following lemma.
\begin{lem}
\label{Q^+pq}
\mbox{}
\begin{enumerate}
\item For $1\leq r\leq p_+-1$, $1\leq s\leq p_--1$, we have
\begin{align*}
&[Q(\X^+_{p_+,p_--1})_{p_+,1}]_P=(U_{p_-}(X)+U_{p_--2}(X))U_{p_+-1}(Y)[\mathcal{K}_{1,1}]_P,\\
&[Q(\X^+_{p_+-1,p_-})_{1,p_-}]_P=U_{p_--1}(X)(U_{p_+}(Y)+U_{p_+-2}(Y))[\mathcal{K}_{1,1}]_P,\\
&[Q(\X^-_{p_+,p_--1})_{p_+,1}]_P=Z(U_{p_-}(X)+U_{p_--2}(X))U_{p_+-1}(Y)[\mathcal{K}_{1,1}]_P,\\
&[Q(\X^-_{p_+-1,p_-})_{1,p_-}]_P=ZU_{p_--1}(X)(U_{p_+}(Y)+U_{p_+-2}(Y))[\mathcal{K}_{1,1}]_P.
\end{align*}
\item We have the following relations
\begin{align}
\label{rel0705}
(Z^2-1)U_{p_--1}(X)[\mathcal{K}_{1,1}]_P=(Z^2-1)U_{p_+-1}(Y)[\mathcal{K}_{1,1}]_P=0.
\end{align}
\end{enumerate}
\end{lem}
By Propositions \ref{fusionQ^+}, \ref{fusionQ^+12}, \ref{QX-} and Lemma \ref{Q^+pqX}, we obtain the following lemma.
\begin{lem}
\label{Q^pqQ}
\mbox{}
\begin{enumerate}
\item For $1\leq r\leq p_+$, $1\leq s\leq p_--1$, we have
\begin{align*}
&[Q(\X^+_{r,s})_{r,s^\vee}]_P=\bigl(U_{2p_--s-1}(X)+U_{s-1}(X)\bigr)U_{r-1}(Y)[\mathcal{K}_{1,1}]_P,\\
&[Q(\X^-_{r,s})_{r,s^\vee}]_P=Z\bigl(U_{2p_--s-1}(X)+U_{s-1}(X)\bigr)U_{r-1}(Y)[\mathcal{K}_{1,1}]_P.
\end{align*}
\item For $1\leq r\leq p_+-1$, $1\leq s\leq p_-$, we have
\begin{align*}
&[Q(\X^+_{r,s})_{r^\vee,s}]_P=\bigl(U_{2p_+-r-1}(Y)+U_{r-1}(Y)\bigr)U_{s-1}(X)[\mathcal{K}_{1,1}]_P,\\
&[Q(\X^-_{r,s})_{r^\vee,s}]_P=Z\bigl(U_{2p_+-r-1}(Y)+U_{r-1}(Y)\bigr)U_{s-1}(X)[\mathcal{K}_{1,1}]_P.
\end{align*}
\item We have the following relations
\begin{equation}
\begin{aligned}
&\bigl(U_{2p_--1}(X)-2ZU_{p_--1}(X)\bigr)[\mathcal{K}_{1,1}]_P=0,\\
&\bigl(U_{2p_+-1}(Y)-2ZU_{p_+-1}(Y)\bigr)[\mathcal{K}_{1,1}]_P=0.
\end{aligned}
\label{relationori}
\end{equation}
\end{enumerate}
\end{lem}
By Propositions \ref{fusionP^+}, \ref{fusionP^+12}, \ref{X^+11P} and Lemma \ref{Q^pqQ}, we obtain the following lemma.
\begin{lem}
\label{Q^pqP}
For $1\leq r\leq p_+-1$, $1\leq s\leq p_--1$, we have
\begin{align*}
&[\PP^+_{r,s}]_P=\bigl(U_{2p_+-r-1}(Y)+U_{r-1}(Y)\bigr)\bigl(U_{2p_--s-1}(X)+U_{s-1}(X)\bigr)[\mathcal{K}_{1,1}]_P,\\
&[\PP^-_{r,s}]_P=Z\bigl(U_{2p_+-r-1}(Y)+U_{r-1}(Y)\bigr)\bigl(U_{2p_--s-1}(X)+U_{s-1}(X)\bigr)[\mathcal{K}_{1,1}]_P.\\
\end{align*}
\end{lem}

\begin{proof}[Proof of Theorem \ref{ideal}]
By Lemmas \ref{Q^+pqX}, \ref{Q^+pq}, \ref{Q^pqQ} and \ref{Q^pqP}, we can see that $\psi$ is surjective. We define the following ideal of $\Z[X,Y,Z]$
\begin{align*}
I=
&\langle (Z^2-1)U_{p_--1}(X),(Z^2-1)U_{p_+-1}(Y),\\
&\ \ U_{2p_--1}(X)-2ZU_{p_--1}(X),U_{2p_+-1}(Y)-2ZU_{p_+-1}(Y) \rangle.
\end{align*}
Then, from the relations (\ref{rel0705}) and (\ref{relationori}), we see that $I$ is contained in ${\rm ker}\psi$. It is easy to see that the dimension of the quotient ring ${\Z[X,Y,Z]}/{I}$
is $9p_+p_--5p_+-5p_-+3$. Therefore we have ${\rm ker}\psi=I$.
\end{proof}

From the structure of $P(\mathbb{I}_{p_+,p_-})$, we can rederive the non-semisimple fusion rules given by \cite{GRW0,GRW,R,W}. For example, by using Propositions \ref{12*}, \ref{W^*W} and \ref{12*sai} and the multiplication formula for Chebyshev polynomials
\begin{align*}
U_k(x)U_l(x)=\sum_{i=|k-l|;2}^{k+l}U_i(x),
\end{align*}
we obtain
\begin{equation}
\label{simple0820}
\begin{split}
\mathcal{X}^\pm_{r,s}\boxtimes\mathcal{X}^\pm_{r',s'}
&=\bigoplus_{i=|r-r|+1;2}^{p_+-|p_+-r-r'|-1}\bigoplus_{j=|s-s'|+1;2}^{p_--|p_--s-s'|-1}\mathcal{K}^*_{i,j}\\
&\oplus\bigoplus_{i=m(r,r';p_+);2}^{r+r'-p_+-1}\bigoplus_{j=|s-s'|+1;2}^{p_--|p_--s-s'|-1}\mathcal{Q}(\mathcal{X}^+_{p_+-i,j})_{i,j}\\
&\oplus\bigoplus_{i=|r-r'|+1;2}^{p_+-|p_+-r-r'|-1}\bigoplus_{j=m(s,s';p_-);2}^{s+s'-p_--1}\mathcal{Q}(\mathcal{X}^+_{i,p_--j})_{i,j}\\
&\oplus \bigoplus_{i=m(r,r';p_+);2}^{r+r'-p_+-1}\bigoplus_{j=m(s,s';p_-);2}^{s+s'-p_--1}\mathcal{P}^+_{p_+-i,p_--j},
\end{split}
\end{equation}
where $m(a,b;c)=a+b-c-1\ {\rm mod}\ 2$ and we use the following notation (cf. (C.1) in \cite{W})
\begin{align*}
\mathcal{Q}(\mathcal{X}^+_{p_+,j})_{0,j}&=\mathcal{X}^+_{p_+,j},&\mathcal{Q}(\mathcal{X}^+_{i,p_-})_{i,0}&=\mathcal{X}^+_{i,p_-},\\
\mathcal{P}^+_{p_+,p_--j}&=\mathcal{Q}(\mathcal{X}^+_{p_+,p_--j})_{p_+,j},&\mathcal{P}^+_{p_+-i,p_-}&=\mathcal{Q}(\mathcal{X}^+_{p_+-i,p_-})_{i,p_-},\\
\mathcal{P}^+_{p_+,p_-}&=\mathcal{X}^+_{p_+,p_-}.
\end{align*}

\begin{remark}
Let us consider the quotient ring
\begin{align*}
P'(\mathbb{I}_{p_+,p_-})=P(\mathbb{I}_{p_+,p_-})/\langle[\mathcal{K}_{1,1}]_P-[\mathcal{K}^*_{1,1}]_P\rangle,
\end{align*}
where $\langle[\mathcal{K}_{1,1}]_P-[\mathcal{K}^*_{1,1}]_P\rangle$ is the ideal generated by $[\mathcal{K}_{1,1}]_P-[\mathcal{K}^*_{1,1}]_P$. Let $\pi$ be the surjection from $P(\mathbb{I}_{p_+,p_-})$ to $P'(\mathbb{I}_{p_+,p_-})$. Note that $\pi([\mathcal{K}_{r,s}]_P)=\pi([\mathcal{K}^*_{r,s}]_P)$ for all $r,s\geq 1$.
From Theorem \ref{ideal}, as the $\Z[X,Y,Z]$-module, $P'(\mathbb{I}_{p_+,p_-})$ is isomorphic to $\mathbb{Z}[X,Y,Z]/I'$, where 
\begin{align*}
&I'=\langle Z^2-1,U_{2p_--1}(X)-2ZU_{p_--1}(X),U_{2p_+-1}(Y)-2ZU_{p_+-1}(Y)\rangle.
\end{align*}
The quotient ring $P'(\mathbb{I}_{p_+,p_-})$ seems to correspond to the non-semisimple fusion ring of the Serre quotient $\mathcal{C}_{p_+,p_-}/\mathcal{S}$, where $\mathcal{S}$ is the Serre subcategory consisting of all minimal simple modules $L(h_{r,s})\ ((r,s)\in\mathcal{T})$. 
In this Serre quotient $\mathcal{C}_{p_+,p_-}/\mathcal{S}$, we have $L(h_{r,s})=0\ ((r,s)\in\mathcal{T})$ and
$
\mathcal{K}_{r,s}\simeq \mathcal{K}^*_{r,s}\simeq \mathcal{X}^+_{r,s}
$
for $r,s\geq 1$. From Propositions \ref{minimal fusion} and \ref{null11rs}, we can conjecture that $\mathcal{C}_{p_+,p_-}/\mathcal{S}$ has the structure of a braided tensor category.

The Serre quotient $\mathcal{C}_{p_+,p_-}/\mathcal{S}$ has similarities with the representation category of the quantum group $\mathfrak{g}_{p_+,p_-}$ \cite{Arike,FF2}.
For example, by setting $\pi([\mathcal{K}_{1,2}]_P)=0$ or $\pi([\mathcal{K}_{2,1}]_P)=0$ in $P'(\mathbb{I}_{p_+,p_-})$, we obtain two fusion rings $P(\mathcal{W}_{p_+})$ and $P(\mathcal{W}_{p_-})$ given by \cite{TWFusion}:
\begin{align*}
&P'(\mathbb{I}_{p_+,p_-})|_{\pi([\mathcal{K}_{1,2}]_P)=0}\simeq P(\mathcal{W}_{p_+}),
&P'(\mathbb{I}_{p_+,p_-})|_{\pi([\mathcal{K}_{2,1}]_P)=0}\simeq P(\mathcal{W}_{p_-}).
\end{align*}
These isomorphisms seem to correspond to two embeddings $\overline{U}_{q_+}(sl_2)\hookrightarrow \mathfrak{g}_{p_+,p_-}$ and $\overline{U}_{q_-}(sl_2)\hookrightarrow \mathfrak{g}_{p_+,p_-}$ \cite{FF2}, where $\overline{U}_{q_\pm}(sl_2)$ is the restricted quantum group at $q_\pm={\rm exp}(\frac{2\pi i p_\mp}{p_\pm})$.
\end{remark}
\subsection{The ring structure of $K(\mathbb{S}_{p_+,p_-})$}
\label{secK}
In this subsection, we introduce a certain Grothendieck fusion ring $K(\mathbb{S}_{p_+,p_-})$ and review the structure of this ring in our setting. The structure of the Grothendieck ring $K(\mathbb{S}_{p_+,p_-})$ is described in \cite{FF2} using quantum groups and determined in \cite{RW} using the Verlinde ring of the singlet $W$-algebra.


Let $\widetilde{K}(\mathcal{C}_{p_+,p_-})$ be the Grothendieck group of $\mathcal{C}_{p_+,p_-}$. The rank of $\widetilde{K}(\mathcal{C}_{p_+,p_-})$ is $2p_+p_-+\frac{(p_+-1)(p_--1)}{2}$ and generated by all simple module:
\begin{align*}
\widetilde{K}(\mathcal{C}_{p_+,p_-})=\bigoplus_{r=1}^{p_+}\bigoplus_{s=1}^{p_-}\bigoplus_{\epsilon=\pm}\Z[\X^\epsilon_{r,s}]_{\widetilde{K}}\oplus \bigoplus_{(r,s)\in \mathcal{T}}\Z[L(h_{r,s})]_{\widetilde{K}}.
\end{align*}
We set
\begin{align*}
&\mathbb{S}_{p_+,p_-}=\bigl\{\mathcal{X}^\epsilon_{r,s}\ |\ 1\leq r\leq p_+,1\leq s\leq p_-,\epsilon=\pm\bigr\}.
\end{align*}
Let $K(\mathbb{S}_{p_+,p_-})$ be the quotient
\begin{align*}
K(\mathbb{S}_{p_+,p_-}):=\widetilde{K}(\mathcal{C}_{p_+,p_-})\big{/}\bigl(\bigoplus_{(r,s)\in \mathcal{T}}\Z[L(h_{r,s})]_{\widetilde{K}}\bigr),
\end{align*}
and let $[\X^\epsilon_{r,s}]_K\in K(\mathbb{S}_{p_+,p_-})$ be the image of $[\X^\epsilon_{r,s}]_{\widetilde{K}}$. Then
\begin{align*}
K(\mathbb{S}_{p_+,p_-})=\bigoplus_{M\in \mathbb{S}_{p_+,p_-}}\mathbb{Z}[M]_K=\bigoplus_{r=1}^{p_+}\bigoplus_{s=1}^{p_-}\bigoplus_{\epsilon=\pm}\Z[\X^\epsilon_{r,s}]_K.
\end{align*}
We can define the structure of a commutative ring on $K(\mathbb{S}_{p_+,p_-})$ such that the product as a ring is given by
\begin{align*}
[M_1]_K\cdot[M_2]_K=[M_1\boxtimes M_2]_K,
\end{align*}
where $M_1,M_2\in \mathbb{S}_{p_+,p_-}$.
\begin{remark}
In \cite{GRW0,GRW,W}, from the context of the boundary conformal field theory, it is shown that a certain subgroup of $\widetilde{K}(\mathcal{C}_{p_+,p_-})$ has the structure of a ring. The structure of this subgroup is more complex than that of $K(\mathbb{S}_{p_+,p_-})$. It is an interesting problem to investigate whether or not this subgroup has an explicit form such as $K(\mathbb{S}_{p_+,p_-})$.
\end{remark}
From the results in Section \ref{FusionRules}, we see that the three operators
\begin{align*}
X=\X^+_{1,2}\boxtimes -,\ \ \ \ \ \ \ \ Y=\X^+_{2,1}\boxtimes -,\ \ \ \ \ \ \ \ Z=\X^-_{1,1}\boxtimes -
\end{align*}
define $\Z$-linear endomorphism of $K(\mathbb{S}_{p_+,p_-})$. Thus $K(\mathbb{S}_{p_+,p_-})$ is a module over $\Z[X,Y,Z]$.
We define the following $\Z[X,Y,Z]$-module map
\begin{align*}
\phi:\Z[X,Y,Z]&\rightarrow K(\mathbb{S}_{p_+,p_-}),\\
f(X,Y,Z)&\mapsto f(X,Y,Z)\cdot[\X^+_{1,1}]_K.
\end{align*}
\begin{thm}[\cite{RW}]
\label{ideal2}
The $\Z[X,Y,Z]$-module map $\phi$ is surjective and the kernel of $\phi$ is given by the following ideal
\begin{align*}
{\rm ker}\phi=\langle Z^2-1,U_{p_-}(X)-U_{p_--2}(X)-2Z,U_{p_+}(Y)-U_{p_+-2}(Y)-2Z \rangle.
\end{align*}
\begin{proof}
Similar to the case of $P(\mathbb{I}_{p_+,p_-})$, from the fusion rules given in Section \ref{FusionRules}, we can see that $\phi$ is surjective and
\begin{align}
\label{fg}
&[\X^+_{r,s}]_K=U_{s-1}(X)U_{r-1}(Y)[\X^+_{1,1}]_K,
&[\X^-_{r,s}]_K=ZU_{s-1}(X)U_{r-1}(Y)[\X^+_{1,1}]_K
\end{align}
for $r,s\geq 1$.
From Proposition \ref{Q^+pq} and (\ref{fg}), we have 
\begin{equation}
\label{relationori2}
\begin{aligned}
&\bigl(U_{p_-}(X)-U_{p_--2}(X)-2Z\bigr)[\X^+_{1,1}]_K=0,\\
&\bigl(U_{p_+}(Y)-U_{p_+-2}(Y)-2Z\bigr)[\X^+_{1,1}]_K=0.
\end{aligned}
\end{equation}
We define the following ideal of $\Z[X,Y,Z]$
\begin{align*}
J=\langle Z^2-1,U_{p_-}(X)-U_{p_--2}(X)-2Z,U_{p_+}(Y)-U_{p_+-2}(Y)-2Z \rangle.
\end{align*}
Then, by the relations (\ref{relationori2}) and by Proposition \ref{change}, we see that $J$ is contained in ${\rm ker}\phi$. It is easy to see that the dimension of the quotient ring ${\Z[X,Y,Z]}/{J}$
is $2p_+p_-$. Therefore we have ${\rm ker}\phi=J$.
\end{proof}
\end{thm}

\section*{Acknowledgement}
We are grateful Simon Wood for discussing about logarithmic conformal field theory and vertex tensor category in detail.
We would also like to thank Simon Lentner for stimulating disscussion.


\vspace{10mm}
\ \ \ \ H.~Nakano, \textsc{Advanced Mathematical Institute, Osaka Metropolitan University, Osaka 558-8585, Japan}\par\nopagebreak
  \textit{E-mail address} : \texttt{hiromutakati@gmail.com}


\begin{thebibliography}{999}
\bibitem{AM}
     D. Adamovi\'{c} and A. Milas,
     ``On the triplet vertex algebra $W(p)$'',
     \textit{Advances in Mathematics} $\bold{217}$ (2008), 2664-2699.
\bibitem{AMW2p}
     D. Adamovi\'{c} and A. Milas,
     ``On $\mathcal{W}$-algebras associated to $(2,p)$ minimal models for certain vertex algebras'',
     \textit{International Mathematics Research Notices} 2010 (2010) 20 : 3896-3934, arXiv:0908.4053.
\bibitem{AMW3p}
     D. Adamovi\'{c} and A. Milas,
     ``On W-algebra extensions of (2, p) minimal models: p $>$3'',
     \textit{Journal of Algebra} \textbf{344} (2011) 313-332. arXiv:1101.0803.
\bibitem{AK}
     D. Adamovi\'{c} and A. Milas,
     ``An explicit realization of logarithmic modules for the vertex operator algebra $W_{p_+,p_-}$'',
     \textit{J. Math. Phys}. \textbf{53}, (2012), 16pp.
\bibitem{AL}
R. Allen, S. Lentner, C. Schweigert and S. Wood, 
``Duality structures for module categories of vertex operator algebras and the Feigin Fuchs boson'',
 arXiv:2107.05718 (2021).     
\bibitem{AW}     
R. Allen and S. Wood, ``Bosonic ghostbusting: the bosonic ghost vertex algebra admits a logarithmic module category with rigid fusion'', \textit{Communications in Mathematical Physics}, 390(2), 959-1015 (2022).          
\bibitem{Arike}
       Y. Arike,
       ``A matrix realization of the quantum group $\mathfrak{g}_{p,q}$'', 
        \textit{International Journal of Mathematics}, 22.03 (2011): 345-398.
\bibitem{BPZ}        
A. A. Belavin, A. M. Polyakov, and A. B. Zamolodchikov,
 ``Infinite conformal symmetry in two-dimensional quantum field theory'',
  \textit{Nuclear Physics B}, 241(2), 333-380 (1984).
\bibitem{BNW}
     B. Boe, D. Nakano, E. Wiesner,
     \textit{Category $\mathcal{O}$ for the Virasoro algebra: cohomology and Koszulity}.
     Pacific J. Math. \textbf{234} (2008), no. 1, 1-21.
\bibitem{CMY}
     T. Creutzig, R. McRae. and J. Yang, 
     ``On ribbon categories for singlet vertex algebras'', 
     \textit{Communications in Mathematical Physics}, 387(2), 865-925, arXiv:2007.12735.
\bibitem{D} 
\textit{NIST Digital Library of Mathematical Functions}, Release 1.0.27 of 2020-
06-15, F. Olver, A. Olde Daalhuis, D. Lozier, B. Schneider, R. Boisvert, C. Clark, B. Miller, B.
Saunders, H. Cohl and M. McClain, eds, http://dlmf.nist.gov/.
\bibitem{Etingof}
    P. Etingof, G. Shlomo, D. Nikshych, and V. Ostrik. 
     \textit{Tensor Categories}. 
      Number volume 205 in Mathematical Surveys and Monographs. American
Mathematical Society, 2015.
\bibitem{FF}
     B. Feigin and D.B. Fuchs.
     ``Representations of the Virasoro algebra'',
     \textit{in Representations of infinite-dimensional Lie groups and Lie algebras, Gordon andd Breach, New York}(1989).
\bibitem{FF2}
     B. L. Feigin, A.M. Gainutdinov, A.M. Semikhatov, and I. Yu Tipunin,
     ``Logarithmic extensions of minimal models: characters and modular transformation'',
     \textit{Nuclear Phys. B} 757(2006),303-343.
\bibitem{FF22}
     B.L. Feigin, A.M. Gainutdinov, A.M. Semikhatov, and I. Yu Tipunin, 
     ``Kazhdan-Lusztig-dual quantum group for logarithmic extensions of Virasoro minimal models'',
     \textit{J. Math. Phys}. 48:032303, 2007.
\bibitem{FF3}
     B. L. Feigin, A.M. Ga{i}nutdinov, A.M. Semikhatov, and I. Yu Tipunin, 
     ``Modular group representations and fusion in logarithmic conformal field theories and in the quantum group center'',
     \textit{Comm. Math. Phys}. 265 (2006), 47–93.
\bibitem{FF4}
     B. L. Feigin, A.M. Ga{i}nutdinov, A.M. Semikhatov, and I. Yu Tipunin, 
     ``Kazhdan-Lusztig correspondence for the representation category of the triplet W-algebra in logarithmic CFT'',
     \textit{Theor. Math. Phys}. 148 (2006) 1210-1235; \textit{Teor. Mat. Fiz}. \textbf{148} (2006) 398-427.
\bibitem{Felder}
     G. Felder,
     ``BRST approach to minimal models'',
     \textit{Nuc. Phy. B} \textbf{317} (1989) 215-236.     
\bibitem{GK}
     M. R. Gaberdiel and H.G. Kausch, 
     ``Indecomposable fusion products'', 
     \textit{Nucl. Phys}. B \textbf{477} (1996) 293 [hep-th/9604026].
\bibitem{GRW0}
     M. Gaberdiel, I. Runkel, and S. Wood, 
     ``Fusion rules and boundary conditions in the $c=0$ triplet model'',
     \textit{J.Phys}. \textbf{A42} (2009) 325403, arXiv:0905.0916 [hep-th].
\bibitem{GRW}
     M. Gaberdiel, I. Runkel, and S. Wood, 
     ``A modular invariant bulk theory for the $c=0$ triplet model'',
     \textit{J.Phys. A:math. Theor}. 44 (2011) 015204, arXiv:1008.0082v1.
\bibitem{H}
     Y. Z. Huang,
     ``Cofiniteness conditions, projective covers and the logarithmic tensor product theory'',
     \textit{J. Pure Appl. Algebra}, 213(4):458-475, 2009.
\bibitem{HLZ1}
     Y. Z. Huang, J. Lepowsky, and L. Zhang, 
     ``Logarithmic tensor category theory for
generalized modules for a conformal vertex algebra, I: Introduction and strongly
graded algebras and their generalized modules'', 
\textit{Conformal Field Theories and
Tensor Categories}, 169-248, \textit{Math. Lect. Peking Univ., Springer, Heidelberg},
2014.   
\bibitem{HLZ2}
     Y. Z. Huang, J. Lepowsky, and L. Zhang,
     ``Logarithmic tensor category theory for generalized modules for a conformal vertex algebra, I\hspace{-1pt}I: Logarithmic formal      calculus and properties of logarithmic intertwining operators'', arXiv:1012.4196 (2010). 
\bibitem{HLZ3}
     Y. Z. Huang, J. Lepowsky, and L. Zhang, 
     ``Logarithmic tensor category theory
for generalized modules for a conformal vertex algebra, III: Intertwining maps
and tensor product bifunctors'', arXiv:1012.4197 (2010).
\bibitem{HLZ4}
     Y. Z. Huang, J. Lepowsky, and L. Zhang, 
     ``Logarithmic tensor category theory for
generalized modules for a conformal vertex algebra, IV: Constructions of tensor
product bifunctors and the compatibility conditions'', 
arXiv:1012.4198 (2010).
\bibitem{HLZ5}
     Y. Z. Huang, J. Lepowsky, and L. Zhang, 
``Logarithmic tensor category theory
for generalized modules for a conformal vertex algebra, V: Convergence condition for intertwining maps and the corresponding compatibility condition'',
arXiv:1012.4199 (2010).
\bibitem{HLZ6}
    Y. Z. Huang, J. Lepowsky, and L. Zhang, 
     ``Logarithmic tensor category theory
for generalized modules for a conformal vertex algebra, VI: Expansion condition, associativity of logarithmic intertwining operators, and the associativity isomorphisms'', arXiv:1012.4202 (2010).
\bibitem{HLZ7}
     Y. Z. Huang, J. Lepowsky, and L. Zhang, 
     ``Logarithmic tensor category theory for generalized modules for a conformal vertex algebra, VII: Convergence
and extension properties and applications to expansion for intertwining maps'',
arXiv:1110.1929 (2011).
\bibitem{HLZ8}
     Y. Z. Huang, J. Lepowsky, and L. Zhang,
      ``Logarithmic tensor category theory
for generalized modules for a conformal vertex algebra, VIII: Braided tensor
category structure on categories of generalized modules for a conformal vertex
algebra'', arXiv:1110.1931 (2011).
\bibitem{IK}
     K. Iohara, Y. Koga.
     \textit{Representation Theory of the Virasoro Algebra}.
     Springer Monographs in Mathematics, Berlin, Springer 2011.
\bibitem{JS}
     A. Joyal and R. Street,
     ``Braided tensor categories'',
     \textit{Adv. Math}. \textbf{102} (1993) 20-78. 
\bibitem{Kanade}
     S. Kanade and D. Ridout, 
     ``NGK and HLZ: Fusion for physicists and mathematicians'', 
      \textit{In D Adamovic and P Papi, editors, Affine, Vertex and
W-algebras}, volume 37 of \textit{Springer INdAM}, pages 135-181, \textit{Cham}, 2019. \textit{Springer}. arXiv:1812.10713 [math-ph].
\bibitem{Ka}
     H. G. Kausch,
     ``Extended conformal algebras generated by multiplet of primary fields'',
     \textit{Phys. Lett. B}, \textbf{259} (1991), 448-455.
\bibitem{KL}
     D. Kazhdan and G. Lusztig, ``Tensor structures arising from affine Lie algebras I\hspace{-1.2pt}V'', \textit{J. Amer. Math. Soc.} \textbf{7} (1994) 383-453.    
\bibitem{MS}      
R. McRae and V. Sopin, 
``Fusion and (non)-rigidity of Virasoro Kac modules in logarithmic minimal models at $(p, q) $-central charge'', arXiv:2302.08907 (2023).        
\bibitem{McRae}
      R. McRae and J. Yang,
      ``Structure of Virasoro tensor categories at central charge $13-6p-6p^{-1}$ for integers $p>1$'',
      arXiv:2011.02170 (2020).
\bibitem{Milas}
       A. Milas, 
      ``Fusion rings for degenerate minimal models'', 
       \textit{J. Algebra} \textbf{254} (2002), no. 2, 300-335.
\bibitem{NT}
      K. Nagatomo and A. Tsuchiya, 
      ``The Triplet Vertex Operator Algebra $W(p)$ and Restricted Quantum Group at Root of Unity'', 
      \textit{Adv. Stdu. in Pure Math., Exploring new Structures and Natural Constructions in Mathematical Physics, Amer. Math. Soc}. \textbf{61} (2011)
1–49, arXiv:0902.4607. 
\bibitem{Nakano}
H. Nakano, ``Projective covers of the simple modules for the triplet $W$-algebra $\mathcal {W}_{p_+, p_-} $." arXiv:2305.12448 (2023).
\bibitem{Lin}
      L. Xianzu, ``Fusion rules of Virasoro vertex operator algebras'', 
      \textit{Proceedings of the American Mathematical Society} 143.9 (2015):3765-3776.
\bibitem{R}
      J. Rasmussen, 
      ``W-extended logarithmic minimal models'', 
      \textit{Nucl. Phys. B} \textbf{807} (2009) 495 [0805.2991 [hep-th]].
\bibitem{RW}
D. Ridout and S. Wood,
``Modular transformations and Verlinde formulae for logarithmic $(p_+, p_-)$-models'', 
  \textit{Nuclear Physics B} 880 (2014): 175-202.        
\bibitem{TK}
      A. Tsuchiya and Y. Kanie, 
     ``Fock space representations of the Virasoro algebra - Intertwining operators'',
     \textit{Publ. RIMS, Kyoto Univ}. 22(1986) 259-327. 
\bibitem{TWFusion}
      A. Tsuchiya and S. Wood,
      ``The tensor structure on the representation category of the $\mathcal{W}_{p}$ triplet algebra'',
      \textit{J. Phys. A} \textbf{46} (2013), no. 44, 445203, 40 pp.
\bibitem{TW}
		A. Tsuchiya and S. Wood,
		``On the extended W-algebra of type $sl_2$ at positive rational level''
		\textit{International Mathematics Research Notices}, Volume 2015, Issue 14, 1 January 2015, Pages 5357-5435.
\bibitem{W}
        S. Wood, ``Fusion Rules of the $\mathcal{W}_{p,q}$ Triplet Models'',
        \textit{J. Phys. A} \textbf{43} (2010) 045212.
\end{thebibliography}
\end{document}